\documentclass[11pt,a4paper]{article}
\usepackage{marvosym}
\usepackage{amsmath}
\usepackage{scalerel}
\usepackage{amsthm}
\usepackage{amsfonts}
\usepackage{amssymb}
\usepackage{comment}
\usepackage{graphicx}
\usepackage{float}
\usepackage{slashed} 
\usepackage{centernot}
\usepackage{braket}
\usepackage{hyperref}
\usepackage[nameinlink]{cleveref}
\usepackage{tikz-cd}
\usetikzlibrary{math} 
\usepackage{adjustbox}
\usepackage{txfonts}
\usepackage{euscript}
\usepackage{thmtools}
\usepackage{thm-restate}
\usepackage{relsize}
\usepackage{lmodern}
\usepackage[inline]{enumitem}
\usepackage[left=2cm,right=2cm,top=3cm,bottom=3cm]{geometry}

\usepackage{lineno}

\hypersetup{
  colorlinks   = true, 
  urlcolor     = {blue!50!black}, 
  linkcolor    = {blue!50!black}, 
  citecolor   = {red!50!black} 
}

\def\restriction#1#2{\mathchoice
              {\setbox1\hbox{${\displaystyle #1}_{\scriptstyle #2}$}
              \restrictionaux{#1}{#2}}
              {\setbox1\hbox{${\textstyle #1}_{\scriptstyle #2}$}
              \restrictionaux{#1}{#2}}
              {\setbox1\hbox{${\scriptstyle #1}_{\scriptscriptstyle #2}$}
              \restrictionaux{#1}{#2}}
              {\setbox1\hbox{${\scriptscriptstyle #1}_{\scriptscriptstyle #2}$}
              \restrictionaux{#1}{#2}}}
\def\restrictionaux#1#2{{#1\,\smash{\vrule height .8\ht1 depth .85\dp1}}_{\,#2}}

\newcommand{\limplies}{\rightarrow}

\DeclareMathOperator{\Aut}{Aut}
\DeclareMathOperator{\aut}{Aut}

\DeclareMathOperator{\sym}{Sym}
\DeclareMathOperator{\bdn}{bdn}

\DeclareMathOperator{\tp}{\mathsf{tp}}
\DeclareMathOperator{\qftp}{\mathsf{qftp}}

\DeclareMathOperator{\val}{val}

\DeclareMathOperator{\Th}{\mathsf{Th}}
\DeclareMathOperator{\Sig}{Sig}

\DeclareMathOperator{\arity}{arity}

\DeclareMathOperator{\HC}{\mathsf{HC}}

\let\strokeL\L
\newcommand{\Los}{\strokeL o\'s }

\newcommand{\Ifrak}{\ensuremath{\mathfrak{I}}}

\newcommand{\Nfrak}{\ensuremath{\mathfrak{N}}}

\newcommand{\Ical}{\ensuremath{\mathcal{I}}}

\newcommand{\Lcal}{\ensuremath{\mathcal{L}}}

\newcommand{\Ncal}{\ensuremath{\mathcal{N}}}

\newcommand{\Mbb}{\ensuremath{\mathbb{M}}}
\newcommand{\Nbb}{\ensuremath{\mathbb{N}}}

\let\str\CMcal
\newcommand\cla\EuScript
\let\lang\CMcal

\newcommand{\lL}{\lang{L}}

\newcommand{\sM}{\str{M}}
\newcommand{\sN}{\str{N}}

\newcommand{\sS}{\str{S}}

\newcommand{\sI}{\str{I}}
\newcommand{\sJ}{\str{J}}
\newcommand{\sH}{\str{H}}
\newcommand{\sG}{\str{G}}
\newcommand{\sR}{\str{R}}

\let\b\mathbb
\newcommand{\bN}{\b{N}}

\newcommand{\cC}{\cla{C}}

\newcommand{\cK}{\cla{K}}

\newcommand{\NOT}{\mathrm{N}}
\newcommand{\codingcla}[1]{\cC_{#1}}
\newcommand{\ncodingcla}[1]{\NOT\codingcla{#1}}

\let\monster\mathbb
\newcommand{\mM}{\monster{M}}

\newcommand{\dprk}{\mathrm{dp-rk}}

\let\fullprod\boxtimes
\let\suppos\ast
\let\disjointunion\sqcup

\newcommand{\cR}{\cla{R}}

\let\reduces\multimapdotbothA

\let\reducesboth\multimapdotboth
\newcommand{\reducesqf}{\overset{\text{qf}}{\reduces}} 
\newcommand{\reducesbothqf}{\overset{\text{qf}}{\reducesboth}}

\makeatletter
\DeclareRobustCommand{\reducesneq}{\mathrel{\mathpalette\reducesneq@{\centernot\multimapdotboth}}}
\newcommand{\reducesneq@}[2]{%
  \vtop{\offinterlineskip
    \ialign{\hfil##\hfil\cr
      $\m@th#1\multimapdotbothA$\cr 
      \noalign{\sbox\z@{$\m@th#1\mkern2mu$}\kern-\wd\z@}
      $\m@th\alexey@demote{#1}#2$\cr
    }%
  }%
}
\newcommand{\alexey@demote}[1]{%
  \ifx#1\displaystyle\scriptstyle\else
  \ifx#1\textstyle\scriptstyle\else
  \scriptscriptstyle\fi\fi
}
\makeatother

\definecolor{lime}{HTML}{A6CE39}
\DeclareRobustCommand{\orcidicon}{%
	\begin{tikzpicture}
	\draw[lime, fill=lime] (0,0) 
	circle [radius=0.16] 
	node[white] {{\fontfamily{qag}\selectfont \tiny ID}};
	\draw[white, fill=white] (-0.0625,0.095) 
	circle [radius=0.007];
	\end{tikzpicture}
	\hspace{-2mm}
}

\foreach \x in {A, ..., Z}{%
	\expandafter\xdef\csname orcid\x\endcsname{\noexpand\href{https://orcid.org/\csname orcidauthor\x\endcsname}{\noexpand\orcidicon}}
}

\definecolor{lime}{HTML}{A6CE39}
\DeclareRobustCommand{\orcidicon}{%
	\begin{tikzpicture}
	\draw[lime, fill=lime] (0,0) 
	circle [radius=0.16] 
	node[white] {{\fontfamily{qag}\selectfont \tiny ID}};
	\draw[white, fill=white] (-0.0625,0.095) 
	circle [radius=0.007];
	\end{tikzpicture}
	\hspace{-2mm}
}

\usepackage{ulem}
\usepackage{todonotes}
\setuptodonotes{noinline}

\colorlet{PierreColor}{cyan}
\colorlet{NadavColor}{-green!40!yellow}
\colorlet{ArisColor}{red}

\DeclareMathOperator{\age}{\mathsf{Age}}

\author{Nadav Meir\footnote{Supported by Narodowe Centrum Nauki, Poland, grant 2016/22/E/ST1/00450, and by Israel Science Foundation grant number 665/20 and 555/21.} \thanks{The first author is supported by Narodowe Centrum Nauki, Poland, grant 2016/22/E/ST1/00450, Israel Science Foundation grant number 555/21 and Israel Science Foundation grant number 665/20.
} \orcidA{}, Aris Papadopoulos\footnote{Supported by a Leeds Doctoral Scholarship, from the University of Leeds.} \orcidB{}, and Pierre Touchard \footnote{Supported by the University of Campania ‘Luigi Vanvitelli’ in the framework of V:ALERE 2019 (GoAL project), and by KU Leuven IF
C14/17/083 and C16/23/010.} \orcidC{}}

\title{Generalised Indiscernibles, Dividing Lines, and Products of Structures\footnotetext{2020 \textit{Mathematics Subject Classification}: Primary: 03C45
; Secondary: 05C55 }}

\theoremstyle{plain}
\newtheorem{theorem}{Theorem}[section]
\newtheorem{lemma}[theorem]{Lemma}
\newtheorem{proposition}[theorem]{Proposition}
\newtheorem{fact}[theorem]{Fact}
\newtheorem{corollary}[theorem]{Corollary}

\newtheorem{thmx}{Theorem}

\theoremstyle{definition}
\newtheorem{definition}[theorem]{Definition}

\newtheorem{conjecture}[theorem]{Conjecture}

\newtheorem*{notation*}{Notation}
\newtheorem{question}[theorem]{Question}
\newtheorem*{question*}{Question}
\newtheorem{example}[theorem]{Example}

\theoremstyle{remark}

\newtheorem{remark}[theorem]{Remark}
\newtheorem*{remark*}{Remark}

\newtheorem*{observation}{Observation}
\newtheorem{claim}{Claim}
\newtheorem{poc}{Proof of Claim}

\newenvironment{claimproof}{%
	\begin{poc}
	}{%
	\hfill$\blacktriangleleft$
	\end{poc}
	}

\definecolor{black}{rgb}{0,0,0}

\colorlet{savedColor}{.}

\crefname{subsection}{subsection}{subsections}
\crefname{subsubsection}{subsubsection}{subsubsections}

\let\vphi\varphi

\begin{document}

\maketitle


\begin{abstract}
    Generalised indiscernibles highlight a strong link between model theory and structural Ramsey theory. In this paper, we use generalised indiscernibles as tools to prove results in both these areas. More precisely, we first show that a reduct of an ultrahomogenous $\aleph_0$-categorical structure which has higher arity than the original structure cannot be Ramsey. In particular, the only nontrivial Ramsey reduct of the generically ordered random $k$-hypergraph is the linear order. We then turn our attention to model-theoretic dividing lines that are characterised by collapsing generalised indiscernibles, and prove, for these dividing lines, several transfer principles in (\emph{full} and \emph{lexicographic}) products of structures. As an application, we construct new algorithmically tame classes of graphs.
\end{abstract}

\section*{Introduction}

A major theme in modern model theory, originating, essentially, from the work of Shelah, is to try and discover properties that divide the class of all first-order theories into two sides, a \emph{`tame side'} and a \emph{`wild side'} \cite{She90}. One wants to be able to deduce structure results for the theories that lie on the tame side and non-structure results for the theories that lie on the wild side. These abstract properties are most commonly referred to as \emph{dividing lines}. Perhaps the most well-known dividing line, which appeared early on in Shelah's work, in the 1970s, is \emph{stability}, and thus the programme of studying the properties and structure of stable theories is called \emph{stability theory}. 

Many beautiful theorems have been proved for stable theories, but since their introduction in the 1970s, there has been a general trend of trying to generalise results for stable theories to wider (and wilder) classes of theories, such as \emph{($n$-)dependent} (also referred to as \emph{$n$-NIP}) \emph{theories}, \emph{distal theories}, \emph{NSOP$_n$ theories}, and much more, leading to what is sometimes referred to as \emph{neostability theory} or \emph{classification theory}.

In a certain light, the various definitions of dividing lines appearing in neostability theory may seem rather ad hoc, and hence a theme from a sort of ``meta-classification theoretic'' point of view is to systematically study dividing lines. Thus rather than explicitly studying the tame and wild theories, according to some dividing line, one may choose to make the object of study the dividing lines themselves.

One of the starting points of this paper is to examine the behaviour of various dividing lines through the study of transfer principles in some natural product constructions. We will discuss this in more detail in the remainder of this introduction. We will mainly study notions of model-theoretic tameness which admit definitions based on (generalised) indiscernible sequences: ($n$-)NIP, monadically NIP, and ($n$-)distal theories, as well as a hierarchy of dividing lines that we introduce now. 

\paragraph{Generalised Dividing Lines} 
Various methods have been proposed as a uniform way of generalising and extending existing and well-studied dividing lines, see for instance \cite{GHS17} and \cite{GM22}. In this paper, we will focus on the approach taken in \cite{GHS17} (and developed further in \cite{GH19}), which provides a uniform scheme of dividing lines arising from \emph{coding} classes of finite structures: given a class of finite structures, $\cla{K}$, we will consider the class $\ncodingcla{\cla{K}}$ of all first-order theories that do not code all the members of $\cla{K}$ in a uniform manner (see \cref{sec:K-configurations} for precise definitions), and $\codingcla{\cla{K}}$ the complement of $\ncodingcla{\cla{K}}$, in the class of all first-order theories.

The study of dividing lines of the form $\ncodingcla{\cla{K}}$ where $\cla{K}$ is a Ramsey class (see \Cref{def:hp-jep-ap-erp}), inevitably leads to the notion of $\mathsf{FLim}(\cla{K})$-indexed indiscernibles (in the sense of \Cref{def:generalised-indiscernibles}), where $\mathsf{FLim}(\cla{K})$ denotes the Fraïssé limit of $\cla{K}$. To explain this connection, recall that \emph{order-indiscernible sequences} -- sequences of elements indexed by linear orders with ``constant'' behaviour on increasing subsequences -- are widely used in model theory, as a way to extract ``essential'' behaviours of other sequences. Using Ramsey's theorem and compactness, given \emph{any} sequence $(a_i)_{i\in \mathbb{Q}}$, one can find an ordered-indiscernible sequence realising its so-called \emph{Ehrenfeucht–Mostowski type} (or \emph{ EM-type}). This is however not simply a property specific to linear orders. In a very precise sense, the true reason this is possible is that the class of all finite linear orders satisfies a structural analogue of Ramsey's theorem (see \Cref{def:hp-jep-ap-erp}).

One can then observe a deep connection between generalised indiscernibles, structural Ramsey theory, and the study of dividing lines arising from coding Ramsey structures.  Intuitively, the existence of `uncollapsed' $\mathsf{FLim}(\cla{K})$-indiscernible sequences (see \Cref{def:indiscernible collapsing qftp}) for a Ramsey class $\cla{K}$ in some model of a first-order theory $T$ indicates that $T$ can somehow `see a trace' of $\cla{K}$ and therefore $T$ cannot be in $\ncodingcla{\cla{K}}$, and vice versa (see \Cref{thm:Ramsey class collapse}, which generalises {\cite[Theorem 3.14]{GH19}}). One of our main tools for analysing dividing lines will be various notions of \emph{indiscernible collapse}. 

The starting point of understanding dividing lines in this way is Shelah's well-known theorem that in a stable theory every indiscernible sequence is totally indiscernible, and, in fact, this property characterises the class of stable theories (see \cite[Theorem~II.2.13]{She90}). Results of this nature have been shown for various other dividing lines, and essentially amount to instances of \Cref{thm:Ramsey class collapse} with the appropriate class $\cla{K}$ for each dividing line. 

\paragraph{Transfer principles}
A \emph{transfer principle} for a given theory $T$ and a (model-theoretic) property $P$, is a statement of the following form: a model $\sM$ of $T$  has property $P$ if, and only if, some simpler structures related to $\sM$ have property $P$. The \emph{Ax-Kochen-Ershov theorem} is one of the most well-known and most celebrated transfer principles. It states that Henselian valued fields of equicharacteristic zero are model complete relative to their residue fields and value groups. There is extensive literature in model theory that focuses on such transfer principles, for they are often an important step toward characterising models with a given property. At the same time, the transfer principles give us information on how ``well-behaved'' a theory is. For instance, the theorem of Delon, which states that equicharacteristic 0 Henselian valued fields transfer the NIP (the absence of the \emph{independence property}) from the residue field to the valued field itself \cite{Del81}, suggests that this class of valued fields is well-behaved for the model theorist, and analogous transfers for this class are known (or expected to be true) with respect to many other dividing lines.

The study of theories that are well-behaved with respect to a given dividing line is an active topic of research. In this work, we take another approach and use transfer principles as a tool that reveals whether a dividing line itself is a well-behaved one. A reasonable approach could be to think that a dividing line is ``well-behaved'' if some natural transfer principles hold. More precisely, we are interested in the following general question:

\begin{question*}
    Let $P$ be some notion of tameness (e.g. some dividing line). Is it true that two structures $\sM$ and $\sN$ are tame with respect to $P$ if, and only if, their \emph{full product}, $\sM \boxtimes \sN$ (see \Cref{def:fullproduct}), and their \emph{lexicographic product}, $\sM[\sN]$ (see \Cref{DefinitionLexicographicProduct2}), are tame with respect to $P$?
\end{question*}

It appears that for some Ramsey classes $\cla{K}$ the class of theories $\ncodingcla{\cla{K}}$ that do not code $\cK$ does not admit such transfers. The important distinction is that this is the case, for Ramsey classes of structures collapsing generalised indiscernibles to a \emph{fixed} reduct, and this is what we will discuss in more detail now.

\paragraph{Main Results} 
As we have already mentioned, generalised indiscernibles, coding configurations, and Ramsey classes are closely interconnected notions. This connection was first established by Guingona and Hill for structures in finite relational language {\cite[Theorem 3.14]{GH19}}. Our first theorem gives a slightly more general version: 

\begin{thmx}[\Cref{thm:Ramsey class collapse}]\label{thm:Ramsey class collapse-intro}
    Let $\sI$ be an $\aleph_0$-categorical Fraïssé limit of a Ramsey class. Then the following are equivalent for a theory $T$:
    \begin{enumerate}
        \item $T\in \ncodingcla{\sI}$.
        \item $T$ collapses $\sI$-indiscernibles.
    \end{enumerate}
\end{thmx}

We then use generalised indiscernibles as a criterion to show that certain classes cannot be Ramsey. More precisely, we obtain the following result:

\begin{thmx}[\Cref{reduct of higher arity is not Ramsey}]\label{reduct of higher arity is not Ramsey-intro}
    Let $\sJ$ be an ultrahomogeneous $n$-ary $\aleph_0$-categorical structure in a relational language $\lL$, and let $\sI$ be a non-$n$-ary reduct of $\sJ$ in a finite relational language $\lL'$. Then $\age(\sI)$ is not a Ramsey class.
\end{thmx}

Next, we provide a negative answer to a question asked in \cite[Section 7]{GP23} and \cite[Question 4.7]{GPS21} about the linearity of the hierarchy given by coding $\cK$-configurations. More precisely, in \cite{GPS21} the authors observe that one has the following strict inclusions:
\[
\codingcla{\cla{E}} \supset \codingcla{\cla{LO}} \supset \codingcla{\cla{OG}} = \codingcla{\cla{OH}_2}\supset \codingcla{\cla{OH}_3} \supset \cdots \supset \codingcla{\cla{OH}_n} \supset \cdots
\]
where $\cla{E}$ is the class of all finite sets (in the language of pure equality), $\cla{LO}$ is the class of all finite linear orders, $\cla{OG}$ the class of all finite ordered graphs and $\cla{OH}_n$ the class of all finite ordered $n$-hypergraphs (see \Cref{sub:standard-structs} for more details). The question then becomes the following:

\begin{question*}
    Are there any other classes $\codingcla{\cK}$? If so, do these classes remain linearly ordered under inclusion?
\end{question*}

We answer this question negatively, as follows:

\begin{thmx}[\Cref{example-full-product-bad}]
    Building on the notation above, let $\cla{OC}$ be the class of finite convexly ordered binary branching $C$-relations (see \Cref{sub:standard-structs}), and let  $\cla{OG}\boxtimes \cla{OC}$ be the full product the classes $\cla{OG}$ and $\cla{OC}$ (see \Cref{def:fullproduct}). Then:
\begin{center}
    
\begin{tikzpicture}
    \draw (-1.5,0) node{$\codingcla{\cla{E}}$};
    \draw (-0.5,0) node{\scalebox{1.5}{$\supset$}};
    \draw (0.4,0) node{$\codingcla{\cla{LO}}$};
    \draw (0.9,0.9) node
        {\scalebox{1.7}{$\mathrel{\rotatebox[origin=c]{45}{$\supset$}}$}};
    \draw (1.7,1.4) node{$\codingcla{\cla{OG}}$};
    \draw (0.9,-0.9) node
    {\scalebox{1.7}{$\mathrel{\rotatebox[origin=c]{-45}{$\supset$}}$}};
    \draw (1.7,-1.4) node{$\codingcla{\cla{OC}}$};
    \draw (2.5,0.9) node
    {\scalebox{1.7}{$\mathrel{\rotatebox[origin=c]{-45}{$\supset$}}$}};
    \draw (2.5,-0.9) node
    {\scalebox{1.7}{$\mathrel{\rotatebox[origin=c]{45}{$\supset$}}$}};
    \draw (3.5,0) node{$\codingcla{\cla{OG}\boxtimes \cla{OC}}$};
    \draw (4.5,0) node{\scalebox{1.5}{$\supset$}};
    \draw (5.4,0) node{$\cdots$};
    \draw (1.7,0.3) node{\scalebox{1.7}{$\mathrel{\rotatebox[origin=c]{90}{$\nsubseteq$}}$}};
    \draw (1.7,-0.3) node{\scalebox{1.7}{$\mathrel{\rotatebox[origin=c]{90}{$\nsupseteq$}}$}};

\end{tikzpicture}
\end{center}

\end{thmx}

We then turn our attention to transfer principles. One of our main tools is a characterisation of generalised indiscernibles in full products (\Cref{PropositionCharacterisationIndiscernibleFullProduct}) and lexicographic sums (\Cref{prop:TransferTrivialIndiscernibilityLexicographicSum}). We immediately obtain transfer theorems for dividing lines characterised by generalised indiscernibility. These results are summarised below:

\begin{thmx}[\Cref{cor:nip-n-transfer-full}, \Cref{cor:tranfer-n-distal-full}, \Cref{cor:transfer-ind-triv-full}]\label{thm:full-transfer-intro}
Let $\sM$, $\sN$ be structure in respective languages $\lL_\sM$ and $\lL_\sN$. For 
\[
    P\in\{\text{NIP$_n$, $n$-distal, indiscernible-triviality}\}
\] 
the following are equivalent:
\begin{enumerate}
    \item $\sM$ and $\sN$ have $P$.
    \item The full product of $\sM$ and $\sN$, $\sM\boxtimes\sN$, has $P$.
\end{enumerate}
\end{thmx}

\begin{thmx}[\Cref{cor:nip-n-transfer-lex}, \Cref{prop:TransferTrivialIndiscernibilityLexicographicSum}, \Cref{cor:TransferPrincipleMonadicNIP}, \Cref{prop:m-distality-lex-transfer}]\label{thm:lex-transfer-intro}
Let $\sM$ be an $\lL_\sM$-structure and $\mathfrak{N}=\{\sN_a\}_{a\in \sM}$ be a collection $\mathfrak{N}$ of $\lL_\mathfrak{N}$-structures indexed by $\sM$. For: 
\[
    P\in\{\text{NIP$_n$},\text{ $n$-distal}, \text{ indiscernible-triviality}, \text{ monadic NIP}\},
\]
the following are equivalent:
\begin{enumerate}
    \item $\sM$ and the common theory of $\{\sN_a:a\in \sM\}$ have $P$.
    \item The lexicographic sum of $\sM$ and $\mathfrak{N}$, $\sM[\mathfrak{N}]$, has $P$.
\end{enumerate}
\end{thmx}

Notice that in particular, monadic NIP transfers in lexicographic sums. We apply this result to generalise one of the results of \cite{PS2023}, to lexicographic sums of ordered graphs of bounded twin-width. More precisely, once we obtain a result characterising ultraproducts of classes of lexicographic products, which is interesting in its own right (\Cref{prop:ultraproducts}), we prove the following:

\begin{thmx}[\Cref{cor:sum-has-bdd-twin-width}]\label{thm:sum-has-bdd-twin-width-intro}
    Let $\cC_1$ and $\cC_2$ be two hereditary classes of finite graphs with bounded twin-width. Then, the class of graphs consisting of lexicographic sums of graphs from $\cC_1$ and $\cC_2$ has bounded twin-width.
\end{thmx}

\paragraph{Structure of the paper} \Cref{sec:prelims,sec:Generalised Indiscernibles and Coding Configurations} contain all the relevant terminology and background definitions needed for the remainder. In \Cref{sec:Generalised Indiscernibles and Coding Configurations}, we focus on coding configurations and collapsing indiscernibles and prove some preliminary results.  Then, in \Cref{sec: Collapsing indiscernible as a witness for the Ramsey property} we prove \Cref{thm:Ramsey class collapse-intro} and \Cref{reduct of higher arity is not Ramsey-intro}. We then devote \Cref{sec:transfer} to proving \Cref{thm:full-transfer-intro} and \Cref{thm:lex-transfer-intro}. In \Cref{sec:twin-width} we use the monadic NIP case in \Cref{thm:lex-transfer-intro} to prove \Cref{thm:sum-has-bdd-twin-width-intro}. We conclude the paper with some open questions, in \Cref{sec:questions}.

\paragraph{Acknowledgements} The research presented in this paper started during the P\&S Workshop, which was part of the Unimod 2022 programme, at the School of Mathematics at the University of Leeds. We would like to thank the organisers of the very well-organised workshop and the School of Mathematics for its hospitality. We would also like to thank S. Braunfeld and N. Ramsey for the discussions during the Model Theory Conference in honour of Ludomir Newelski's 60th birthday, which led to \Cref{example-full-product-bad}, and D. Macpherson for his valuable comments on previous versions of this paper. Part of this research was also conducted in the Nesin Matematik Köyü, and we would like to thank it for its hospitality.

\paragraph{Notation} We assume familiarity with basic model theory, and refer the reader to \cite{Hodges1993} or \cite{TZ12}, for the relevant model-theoretic background. Most of the notation we will use throughout this paper is standard. Languages will be denoted by $\lL,\lL'$ etc. We abusively also denote by $\lL$ the set of all $\lL$-formulas, and write $\Sig(\lL)$ to specifically refer to the signature. If $f$ is a function symbol and $R$ is a relation symbol, we denote by $\arity(f)$ and $\arity(R)$ their respective arities.  Structures are typically denoted by $\sM,\sN,\dots$ while their respective base sets are denoted by $M,N,\dots$. Classes of structures are usually denoted by $\cK$. By a \emph{class of theories}, we mean a collection of theories with a certain elementary property $P$, such as the class $\ncodingcla{\cK}$ of theories who do not `code' a class $\cK$ of structures. 

\tableofcontents

\section{Preliminaries}\label{sec:prelims}

\subsection{Reducts}

A reduct of an $\lL$-structure $\sM$ can be defined as a structure obtained from $\sM$ by forgetting some of the symbols in $\lL$. We will use a more general notion of a reduct, sometimes also called a \emph{definable} or \emph{first-order reduct}:

\begin{definition}\label{def:definable reduct}
   Let $\sM_0$ and $\sM_1$ be structures on the same domain. 
   \begin{itemize}
       \item We say that $\sM_0$ is a \emph{first-order reduct}, \emph{$\emptyset$-definable reduct} (or just \emph{reduct} for short) \emph{of $\sM_1$} if every relation and every function in $\sM_0$ is definable in $\sM_1$ without parameters, and write $\sM_0\reduces_{FO}\sM_1$, or $\sM_0\reduces\sM_1$. In this case, we also say that $\sM_1$ is an \emph{expansion of $\sM_0$}. 
        \item We say that $\sM_0$ and $\sM_1$ are \emph{interdefinable}, and we write $\sM_0\reducesboth \sM_1$, if both $\sM_0\reduces\sM_1$ and $\sM_1\reduces\sM_0$ hold. 
        \item We say that $\sM_0$ is a \emph{quantifier-free reduct} of $\sM_1$ and we write $\sM_0\reducesqf\sM_1$ if every relation and every function in $\sM_0$ is quantifier-free definable in $\sM_1$ without parameters. 
        \item We say that $\sM_0$ and $\sM_1$ are \emph{quantifier-free interdefinable} if both $\sM_0\reducesqf\sM_1$ and $\sM_1\reducesqf\sM_0$ hold, and we write $\sM_0\reducesbothqf\sM_1$.
\end{itemize}
\end{definition}

Reducts of structures and, in particular, reducts of ultrahomogenous structures\footnote{Recall: $\sM$ is \emph{ultrahomogeneous} if every partial isomorphism between finite substructures of $\sM$ extends to an automorphism of $\sM$. We will call $\sM$ \emph{finitely homogeneous} if it is ultrahomogeneous in a finite relational language.}, play an important role in this paper, as they will be used to construct dividing lines between tame and wild structures. Notice that interdefinability (of structures on the same domain) is a much more restrictive condition than saying that two theories define each other. For instance, two independent orders on $\mathbb{Q}$ define each other (both are models of the same theory) but are not interdefinable.

To give some additional context, we recall here a well-known conjecture, due to Thomas: 

\begin{conjecture}[{Thomas, \cite{Tho96}}]\label{conjecture:thomas strong}
    If $\sM$ is a countable ultrahomogeneous structure in a finite relational language, then $\sM$ has finitely many reducts (up to interdefinability).
\end{conjecture}

This conjecture has been verified for many well-known examples (e.g. in \cite{Cameron_1976} it is shown that $(\mathbb{Q},\leq)$ has $3$ \emph{proper} (i.e. different from the structure itself and an infinite set) reducts up to interdefinability; in \cite{Thomas_1991} it is shown that the random graph has $3$ proper reducts up to interdefinability; in \cite{BPP2015} it is shown that the ordered random graph has $42$ proper reducts up to interdefinability, etc.) but the general case remains open.

It is well known that the set of reducts of a structure $\sM$ forms a lattice with respect to the relation $\reduces$. Moreover, if $\cR$ is the set of reducts of $\sM$, up to interdefinability, then the lattice $(\cR,\reduces)$ is precisely the opposite lattice of $(\{\aut(\sR): \sR\reduces \sM\},\leq)$. Recall that $\{\aut(\sR): \sR\reduces \sM\}$ is precisely the set of closed subgroups of $\sym(\sM)$ (with respect to the product topology) containing $\aut(\sM)$. 

\subsection{Standard ultrahomogeneous structures} \label{sub:standard-structs}

In this \namecref{sub:standard-structs} we fix some of our notation and recall some of the standard relational structures that we will be using throughout this paper. Before we do this, we recall some basic definitions. Let us start by recalling the \emph{(structural) Erd\H{os}-Rado partition arrow}. Let $\lL$ be a countable first-order language. Given $\lL$-structures $A\subseteq B\subseteq C$, we will write $\binom{B}{A}$ for the set of all embeddings of $A$ into $B$.  Given $k\in\Nbb$, we write $C\to (B)^A_k$ to mean that for every $k$-colouring $c:\binom{C}{A}\to \{1,\dots,k\}$ there is some $B'\in\binom{C}{B}$ such that $c$ restricted to $B'$ is constant. 

\begin{definition}[HP, JEP, (S/F)AP, RP]\label{def:hp-jep-ap-erp}
    A class of $\lL$-structures $\cla{C}$ has the:
	\begin{enumerate}
		\item \emph{Hereditary Property} (HP) if whenever $A\in\cla{C}$ and $B\subseteq A$ we have that $B\in\cla{C}$.
		\item \emph{Joint Embedding Property} (JEP) if whenever $A,B\in\cla{C}$ there is some $C\in\cla{C}$ such that both $A$ and $B$ are embeddable in $C$.
		\item \emph{Amalgamation Property} (AP) if whenever $A,B,C\in\cla{C}$ are such that $A$ embeds into $B$ via $e:A\hookrightarrow B$ and into $C$ via $f:A\hookrightarrow C$ there exist some $D\in\cla{C}$ and embeddings $g:B\hookrightarrow D$, $h:C\hookrightarrow D$ such that $g\circ e = h\circ f$.
        \item The \emph{Strong Amalgamation Property} (SAP) if whenever $A,B,C\in\cla{C}$ are such that $A$ embeds into $B$ via $e:A\hookrightarrow B$ and into $C$ via $f:A\hookrightarrow C$ there exists some $D\in\cla{C}$ and embeddings $g:B\hookrightarrow D$, $h:C\hookrightarrow D$ such that $g\circ e = h\circ f$, and moreover for all $b\in B$ and $c \in C$, if $g(b) = h(c)$, there is some $a\in A$ such that $e(a) = b$ and $f(a) = c$.
        \item The \emph{Free Amalgamation Property} (FAP) if $\lL$ is a relational language, and for all $R\in\mathsf{Sig}(\lL)$ we have that $R^{A\otimes_C B}=R^A\cup R^B$.
        \item \emph{Ramsey Property} (RP) if whenever $A,B\in\cla{C}$ are such that $A\subseteq B$, then there is some $C\in\cla{C}$ such that $C\to (B)^A_2$.
	\end{enumerate}
    We say that a countable class $\cla{C}$ is a \emph{Fra\"iss\'e class} if it has HP, JEP and AP, and we say that it is a \emph{Ramsey class} if it has HP, JEP and RP.
\end{definition}

We have (FAP)$\Rightarrow$ (SAP) $\Rightarrow$ (AP), and both implications here are strict. \emph{Fra\"iss\'e's theorem} (see, for instance, \cite[Theorem $7.1.2$]{Hodges1993}) tells us that if $\cla{C}$ is a Fra\"iss\'e class then there is a unique (up to isomorphism) countable ultrahomogeneous structure $\sM$ whose age is $\cla{C}$. We will denote this $\sM$ by $\mathrm{Flim}(\cla{C})$. 

Given an arbitrary class of structures $\cC$ there is a natural way of closing $\cC$ under substructures, making it satisfy HP. We call this the \emph{hereditary closure of $\cC$} and denote it by $\HC(\cC)$. Explicitly, for a class of structures $\cC$: 
\[
    \HC(\cC):=\{B\subseteq A : A\in \cC\}.
\]

Conventionally, for a structure $\sM$, we denote by $\age(\sM)$ the class of finitely generated substructures that embed into $\sM$. For a class of structures $\cC$, we sometimes overload this notation by writing $\age(\cC)$ for the class of finitely generated structures embeddable in some structure $A\in \cC$. 

\begin{remark}
    In the notation above, the following are all clear:
\begin{enumerate}
    \item $\HC(\cC)$ has HP.
    \item $\age(\{\sM\}) = \age(\sM)$.
    \item $\age(\cC) \subseteq \cC$ if $\cC$ has HP.
    \item $\age(\HC(\cC)) = \age(\cC) = \age(\age(\cC))$.
\end{enumerate}
\end{remark}
 
\begin{fact}[{\cite[Theorem 4.2(i)]{Nesetril_2005}}]
A Ramsey class of finite structures is a Fraïssé class.
\end{fact}

\paragraph{Basic classes of structures}
We recall and fix notation for some basic classes of structures, which have already been mentioned in the introduction:
\begin{itemize}
    \item $\cla{E}$ denotes the class of all finite sets in the language of pure equality.
    \item $\cla{LO}$ denotes the class of finite (total) orders $(X,<)$, where $<$ is a binary relation symbol for the order relation.
    \item $\cla{CO}$ denotes the class of finite cyclic orders $(X,CO)$ where $CO$ is a ternary relation symbol for the cyclic order relation (discussed in more detail later).
\end{itemize}

We will now introduce two well-known ultrahomogeneous structures, namely the $C$-relation and its reduct, the $D$-relation\footnote{We deeply thank D. Bradley-Williams for his useful and generous comments on $D$-relations and related issues.}. For more details on these structures, we direct the reader to \cite{AM22} or \cite{BJP16}.

\paragraph{Generalised chain relations}

\begin{definition}[{\cite[Paragraph 3.3]{BJP16}}]\label{Crelation}
    A ternary relation $C(x;y,z)$ on a set $X$ is called a \emph{$C$-relation} if for all $a,b,c,d\in X$ we have that:
    \begin{enumerate}
        \item\label{c1} $C(a;b,c) \limplies C(a;c,b)$;
        \item\label{c2} $C(a;b,c) \limplies \lnot C(b;a,c)$;
        \item\label{c3} $C(a;b,c) \limplies (C(a;d,c)\lor C(d;b,c))$;
        \item\label{c4} $a\neq b\limplies C(a;b,b)$.
    \end{enumerate}
    We say that a $C$-relation is \emph{derived from a binary tree} or that it is \emph{binary branching} if, in addition, for all distinct elements $a,b,c\in X$ we have that: 
    \begin{enumerate}\setcounter{enumi}{4}
        \item\label{c5} $C(a;b,c)\lor C(b;a,c)\lor C(c;a,b)$.
    \end{enumerate}
    Let $\prec$ be a total order on $X$. We say that $\prec$ is \emph{convex} for $C$ if for all $a,b,c\in X$, if $C(a; b, c)$ and $a \prec c$, then either $a \prec b \prec c$ or $a \prec c \prec b$. We denote by $\cla{OC}$ the class of all convexly ordered finite binary branching $C$-relations.
\end{definition}

\begin{fact}[{\cite[Theorem 5.1]{Bodirsky2015}}]
    The class $\cla{OC}$ is a Ramsey class.
\end{fact}

\paragraph{Generalised direction relations}

\begin{definition}[{\cite[Paragraph 3.4]{BJP16}}] \label{Drelation}
A quaternary relation $D(x, y; z, w)$ on a set $X$ is called a \emph{$D$-relation} if for all $x,y,z,w,a \in X$ we have that:
\begin{enumerate}
    \item\label{d1} $D(x, y; z, w)\rightarrow D(y, x; z, w) \wedge D(x, y; w, z) \wedge D(z,w; x, y)$;
    \item\label{d2}  $D(x, y; z, w)\rightarrow \neg D(x, z; y, w);$
    \item\label{d3} $D(x, y; z, w) \rightarrow (D(a, y; z, w) \vee D(x, y; z, a))$;
    \item\label{d4} $(x\neq z \wedge y\neq z) \rightarrow D(x, y; z, z )$.
\end{enumerate}
    Similarly to \Cref{Crelation}, we say that a $D$-relation is \emph{derived from a binary tree} or that it is \emph{binary branching} if, in addition for any four elements $x,y,z,w\in X$, if at least $3$ of which are distinct then we have: 
    \begin{enumerate}\setcounter{enumi}{4}
        \item\label{d5} $D(x,y;z,w)\lor D(x,z;y,w)\lor D(x,w;y,z)$.
    \end{enumerate}
From a binary tree, we obtain such a $D$-relation on the set of leaves by setting $D(a,b;c,d)$ if the paths between $a$ and $b$ and between   $c$ and $d$ are disjoint. For instance, if $a,b,c$ and $d$ are arranged as follows:

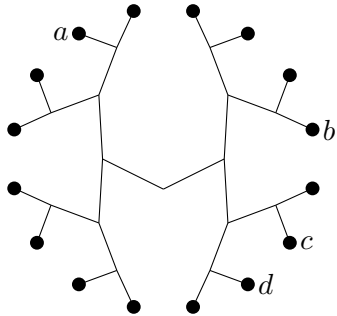
\begin{center}
\begin{tikzpicture}[scale=2]
    \draw (180/16:1) node[right] {$b$};
    \draw (11*90/8:1) node[left] {$a$};
    \draw (-3*90/8:1) node[right] {$c$};
    \draw (-5*90/8:1) node[right] {$d$};
     \tikzmath{
     \S = 5;
        function branch(\x,\y,\s,\a) {
            if (\y > 1-1/\s) then {
                {\fill (\x:\y) circle (1.3pt);};
            } else {
                {\draw (\x:\y) -- (\x-\a:\y+1/\s);
                \draw (\x:\y) -- (\x+\a:\y+1/\s);};
                \y1=\y+1/\s; \x1=\x-\a; \x2=\x1+2*\a; \a1=\a/2;
                branch(\x1,\y1,\s,\a1); branch(\x2,\y1,\s,\a1);
            };
        };
     branch(-90,1/\S,\S,90);
     }
\end{tikzpicture}
\end{center}

    then $D(a,b;c,d)$ holds.
    
    Recall that a (strict) \emph{cyclic order} $CO$ on $X$ is a ternary relation such that, for all $x,y,z,w \in X$:
    \begin{enumerate}
    \item\label{CO1} $CO(x, y,z)\rightarrow CO(y,z,x) $;
    \item\label{CO2}  $CO(x, y,z)\rightarrow \neg CO(x,z,y) $
    \item\label{CO3} $CO(x, y,z) \wedge CO(y, y,w)\rightarrow CO(x, y,w)$;
    \item\label{CO4} if $x,y,z$ are distinct, then $CO(x, y,w)$ or $CO(x, w,y)$.
\end{enumerate}

    A cyclic order $CO$ on a binary $D$-relation $X$ is \emph{convex} for $D$ if for all distinct $x,a,b,c\in X$:
    \[CO(a,b,c)\wedge  CO(c,x,a) \rightarrow D(x,a;b,c)\vee D(a,b;c,x).\]
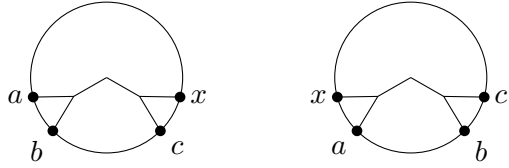
\begin{center}
    \begin{tikzpicture}
        \draw (0,0) circle (1) ;
        \draw (0,0) -- (-30:0.5)-- (-15:1);
        \fill (-15:1) circle (2pt);
        \draw (-15:1) node[right] {$x$};
        \draw (-30:0.5) -- (-45:1);
        \fill (-45:1) circle (2pt);
        \draw  (-45:1) node[below right] {$c$};
        \draw (0,0) -- (210:0.5)-- (195:1);
        \fill (195:1) circle (2pt);
        \draw  (195:1) node[left] {$a$};
        \draw (210:0.5) -- (225:1);
        \fill (225:1) circle (2pt);
        \draw   (225:1) node[below left] {$b$};

        \begin{scope}[shift={(4,0)}]
        \draw (0,0) circle (1) ;
        \draw (0,0) -- (-30:0.5)-- (-15:1);
        \fill (-15:1) circle (2pt);
        \draw (-15:1) node[right] {$c$};
        \draw (-30:0.5) -- (-45:1);
        \fill (-45:1) circle (2pt);
        \draw  (-45:1) node[below right] {$b$};
        \draw (0,0) -- (210:0.5)-- (195:1);
        \fill (195:1) circle (2pt);
        \draw  (195:1) node[left] {$x$};
        \draw (210:0.5) -- (225:1);
        \fill (225:1) circle (2pt);
        \draw   (225:1) node[below left] {$a$};
        \end{scope}
    \end{tikzpicture}
    \end{center}
Equivalently, we have for all distinct $a,b,c,d\in X$, 
\[CO(a,b,c)\wedge  CO(b,c,d) \rightarrow \neg D(a,c;b,d).\]
Graphically, this means that the following configuration doesn't occur:
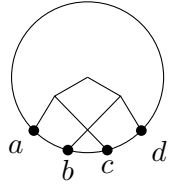
\begin{center}
     \begin{tikzpicture}
        \draw (0,0) circle (1) ;
        \draw (0,0) -- (-30:0.5)-- (255:1);
        \fill (-75:1) circle (2pt);
        \draw (-75:1) node[below] {$c$};
        \draw (-30:0.5) -- (-45:1);
        \fill (-45:1) circle (2pt);
        \draw  (-45:1) node[below right] {$d$};
        \draw (0,0) -- (210:0.5)-- (-75:1);
        \fill (255:1) circle (2pt);
        \draw  (255:1) node[below] {$b$};
        \draw (210:0.5) -- (225:1);
        \fill (225:1) circle (2pt);
        \draw   (225:1) node[below left] {$a$};
    \end{tikzpicture}

    \end{center}

    We shall denote by $\cla{COD}$ the class of all convexly ordered finite binary branching $D$-relations.
\end{definition}
We leave the following representation of a finite structure in $\cla{COD}$ to perhaps help convey some graphical intuition:
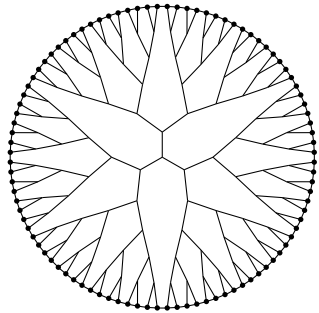
\begin{center}
\begin{tikzpicture}[scale=2]
    \draw (0,0) circle (1);
    \draw (0,0) -- (90:1/6);
    \draw (0,0) -- (-30:1/6);
    \draw (0,0) -- (-150:1/6);
     \tikzmath{
     \S = 6;
        function branch(\x,\y,\s,\a) {
            if (\y > 1-1/\s) then {
                {\fill (\x:\y) circle (0.5pt);};
            } else {
                {\draw (\x:\y) -- (\x-\a:\y+1/\s);
                \draw (\x:\y) -- (\x+\a:\y+1/\s);};
                \y1=\y+1/\s; \x1=\x-\a; \x2=\x1+2*\a; \a1=\a/2;
                branch(\x1,\y1,\s,\a1); branch(\x2,\y1,\s,\a1);
            };
        };
     branch(-30,1/\S,\S,30);
     branch(-150,1/\S,\S,30);
     branch(90,1/\S,\S,30);
     }
\end{tikzpicture}
\end{center}

\begin{fact}\label{fact:CODhasQE}
    The theory of dense cyclically ordered binary branching $D$-relations is complete, $\aleph_0$-categorical, and admits quantifier elimination in the language $\{D,CO\}$. It follows that $\cla{COD}= \age(\sM)$ is a Fraïssé class, where $\sM$ is the unique ultrahomogeneous countable model of the theory of dense cyclically ordered binary branching $D$-relations.
\end{fact}

For completeness, we give a sketch of a proof: 
\begin{proof}[Proof (sketch)]
    Let $T$ denote the theory of dense cyclically ordered binary branching $D$-relations and let $\sM$ be a model of $T$. By assumption, there are, up to equivalence, five kinds of literals in a single variable $x$, namely: $x=a$,$ x\neq a$, $CO(a,x,b)$, $D(a,b,c,x)$, $D(x,a,b,c)$, where $a,b,c$ are parameters in $\sM$ such that $CO(a,b,c)$ .  To simplify systems of literals, we need the following claim:
    \begin{claim}
        Let $a,b,c,x \in \mathcal{M}$ all distinct such that $CO(a,b,c)$. 
        Then $D(x,a,b,c)$ implies 
        \[\left(D(x,a,b,c) \wedge  CO(c,x,a) \right)\vee \left(D(b,c;a,x) \wedge CO(a,x,b) \right).\]
    \end{claim}

    \begin{center}
    \begin{tikzpicture}
        \draw (0,0) circle (1) ;
        \draw (0,0) -- (-30:0.5)-- (-15:1);
        \fill (-15:1) circle (2pt);
        \draw (-15:1) node[right] {$c$};
        \draw (-30:0.5) -- (-45:1);
        \fill (-45:1) circle (2pt);
        \draw  (-45:1) node[below right] {$b$};
        \draw (0,0) -- (210:0.5)-- (195:1);
        \fill (195:1) circle (2pt);
        \draw  (195:1) node[left] {$x$};
        \draw (210:0.5) -- (225:1);
        \fill (225:1) circle (2pt);
        \draw   (225:1) node[below left] {$a$};

        \begin{scope}[shift={(4,0)}]
        \draw (0,0) circle (1) ;
        \draw (0,0) -- (-30:0.5)-- (-15:1);
        \fill (-15:1) circle (2pt);
        \draw (-15:1) node[right] {$x$};
        \draw (-30:0.5) -- (-45:1);
        \fill (-45:1) circle (2pt);
        \draw  (-45:1) node[below right] {$a$};
        \draw (0,0) -- (210:0.5)-- (195:1);
        \fill (195:1) circle (2pt);
        \draw  (195:1) node[left] {$b$};
        \draw (210:0.5) -- (225:1);
        \fill (225:1) circle (2pt);
        \draw   (225:1) node[below left] {$c$};
        \end{scope}
    \end{tikzpicture}
    \end{center}
    \begin{claimproof}
        Assume $CO(a,b,c)$, $D(x,a;b,c)$ and $CO(a,x,c)$. We need to show that $ CO(c,x,b)$, as then, we will have $ CO(a,x,b)$.
        Assume not, then we have $\neg CO(c,x,b)\wedge  CO(c,a,b)$. By convexity, this implies  
        $D(x,c;a,b) \vee D(c,a;b,x) $. By Axiom \ref{d3} of \Cref{Drelation}, this implies $\neg D(x,a;b,c)$ and we have a contradiction.       
    \end{claimproof}
    
    Consider a system $S(x)$ of literals 
    \[
        \{CO(a,x,b) \ : \ (a,b)\in P \} \cup \{ D(a,b,c,x) \ : \ (a,b,c)\in L \} \cup \{ D(x,a,b,c) \ : \ (a,b,c)\in R \} \cup \{x\neq a \ : \ a\in Q \},
    \]
    where $P \subseteq \mathcal{M}^2$, $L,R \subseteq CO \subseteq \mathcal{M}^3$ and $Q$ contains the set of all parameters in $P,L$ and $R$. Using Axiom \ref{d5}, we may assume that if $a,b,c$ are parameters in $P, L$ or $R$ such that $CO(a,b,c)$ holds, then $(a,b,c)\in L \cup R$. By the previous claim, we may assume that if $(a,b,c)\in L\cup R$, then $(c,a)\in P$.

    One can see, by convexity, that $S(x)$ admits a solution if, and only if the subsystem
   
    \[
        \{CO(a,x,b) \ : \ (a,b)\in P\} \cup \{x\neq a \ : \ a\in Q \}
    \]
    has a solution. Since $\sM$ is a dense cyclic order, this system has a solution if and only if for all $(a,b),(a',b')\in P$,  $b\neq a'$ and  $\neg(CO(a,b,a')\wedge CO(b,a',b'))$. We can this way eliminate one existential quantifier, and this process does not depend on the model $\sM$ we considered. It follows that the theory eliminates quantifiers. Finally, since the language is finite relational, quantifier elimination implies completeness and $\aleph_0$-categoricity of the theory, as well as the ultrahomogeneity of its unique countable model.
\end{proof}

We will see in \cref{sec: Collapsing indiscernible as a witness for the Ramsey property} (\Cref{Example:CODisnotRamsey}) that, unlike $\cla{OC}$, the class $\cla{COD}$ is \emph{not} a Ramsey class, even augmented with a generic order.

\paragraph{Aside on Ordered $D$-relations}
One can of course consider a convex order on the $D$-relation.  A total order $\leq$ on $X$ is \emph{convex} for $D$ if for all $x,y,z,w\in X$ such that $x\leq y$ and $x\leq z$, if $D(x, y; z,w)$ either $x,y < z,w $ or $x < z,w < y$. We denote by $\cla{OD}$ the class of all convexly ordered finite binary branching $D$-relations. However, this does not give rise to a new structure. Indeed, we will show that $\cla{OD}$ and $\cla{OC}$ are interdefinable (without parameters).

First, we recall that a \emph{pointed} $D$-relation (i.e. a $D$-relation with a fixed ``named'' point) is quantifier-free interdefinable with a $C$-relation: 

\begin{fact}[{\cite[Theorems~22.1 and 23.4]{AN_1998}}]\label{fact:DCrelationInterdefinable}
    Let $(X,D)$ be a $D$-relation on a set $X$ and let $a\in X$ be any point. We can define a $C$-relation $C_0$ on $X_0 = X\setminus\{x\}$ by setting:
    \[
        C_0(x,y,z)\text{ if, and only if } D(a,x;y,z).
    \]
    Conversely, if $(X,C)$ is a $C$-relation on a set $X$ and $a\notin X$ then we can define a $D$-relation $D_0$ on $X\cup\{a\}$ by setting $D_0(x,y;z,w)$ if, and only if:
    \begin{itemize}
        \item $x=y$ and $x\neq z$, $x\neq w$; or $z=w$ and $z\neq x$, $z\neq y$;
    \end{itemize}
    or
    \begin{itemize}
        \item $x,y,w,z$ are all distinct, and at least one of the following holds:
            \begin{itemize}
                \item $x=a$ and $C(y;w,z)$; or $y=a$ and $C(x;w,z)$; or $z=a$ and $C(w;x,y)$; or $w=a$ and $C(z;x,y)$; or
                \item $C(x;w,z)\land C(y;w,z)$ or $C(z;x,y)\land C(w;x,y)$.
            \end{itemize}
    \end{itemize}
    Moreover, the two constructions are inverse to each other.
\end{fact}

If the $D$-relation is convexly ordered, we can recover $\emptyset$-definably a $C$-relation by setting, for any ordered triple $x,y,z \in X:$
\begin{align}\label{Eq:DefinitionCrelationFromConvexDRelation}
        C(x;y,z) \Leftrightarrow \forall a\, \exists b\leq a\, \ D(b,x;y,z). 
\end{align}

In particular, we have the following:

\begin{fact}\label{fact: interdefinability OD and OC}
The relation $C$ given by (\ref{Eq:DefinitionCrelationFromConvexDRelation}) is a convexly ordered $C$-relation. Conversely, the $D$-relation we started with is precisely the $D$-relation induced by $C$ as in \Cref{fact:DCrelationInterdefinable}. \emph{In particular, the structures $(X,D,<)$ and $(X,C,<)$ are $\emptyset$-interdefinable.}
\end{fact}

Note that we use the total order to obtain $\emptyset$-interdefinability (i.e. interdefinability without parameters), at the cost of using quantifiers in the defining formulas. The point is that the order allows us to talk about a (possibly) imaginary ``first'' element of $(X,<)$. First, we need a small lemma:

\begin{lemma}\label{lem:small-d-lemma}
    Let $(X,D,<)$ be a convexly ordered $D$-relation. Let $x,x',x'',y,z,w\in X$. If $x''\leq x' \leq x \leq y$ , $x\leq z$ and $D(x,y;z,w)\wedge D(x'',y;z,w)$, then $D(x',y;z,w)$.
\end{lemma}
\begin{proof}
    Assume for a contradiction, that $D(x',y;z,w)$ does not hold. By Axiom \labelcref{d3}, since $D(x,y;z,w)$ and $D(x'',y;z,w)$ hold, we have $D(x,y;z,x')$ and $D(x'',y;z,x')$. Now, by convexity and the fact that $x''\leq x' \leq x$, from $D(x,y;z,x')$ we get that: 
    \[
        x' < x,y < z.
    \]
    Similarly, using the other $D$-relations and $x'' \leq x\leq y$, we have   
    \[
        x'' < x',z < y,
    \] 
    which is a contradiction.
\end{proof}

\begin{proof}[Proof of \Cref{fact: interdefinability OD and OC}]
    We start with the following claim:
    \begin{claim}\label{claim:initial-segment-characterisation}
        $C(x;y,z)$ holds if, and only if, on an initial segment $I$ of $X$, we have $D(b,x;y,z)$ for all $b\in I$.
    \end{claim}
    \begin{claimproof} 
        To see this, assume $C(x;y,z)$. For all $b,b'$ such that $b'\leq b<x,y,z$, if $D(b,x;y,z)$ holds, then $ D(b',x;y,z)$ also holds. Indeed, since $C(x;y,z)$ holds, there is $b''\leq b'$ such that  $D(b'',x;y,z)$ holds. We have then that $D(b',x;y,z)$ by convexity and \Cref{lem:small-d-lemma}. Then, the equivalence is immediate.        
    \end{claimproof}

    Let us now show that $C$ is indeed a $C$-relation. Let $x,y,z$ be elements in $X$. To show Axioms \labelcref{c1,c2}, assume $C(x;y,z)$. Then there is an initial segment $I$ such that for all $b\in I$,  $D(b,x;y,z)$ and therefore $D(b,x;z,y)$ and $\neg D(b,y;x,z)$ by definition of a $D$-relation. We have therefore $C(x;z,y)$ and $\neg C(y;x,z)$.
    
    To show Axiom \labelcref{c3}, assume $C(x;y,z)$ and pick $a\in X$. Then on an initial segment $I$, we have $D(x,b;y,z)$  for all $b\in I$. By Axiom \labelcref{d3} of the definition of a $D$-relation, we have  $D(a,b;y,z)$ or $D(x,b;y,a)$ for all $b\in I$. Using the convexity of $D$, either $D(a,b;y,z)$ holds for any $b$ in an initial segment, or $D(x,b;y,a)$ holds for any $b$ in an initial segment. Therefore $C(a;y,z)$ or $C(x;y,a)$.
    We have Axiom \labelcref{c4} by definition. The fact that $C$ is derived from a binary tree (Axiom \labelcref{c5}) and is convex for $<$ can be shown similarly to Axiom \labelcref{c3}. 

    Finally, to show that $D$ can be recovers from $C$, we observe that for any $a<x,y,z,w$ in $X$, $D(x,y;z,w)$ holds if, and only if, 
    $(D(a,x;z,w)\wedge D(a,y;z,w) )\vee (D(x,y;w,a)\wedge D(x,y;z,a))$ holds. This follows, for instance, from \Cref{fact:DCrelationInterdefinable}. We deduce from \Cref{claim:initial-segment-characterisation} that 
    $D(x,y;z,w)$ holds if, and only if, $(C(x;z,w)\wedge C(y;z,w) )\vee (C(w;x,y)\wedge C(z,x,y))$ as wanted.
\end{proof}

\paragraph{Ordered Hypergraphs} Fix $m\in\Nbb$. Let $\lL_0=\{R_i:i<m\}$ where each $R_i$ is a relation symbol of arity $r_i$, for each $i<m$. A \emph{hypergraph of type $\lL_0$} or \emph{$\lL_0$-hypergraph} is a structure $\left(A,\left(R_i\right)_{i<m}\right)$ such that, for all $i<m$:
\begin{itemize}
    \item (\emph{Uniformity}): If $R_i(a_0,\dots,a_{r_i-1})$ then all $a_0,\dots,a_{r_i-1}$ are distinct.
    \item (\emph{Symmetry}): If $R_i(a_0,\dots,a_{r_i-1})$ then we also have $R_i\left(a_{\sigma(0)},\dots,a_{\sigma(r_i-1)}\right)$ for any permutation ${\sigma\in S_{r_i}}$.
\end{itemize}

The point is that each $R_i$ is interpreted in $A$ as an $r_i$-ary ``hyperedge'' relation, i.e. $R_i\subseteq[A]^{r_i}$.

Let $\mathcal{L}_{0}^+=\mathcal{L}_0\cup\{<\}$. If $\sM=(A,(R_i)_{i<m},<)$ is an $\mathcal{L}_0^+$-structure whose $\mathcal{L}_0$-reduct is an $\mathcal{L}_0$-hypergraph and $<$ is interpreted as a linear order in $M$, then we say that $M$ is an \emph{ordered $\mathcal{L}_0^+$-hypergraph}. 

Let $\cla{C}$ be the class of all linearly ordered finite $\mathcal{L}_0$-hypergraphs (for arbitrary $\mathcal{L}_0^+$). Then, $\cla{C}$ is a Fraïssé class and its Fraïssé limit is the \emph{ordered random $\mathcal{L}_0^+$-hypergraph}, whose order is isomorphic to $(\mathbb{Q},<)$. 

\begin{fact}[\cite{NR1977}]
For any finite $\mathcal{L}_0^+$, let $\cla{C}$ be the class of all ordered $\mathcal{L}_0^+$-hypergraphs. Then $\cla{C}$ is a Ramsey class.
\end{fact}

In particular, if $\mathcal{L}_0^+ = \{<,R\}$, where $R$ is a single relation symbol of arity $n$ we call an ordered $\mathcal{L}_0^+$-hypergraph an \emph{ordered $n$-\textit{uniform} hypergraph}. The Fraïssé limit of the class of ordered $n$-uniform hypergraphs is the \emph{ordered random $n$-uniform hypergraph}, which we denote by $\cla{OH}_n$. 

\begin{remark} A first-order $\mathcal{L}_0$-structure $(M,<,R)$ is a model of $\Th(\cla{OH}_n)$ if, and only if:
\begin{itemize}
    \item $(M,<,R)$ is an ordered $n$-uniform hypergraph, in the sense above.
    \item $(M,<)$ is a model of $\mathsf{DLO}$.
    \item For all finite disjoint subsets $A_0,A_1\subseteq M^{n-1}$ and any $b_0,b_1\in M$ such that $b_0<b_1$, there is some $b\in M$ such that:
    \begin{itemize}
        \item $b_0<b<b_1$.
        \item For every $(a_{0,1},\dots,a_{0,n-1})\in A_0$, and $(a_{1,1},\dots,a_{1,n-1})\in A_1$ we have that:
        \[
        R(b,a_{0,1},\dots, a_{0,n-1})\text{ and } \lnot R(b,a_{1,1},\dots,a_{1,n-1}).
        \]
    \end{itemize}
\end{itemize}
\end{remark}

In particular, an ordered random $1$-hypergraph is a dense linear order with a dense co-dense subset. We will denote $\cla{OG}$, the ordered random $2$-hypergraph, $\cla{OH}_2$, (since $2$-hypergraphs are just graphs). 

\subsection{Product constructions}\label{sec:products}
We now introduce the two constructions that we will consider in this paper: the full product and the lexicographic sum. For both of them, we recall a quantifier elimination result relative to their factors. This will be later used in \Cref{sec:transfer} in order to describe generalised indiscernible sequences (see \Cref{sub:generalised-indiscernibles}) in full products and lexicographic sums, in terms of generalised indiscernible sequences in their factors. This description will then be our main tool for proving transfer principles in these products.

\subsubsection{Full Product}

\begin{definition}\label{def:fullproduct}
For $i\in \{1,2\}$, let $\sM_i$ be an $\lL_{\sM_i}$-structure with main sort $M_i$. We define the \emph{full product of $\sM_1$ and $\sM_2$}, denoted $\sM_1 \boxtimes \sM_2$, to be the (multisorted) structure:
\[
    \{M_1\times M_2,\sM_1,\sM_2,\pi_{M_1}, \pi_{M_2}\}
\]
where for $i\in \{1,2\}$, $\pi_{M_i}$ is the natural projection $M_1\times M_2 \rightarrow M_i$ and the sort $\sM_1$ and $\sM_2$ are equipped with their respective structure.\footnote{This notion of product should not be confused with the \emph{Feferman product} (of two structures); full products are equipped with projection maps, which cannot always be recovered in the Feferman product.} We denote by $\lL_{\sM_1\boxtimes\sM_2}$ the corresponding language (which contains \emph{disjoint} copy of $\lL_{\sM_1}$ and $\lL_{\sM_2}$).
\end{definition}

The full product $\sM_1 \boxtimes \sM_2$ is sometimes called the \emph{direct product}, and can be described in a one-sorted language. However, the language $\lL_{\sM_1\boxtimes\sM_2}$ has the following immediate advantage: 

\begin{fact}\label{fact: QE relative in full product}
For $i\in \{1,2\}$, let $\sM_i$ be as above. Then $\sM_1 \boxtimes \sM_2$ eliminates quantifiers relative to $\sM_1$ and $\sM_2$. Equivalently, every $\lL_{\sM_1\boxtimes\sM_2}$-formula is equivalent to a formula without quantifier in the main sort $M_1\times M_2$.
\end{fact}

\begin{proof}
    We may Morleyise $\sM_1$ and $\sM_2$, so, without loss, we may assume that they eliminate quantifiers in their respective relational languages $\lL_1$ and $\lL_2$. We prove the \namecref{fact: QE relative in full product} by induction on the complexity of an $\lL_{\sM_1\fullprod\sM_2}$-formula $\vphi(\bar{x})$:
    \begin{itemize}
        \item If $\vphi = R$ for some $R\in \lL_i$, then, by definition $\sM_1\fullprod\sM_2 \models \vphi(\bar{x})\iff \sM_i\models \vphi(\pi_i(\bar{x}))$.
        \item\label{item relative QE equality} If $\vphi(x,y)$ is $x=y$, then, by definition,  $\sM_1\fullprod\sM_2 \models \vphi(x,y)\iff \sM_1\models \vphi(\pi_1(x,y))\land \sM_2\models \vphi(\pi_2(x,y))$.
        \item If $\vphi$ is of the form $\neg \psi$ or $\psi_1\land \psi_2$, then the \namecref{fact: QE relative in full product} follows from the induction hypothesis.
        \item Finally, assume $\vphi(\bar{x})$ is of the form $\exists y\, \psi(\bar{x};y)$. Applying the induction hypothesis to $\psi$, we may assume $\psi$ is in \emph{disjunctive normal form}, i.e. $\psi(\bar{x};y) = \bigvee_{i\in I}\bigwedge_{j\in J_i}\theta_{i,j}(\bar{x};y)$, where $I, \Set{J_i}_{i\in I}$ are finite and $\theta_{i,j}$ are all atomic or negated atomic. As disjunction always commutes with the existential quantifier, we may further assume $\psi(\bar{x};y) = \bigwedge_{j\in J}\theta_{j}(\bar{x};y)$. 
        Breaking up equality in $M_1\times M_2$ to a conjunction of equalities in $\sM_1$ and $\sM_2$ as above, we may assume $\psi$ is 
        \[\psi(\bar{x};y) = \psi_1(\pi_1(\bar{x};y))\land\psi_2(\pi_2(\bar{x};y)).\]
        The \namecref{fact: QE relative in full product} follows since 
        \[\exists y \in M_1\times M_2\, \psi_1(\pi_1(\bar{x};y))\land\psi_2(\pi_2(\bar{x};y))\iff \exists y \in M_1, \psi_1(\pi_1(\bar{x};y)) \land \exists y \in M_2, \psi_2(\pi_2(\bar{x};y)).\]
    \end{itemize}
\end{proof}

\begin{definition}
    Let $\cla{C}_1$ and $\cla{C}_2$ be two classes of structures. We denote by $\cC_1\boxtimes \cC_2$ the smallest hereditary class of structures containing $C \boxtimes D$ for all $C\in \cC_1$ and $D \in \cC_2$.
\end{definition}

\subsubsection{Lexicographic product}
The \emph{lexicographic sum} of relational structures was studied by the first author in \cite{Mei16}. It is a method of constructing an $\lL$-structure $\sM[\sN]$ from two $\lL$-structures $\sM$ and $\sN$, where $\lL$ is a relational language, in a way that generalises the lexicographic (wreath) product of graphs. We recall here the relevant definitions and a quantifier elimination result.

\begin{definition} \label{DefinitionLexicographicProduct2}
    Let $\sM$ be an $\lL_\sM$-structure and $\mathfrak{N}=\{\sN_a\}_{a\in \sM}$ be a collection $\mathfrak{N}$ of $\lL_\mathfrak{N}$-structures indexed by $\sM$. The \emph{lexicographic sum of $\mathfrak{N}$ with respect to $\sM$}, denoted by $\sM[\mathfrak{N}]$, is the multisorted structure with:
    \begin{itemize}
        \item a main sort with base set  $S:=\bigcup_{a\in M} \{a\}\times \sN_a$, 
        \item a sort for the structure $\sM$, 
        \item the natural projection map $v: S\rightarrow \sM$.
    \end{itemize}
     To distinguish symbols in the main sort $S$ from symbols in the ribs $\sN_a$, we denote by $\lL_{\bullet, \mathfrak{N}} := \{P_\bullet\}_{P\in \lL_\mathfrak{N}}$ a copy of $\lL_\mathfrak{N}$.
The set $S$ is equipped with an $\lL_{\bullet,\mathfrak{N}}$-structure: 
    \begin{itemize}
        \item if $P\in \lL_\mathfrak{N}$ is an $n$-ary predicate, then
         \begin{align*}
            P_\bullet^{\sM[\mathfrak{N}]} :=  \left\{\left((a,b_1),\ldots, (a,b_n)\right) \ \vert \ a\in M \ \text{and } \sN_a \models P(b_1,\ldots,b_n)\right\}.
         \end{align*}
         \item if $f\in \lL_\mathfrak{N}$ is an $n$-ary function symbol, then
         \begin{align*}
            f_\bullet^{\sM[\mathfrak{N}]}((a_1,b_1),\ldots, (a_n,b_n)) := \begin{cases} (a_n,f^{\sN_a}(b_1,\cdots,b_n) ) & \text{ if } a=a_1=\cdots=a_n,\\
            u & \text{ otherwise}.
            \end{cases}.
         \end{align*}
        \end{itemize}
    where $u$ is a specific constant which stands for `undetermined'. We can also pick a constant in the language.          
  
  We denote by $\lang{L}_{\sM[\mathfrak{N}]}$ the multisorted language 
  \[
  (S,\lang{L}_{\bullet,\mathfrak{N}}) \cup (M,\lang{L}_{\sM}) \cup \{v:S \rightarrow M\}.
  \]
  If for every $a\in \sM$, $\sN_a$ are isomorphic copies of a structure $\mathcal{N}$, we simply denote the lexicographic sum by $\sM[\sN]$, and we call it the \emph{lexicographic product of $\sM$ and $\sN$}.
\end{definition}

\begin{remark}
    Notice that the projection to the second coordinate is not in the language of the lexicographic product. But, to simplify notation, for all $a\in\sM$ we identify $\sN_a$ with $\{a\}\times\sN_a$ and write $\sN_a\vDash\vphi(c)$ for $c=(a,n)\in\sM[\mathfrak{N}]$ and $\vphi$ is an $\lL_\mathfrak{N}$-formula such that $\sN_a\vDash\vphi(n)$.
\end{remark}

    The sort $\sM$ in $\sM[\mathfrak{N}]$ can come with additional structure: we can indeed define the predicates $P_\vphi^\sM:= \{a\in \sM \ \vert \ \sN_a\models \vphi \}$ for all $\lL_\mathfrak{N}$-sentences $\vphi$. For technical reasons, we do not add these predicates in the language $\lang{L}_{\sM[\mathfrak{N}]}$, but they play an important role in our analysis.

    Any $\rho\in S_1^{\Th(\sM)}$ can be seen as a filter on the set of $\emptyset$-definable (unary) subsets of $\sM$. By \Los, the theory of an ultraproduct $\prod_{\mathcal{U}}\sN_a$ depends only on the type $\rho$ that $\mathcal{U}$ extends. By abuse of notation, we write such an ultraproduct $\sN_\rho$ or $\prod_{\rho}\sN_a$.

    An elementary extension of $\sM[\mathfrak{N}]$ is of the form $\tilde{\sM}[\tilde{\mathfrak{N}}]$ where $\tilde{\sM}$ is an $\lL_{\sM}$-structure and  $\tilde{\mathfrak{N}}$ is a collection of $\lL_{\mathfrak{N}}$-structures $\tilde{\sN}_{\tilde{a}}$ such that:
    \begin{itemize}
        \item $\tilde{\sM}$ is an elementary extension of $ \sM$ in $\lL_\sM \cup \{P_\vphi\}_{\vphi\in\lL_{\mathfrak{N}}}$;
        \item For $a\in \sM$, we have that $\tilde{\sN}_{a}$ is an elementary extension of $\sN_a$;
        \item For $\tilde{a} \in \tilde{\sM}\setminus\sM$, we have that $\tilde{\sN}_{\tilde{a}}$ is elementary equivalent to the ultraproduct $\prod_{\rho}\sN_a$ where $\rho=\tp(\Tilde{a})\in S_1^{\Th(\sM)}$.
    \end{itemize}

\begin{fact}\label{thm:QuantifierEliminationLexSum}
 Consider the lexicographic sum $\sS:=\sM[\mathfrak{N}]$ of a class of $\lL_\mathfrak{N}$-structures $\mathfrak{N}:= \{\sN_a\}_{a\in \sM}$ with respect to an $\lL_\sM$-structure $\sM$. Assume that 
 \begin{enumerate}
     \item For all sentences $\vphi\in \lL_{\mathfrak{N}}$, the set $\{a\in \sM \ \vert \ \sN_a \models \vphi \}$ is $\emptyset$-definable in $\sM$,
     \item For all $\rho\in S_1^{\Th(\sM)}$, $\sN_\rho$ admits quantifier elimination in $\lL_{\mathfrak{N}}$.
 \end{enumerate}
    Then $\sM[\mathfrak{N}]$ eliminates quantifiers relative to $\sM$ in $\lL_{\sM[\mathfrak{N}]}$.   
\end{fact}
For a proof, the reader can refer to \cite[Theorem~2.7]{Mei16}.

\begin{definition}
    Let $\cla{C}_1$ and $\cla{C}_2$ be two classes of finite structures. We denote by $\cC_1[\cC_2]$ the class of lexicographic sums $C[(D_c)_{c\in C}]$ where $C\in \cC_1$ and $D_c \in \cC_2$ for every $c\in C$.
\end{definition}

\subsection{Dividing lines in model theory}

Throughout this section, $T$ will always denote a complete $\lL$-theory and $\mM\vDash T$ a $\kappa(\Mbb)$-saturated and $\kappa(\Mbb)$-homogeneous (monster) model of $T$, for some very large cardinal $\kappa(\Mbb)$. Unless otherwise stated all (tuples of) elements and subsets will come from this monster model and will be \emph{small}, i.e. of size less than $\kappa(\Mbb)$.

\subsubsection{NIP and higher-arity generalisations}

In this section, we briefly recall some basic definitions from neostability theory, in particular, NIP and its higher-arity generalisations. The notion of NIP, originating in the work of Shelah, has been studied intensively in the past years, and the definition is standard, but we include it here for the sake of keeping this paper as self-contained as possible. For more information on NIP, we direct the reader to \cite{Sim15}.

\begin{definition}[Independence Property, NIP]
    We say that a formula $\vphi(\bar x;\bar y)$ has the \emph{independence property} in $T$ if there exist $(\bar a_i)_{i\in \Nbb}$ and $(\bar b_I)_{I\subseteq\Nbb}$ such that:
    \[
        \vDash \vphi(\bar b_{I},\bar a_{i}) \text{ if, and only if, } i\in I.
    \]
    We say that $T$ is \emph{dependent} or \emph{NIP} (No Independence Property) if no formula has the independence property in $T$.
\end{definition}

In this text, we will consider a ``higher-arity'' generalisation of NIP, also due to Shelah in \cite{She14}, and later studied in more detail in \cite{CPT19}. This will be one of our main examples of dividing lines arising from generalised indiscernibles. 

\begin{definition}[$n$-Independence Property, NIP$_n$]
    We say that a formula $\vphi(\bar x;\bar y_1,\dots,\bar y_{n})$ has the \emph{$n$-Independence Property} (in $T$) if there exist $(\bar a_i^1\frown\dots\frown\bar a_{i}^n )_{i\in \Nbb}$ and $(\bar b_I)_{I\subseteq\Nbb^n}$ such that:
    \[
        \vDash \vphi(\bar b_{I};\bar a_{i_1}^1,\dots,\bar a_{i_n}^n) \text{ if, and only if, } (i_1,\dots,i_n))\in I.
    \]
    We say that $T$ is \emph{$n$-dependent} or \emph{NIP$_n$} (No $n$-Independence Property) if no formula has the $n$-independence property in $T$.
\end{definition}

\begin{remark}
    It is easy to observe that NIP$_1$ corresponds precisely to NIP. Moreover, for all $n\in\Nbb$ we have that if $T$ is NIP$_{k}$ then it is NIP$_{k+1}$, and all these implications are strict, witnessed by the random $k$-hypergraph, which is NIP$_{k+1}$, but not NIP$_k$. 
\end{remark}

In the next subsections, we will also recall the definitions of various classes of theories contained in NIP and NIP$_n$.

\subsubsection{Distality and \texorpdfstring{$n$}{n}-Distality}

Distality was introduced by Simon in \cite{SIMON2013}. In a certain sense, the notion of distality is meant to capture the ``purely unstable'' NIP structures. A theory is distal only if it admits no infinite stable quotient. All $o$-minimal structures are known to be distal, but there are a lot more natural examples of distal structures. There is prolific literature on distality and its applications, and many equivalent definitions of distality are used. A concise one that will be used in our analysis is the following:

\begin{definition} \label{Definition: explicit witness of non-distality}
    The theory $T$ is said \emph{distal} if for every indiscernible sequence $(a_{i})_{i\in\mathbb{Q}}$ in $\mathbb{M}$ and every tuple $b\in \mathbb{M}^{\vert b \vert}$ such that~$(a_{i})_{i\in \mathbb{Q}\setminus \{ 0 \}}$ is indiscernible over $b$ we have that $(a_{i})_{i\in \mathbb{Q}}$ is indiscernible over $b$.
\end{definition}

For other equivalent definitions, and more details on distality, we direct the reader to the work of Aschenbrenner, Chernikov, Gehret, and Ziegler \cite{ACGZ22}. 

\begin{example}\label{ex:distal} We give below some basic examples of distal and non-distal theories:
    \begin{itemize}
        \item Total linear orders are distal.
        \item Meet-trees $(T,\leq \wedge)$ are not distal, in general. For instance, the complete theory of a dense meet-tree is not distal: let $r\in T$ and $(a_i)_{i\in \mathbb{Q}}$ are elements such that $a_i \wedge a_j = r$ for every $i\neq j$. Then $(a_i)_i$ is totally indiscernible (over $r$).  
    \end{itemize}
\end{example}

We will also consider a recent generalisation of distality, called \emph{$n$-distality}, introduced by Walker in \cite{Wal23}. First, we introduce some terminology.

Let $\sI$ be an indiscernible sequence $(b_i: i \in I) \subseteq U$ indexed by an infinite linear order $(I, <)$. Suppose $\sI_0 + \cdots + \sI_n$ is a partition of $\sI$ corresponding to a Dedekind partition $I_0+ \cdots +I_n$ of $I$. Let $A$ be a sequence $(a_0,\ldots,a_{n-1}) \subseteq U$. Assume $\lvert b_i \lvert =\lvert a_j\lvert $ for all $i\in I$ and $j<n$. 

We say that $A$ \emph{inserts (indiscernibly) into} $\sI_0+\cdots+\sI_n$ if the sequence remains indiscernible after inserting each $a_i$ at the $i^{\text{th}}$-cut, i.e., the sequence 
    \[
        \sI_0 + a_0 + \sI_1 + a_1 + \cdots + \sI_{n-1}+ a_{n-1}+ \sI_n
    \]
is indiscernible. Moreover, for any $A' \subseteq A$, we say that $A'$ \emph{inserts (indiscernibly) into} $\sI_0+\cdots+\sI_n$ if the sequence remains indiscernible after inserting each $a_i\in A'$ at the $i^{\text{th}}$-cut. For simplicity, we may say that \emph{$A$ (or $A'$) inserts into $\sI$} when the partition of $\sI$ under consideration is clear.

\begin{definition}
    For $m>0$ and an indiscernible sequence $\sI$, we say that the Dedekind partition $\sI_0+\cdots +\sI_{m+1}$ is \emph{$m$-distal} if every sequence $A=(a_0, \ldots, a_{m}) \subseteq U$ which does not insert into $\sI$ contains some $m$-element subsequence which does not insert into $\sI$. A theory is \emph{$m$-distal} if all Dedekind partitions of indiscernible sequences in the monster model are $m$-distal.  
\end{definition}

One can see that a theory is $1$-distal if, and only if, it is distal. Moreover, as shown in \cite{Wal23}, every $n$-distal theory is NIP$_n$.

\subsubsection{Monadic NIP}\label{Subsubsec:Monadic NIP} 
Another differently flavoured strengthening of NIP is the notion of \emph{monadic NIP}. Monadic NIP was introduced by Baldwin and Shelah in \cite{BaldwinShelah1985}, but the study of monadically NIP theories has seen a resurgence in recent years, both from the point of view of pure model theory (see, for instance, \cite{BL21}), and from the point of view of theoretical computer science (see, for instance, \cite{BGOdMSTT2022}).

\begin{definition}[Monadic NIP]
    We say that $T$ is \emph{monadically NIP} if, for every $\sM\vDash T$ and every expansion $\sN$ of $\sM$ by unary predicates, we have that $\sN$ is NIP. 
\end{definition}

It is not immediately clear from the definition above that if $\sM$ is such that every expansion of $\sM$ by unary predicates is NIP, then $\Th(\sM)$ is monadically NIP. This follows from \Cref{thm:characterisations-of-monadic-NIP}, below. Before we state that theorem, we need to introduce some further model-theoretic background.

\begin{definition}[Indiscernible-triviality {\cite[Definition~3.8]{BL21}}]
We say that $T$ has \emph{indiscernible-triviality} for any order-indiscernible sequence $(a_i:i\in \omega)$, and every set $B$ of parameters, if $(a_i:i\in \omega)$ is indiscernible over each $b \in B$ then $(a_i:i\in \omega)$ is indiscernible over $B$.
\end{definition}

The following definition originating from \cite{She14} is a strengthening of NIP:

\begin{definition}[dp-rank, dp-minimality]\label{def:dp-rank}
Fix $n\in\omega$. An \emph{ICT-pattern of depth $n$} consists of a sequence of formulas $(\vphi_i(\bar x,\bar y):i\leq n)$ and an array of parameters $(\bar a_{i}^j:i\in\omega,j\leq n)$ such that for all $\eta:[n]\rightarrow \omega$ we have that:
\[
    \left\{\vphi\left(\bar x,\bar a^j_{\eta(j)}\right):j\leq n\right\}\cup\left\{\lnot\vphi\left(\bar x,\bar a_i^j\right):j\leq n, i\neq\eta(j)\right\}
\]
is consistent. We say that $T$ has \emph{dp-rank $n$} if there is an ICT-pattern of depth $n$, but no ICT pattern of depth $n+1$ (in the monster model of $T$). We say that $T$ is \emph{dp-minimal} if it has dp-rank $1$.
\end{definition}

These notions have been intensively studied see e.g. \cite{CS19}, \cite{KOU13}. The reader will find in the literature various alternative ways of defining dp-minimality. For instance, in \cite{OU_2011}, it is shown that $T$ is dp-minimal if and only if it is \emph{inp-minimal} and NIP. We now can state one of the central theorems of \cite{BL21}, which includes a ``Shelah-style forbidden-configuration'' characterisation of monadic NIP.

\begin{fact}[{\cite[Theorem 4.1]{BL21}}]\label{thm:characterisations-of-monadic-NIP}
Let $T$ be a first-order theory. Then, the following are equivalent:
\begin{enumerate}
    \item $T$ is monadically NIP.
    \item \label{coding-configuration} 
    For all $\sM\vDash T$, we cannot find a  $\lL$-formula $\vphi(\bar{x},\bar{y},z)$, sequences of tuples $(\bar a_i:i\in\omega),(\bar b_j:j\in\omega)$, and a sequence of singletons $(c_{k,l}:k,l\in\omega)$, from $\sM$ such that:
    \[
        \sM\vDash\vphi(\bar a_i,\bar b_j,c_{k,l})\text{ if, and only if, } (i,j) = (k,l),
    \]
    \item \label{it+dp} $T$ is dp-minimal and has indiscernible-triviality.
\end{enumerate}
\end{fact}

Suppose that $\sM$ is a sufficiently saturated structure with IP. Then it is a well-known fact that we can find a formula $\vphi(x,\bar y)$, where $|x| = 1$, that witnesses IP in $\sM$. If we are allowed to use unary predicates (or to add parameters \cite{Simon2021}), then the following result, from \cite[Lemma~$8.1.3$]{BaldwinShelah1985} shows that we can achieve more:

\begin{fact}\label{thm:baldwin-shelah}
        If $\sM$ is a sufficiently saturated structure with IP, then there is a monadic expansion $\sN$ of $\sM$ and a formula $\vphi(x,y)$ with $|x|=|y|=1$ which witnesses that $\sN$ has IP.
\end{fact}

In particular, the following corollary follows immediately from the \namecref{thm:baldwin-shelah} above and the fact that monadic expansions of monadic expansions are monadic expansions:

\begin{corollary}
    If $T$ is not monadically NIP, then there is a monadic expansion of some $\sM\vDash T$ which has IP witnessed by a formula $\vphi(x,y)$ with $|x|=|y|=1$.
\end{corollary}

\begin{example}\label{ExampleMonadicNIP} We conclude this section by revisiting \Cref{ex:distal}, to give some examples that one can keep in mind when discussing monadically NIP structures.
    \begin{itemize}
        \item Linear orders are monadically NIP.
        \item Meet-trees $(T,\leq \wedge)$ are monadically NIP.  This follows from \cite[Proposition~$4.7$]{Sim11}, where it is shown that coloured meet-trees (that is, monadic expansion of meet-trees) are dp-minimal, so in particular NIP. 
    \end{itemize}
\end{example}

\section{Generalised Indiscernibles and Coding Configurations}\label{sec:Generalised Indiscernibles and Coding Configurations}
\subsection{Generalised Indiscernibility}\label{sub:generalised-indiscernibles}
    
\emph{Generalised indiscernible sequences} were introduced by Shelah \cite[Section VII.2]{She90}, as a tool for studying the tree property. These notions have proven an important tool in classification theory and the study of tree properties (see, for instance, \cite{KK2011} and \cite{KKS_2013}). More recently, generalised indiscernibles have been used by Guigona, Hill, and Scow \cite{GHS17} as a means for producing new dividing lines. As the name suggests, these generalise the classical notion of order-indiscernible sequences. We recall the definition:
    
\begin{definition}[Generalised indiscernibles]\label{def:generalised-indiscernibles}
    Let $\lL^\prime$ be a first-order language and $\sI$ an $\lL^\prime$-structure. Given an $\sI$-indexed sequence of tuples $ (\bar a_i)_{i\in\sI}$ from the monster model of $T$, and a small subset $A$ of the monster, we say that $(\bar a_i)_{i}$ is an \emph{$\sI$-indiscernible sequence\footnote{In the literature, what we refer to as an $\sI$-indiscernible sequence is often called an \emph{$\sI$-indexed indiscernible sequence}.} over $A$} if for all positive integer $n$ and all sequences $i_1,\dots,i_n,j_1,\dots,j_n$ from $\sI$ we have that if $\qftp_\sI^{\lL^\prime}(i_1,\dots,i_n) = \qftp_\sI^{\lL^\prime}(i_j,\dots,j_n)$ then $\tp(\bar a_{i_1},\dots,\bar a_{i_n}/A) =  \tp(\bar a_{j_1},\dots,\bar a_{j_n}/A).$ If $A=\emptyset$, we say that $(\bar a_i)_{i}$ is simply an \emph{$\sI$-indiscernible sequence}.
\end{definition} 

Of course, not all structures are created equal. Some structures, as it turns out are more suitable for indexing generalised indiscernible sequences than others. Indeed, a key property of (ordered) sequences is that their EM-type can always be realised by an order-indiscernible sequence. This fact is sometimes referred to as the \emph{``Standard Lemma''} (see \cite[Lemma~5.1.3]{TZ12}). Its proof shows a very tight connection between order-indiscernibles and \emph{Ramsey's theorem}. This connection carries through to generalised indiscernibles, this time with structural Ramsey theory (in the sense of \Cref{sub:standard-structs}). 

 The following definition captures the essence of the standard lemma idea for generalised indiscernibles: 
\begin{definition}[The Modelling Property]\label{def:mp}
    Let $\lL^\prime$ be a first-order language, $\sI$ an $\lL^\prime$-structure, and $ (\bar a_i)_{i\in\sI}$ be an $\sI$-indexed sequence of tuples.  Given an $\sI$-indexed sequence of tuples $(\bar b_i)_{i\in\sI}$ from the monster, we say that $(\bar b_i)_{i}$ is \emph{(locally) based on} $(\bar a_i)_{i}$ if for all finite sets of $\lL$-formulas $\Delta\subseteq\lL$, all $n\in\mathbb{N}$ and all $i_1,\dots,i_n$ from $\sI$ there is some $j_1,\dots,j_n$ from $\sI$ such that $\qftp_\sI^{\lL^\prime}(i_1,\dots,i_n) = \qftp_\sI^{\lL^\prime}(j_1,\dots,j_n)$ and $\tp_\Delta(\bar b_{i_1},\dots,\bar b_{i_{n}}) = \tp_\Delta(\bar a_{j_1},\dots,\bar a_{j_{n}})$.

    We say that $\sI$ has the \emph{modelling property in $T$} if for each $\sI$-indexed sequence $ (\bar a_i)_{i\in\sI}$ of tuples from the monster model of $T$ there exists an $\sI$-indiscernible sequence $ (\bar b_i)_{i\in\sI}$ based on $ (\bar a_i)_{i\in\sI}$. We say that $\sI$ has the \emph{modelling property} if, for all first-order theories $T$, $\sI$ has the modelling property in $T$.

\end{definition}

For a more detailed discussion of these concepts, the reader can refer to the work of the first two authors \cite{MP22}. We recall a  result of that paper which generalises a theorem of Scow, and one of it's consequences:

\begin{fact}[{\cite[Theorem A]{MP22}}]\label{rc mp}
    Let $\lL'$ be a first-order language, $\cC$ a class of finite $\lL'$-structures, and $\sI$ an infinite locally finite $\lL'$-structure such that $\age(\sI) = \cC$. Then, the following are equivalent:
    \begin{enumerate}
        \item $\cC$ has the Ramsey property.
        \item $\sI$ has the modelling property.
    \end{enumerate}
\end{fact}

Observe that if $\sI$ and $\sJ$ are two structures such that $\sJ$ is a quantifier-free reduct of $\sI$, then, automatically, any $\sJ$-indiscernible sequence will be an $\sI$-indiscernible sequence. Of course, the converse need not always hold. For instance, in an arbitrary theory $T$, not every order-indiscernible sequence is totally indiscernible. On the other hand, if $T$ is stable, then every order-indiscernible sequence is totally indiscernible, and this implication in fact characterises stability \cite[Theorem~II.2.13]{She90}. As we will discuss in detail later in the paper, this sort of phenomenon, which is made precise in the following definition, can be used as an alternative definition for several dividing lines.

\begin{definition}\label{def:indiscernible collapsing}
    Let $\sI$ and $\sJ$ be two structures such that $\sJ$ is a reduct of $\sI$. We say that a theory $T$ \emph{collapses indiscernibles from $\sI$ to $\sJ$} if every $\sI$-indiscernible sequence in the monster model is a $\sJ$-indiscernible sequence. We say that $T$ \emph{collapses $\sI$-indiscernibles} if it collapses any $\sI$-indiscernible sequence to $\sJ$-indiscernible sequence, where $\sJ$ is a strict quantifier-free reduct of $\sI$, i.e. $\sI \centernot\reducesqf \sJ$.
\end{definition}

In \cite{GH19}, the following variant of a definition was given for non-collapsing indiscernibles:
\begin{definition}[{non-collapsing indiscernibles (for finite languages) \cite[Definition 3.2]{GH19}}]\label{def:indiscernible collapsing qftp}
     Let $\sI$ be a structure and let $T$ be a theory. A sequence $(a_i)_{i\in \sI}$ of element in a model $\sM$ of $T$ is \emph{non-collapsing indiscernibles} if 
     \[
     \qftp(i_1,\dots, i_n) = \qftp(j_1,\dots, j_n) \iff \tp(a_{i_1},\dots, a_{i_n}) = \tp(a_{j_1},\dots, a_{j_n}),
     \]
     for all $i_1,\dots,i_n, j_1,\dots, j_n\in \sI$.
\end{definition}
The definition above was given under the assumption $\sI$ is an ultrahomogeneous structure in a finite relational language. The following \namecref{uncollapsed finite iff uncollapsed} shows that in this context, the two definitions coincide. \Cref{example: Guingona equivalence fails in infinite language} shows that \Cref{uncollapsed finite iff uncollapsed} fails in general.

\begin{lemma}\label{uncollapsed finite iff uncollapsed}
    Let $\sI$ be an ultrahomogeneous $\aleph_0$-categorical structure in a countable language, let $\sM$ be some structure and $(a_i)_{i\in\sI}\in \sM$ be an $\sI$-indiscernible sequence. Then the following are equivalent:
    \begin{enumerate}
        \item \label{item:uncollapse fin lang}
        $(a_i)_{i\in\sI}$ is non-collapsing, according to \Cref{def:indiscernible collapsing qftp}.
        \item \label{item:uncollapse reduct}
        For all $\sJ\reduces \sI$ such that $(a_i)_{i\in\sI}$ is  $\sJ$-indiscernible, $\sJ\reducesbothqf\sI$.
    \end{enumerate}

    \begin{proof} Let $\lL_{\sI},\lL_{\sM}$ be the respective languages of $\sI,\sM$. Observe that since $\sI$ is $\aleph_0$-categorical, every $\sJ\reduces \sI$ must be $\aleph_0$-categorical too.
        \begin{itemize}
            \item[(\labelcref{item:uncollapse fin lang}$\implies$\labelcref{item:uncollapse reduct})] 
            Let $\sJ\reduces\sI$ such that $(a_i)_{i\in\sI}$ is  $\sJ$-indiscernible. We need to show that, in this case, for every $\lL_{\sI}$-formula $\vphi(x_1,\dots,x_n)$ there exists a quantifier-free $\lL_\sJ$-formula $\psi(x_1,\dots,x_n)$ such that for all $i_1,\dots,i_n\in\sI$ we have that $\sI\vDash\vphi(i_1,\dots,i_n)$ if, and only if, $\sJ\vDash\psi(i_1,\dots,i_n)$. 
            
            By \Cref{item:uncollapse fin lang} and $\sJ$-indiscernibility we have that: 
            \begin{equation}\label{eq: qftp J qftp I}
                \qftp_{\sJ}(i_1,\dots, i_n) = \qftp_{\sJ}(j_1,\dots, j_n) \implies \qftp_{\sI}(i_1,\dots, i_n) = \qftp_{\sI}(j_1,\dots, j_n),
            \end{equation} 
            for all $i_1,\dots,i_n, j_1,\dots, j_n\in \sI$. Given $\vphi(x_1,\dots,x_n)$, by Ryll-Nardzewski, there are finitely many complete $n$-types $p_1,\dots, p_k\in S_{n}^{\sI}(\emptyset)$ such that $\sI\models\vphi(\bar{x})\leftrightarrow \bigvee_{i=1}^k p_i(\bar{x})$. By quantifier elimination in $\sI$, we may assume $p_1,\dots, p_k$ are quantifier-free complete types. 
            
            By \Cref{eq: qftp J qftp I} and by $\aleph_0$-categoricity of $\sJ$, for each $p_i$, there are complete quantifier-free types $q_{i,1},\dots, q_{i,l_i}$ in $\sJ$ such that 
            $\sJ\models \bigvee_{j=1}^{l_i} q_{i,j}(i_1,\dots, i_n)\iff \sI\models p_i(i_1,\dots, i_n)$ for all $i_1,\dots, i_n\in \sI$. In conclusion, there are finitely many quantifier-free types
            $q_1,\dots, q_m$ such that 
            \begin{equation}\label{eq:qftp disjunction J iff vphi I}
                \sJ\models \bigvee_{j=1}^m q_j(i_1,\dots, i_n)\iff \sI\models \vphi(i_1,\dots, i_n)
            \end{equation}
             for all $i_1,\dots, i_n\in \sI$.
            Now, since $\sJ$ is $\aleph_0$-categorical, we may replace each $q_j$ in \Cref{eq:qftp disjunction J iff vphi I} with an isolating formula for it, so $\vphi$ is $\emptyset$-quantifier-free-definable in $\sJ$, and the result follows.
            
            \item[(\labelcref{item:uncollapse reduct}$\implies$\labelcref{item:uncollapse fin lang})] 
            We need to show that given any $\sI$-indiscernible sequence $(a_i)_{i\in\sI}$ we have that:
            \[
                \qftp(i_1,\dots, i_n) = \qftp(j_1,\dots, j_n) \iff \tp(a_{i_1},\dots, a_{i_n}) = \tp(a_{j_1},\dots, a_{j_n}),
             \]
            for all $i_1,\dots,i_n, j_1,\dots, j_n\in \sI$.
            
            To this end, we will start by constructing a reduct $\sJ\reduces\sI$. Let $\lL_{\sJ}$ be the language consisting of an $n$-ary relation symbol $R_{\vphi}$ for every $\lL_{\sM}$-formula $\vphi$ in $n$ free variables.
            Let $\sJ$ be the structure with the same universe as $\sI$ such that for every $\lL_{\sM}$-formula $\vphi$ and every $i_1,\dots, i_n\in \sI$ we set $\sJ\models R_{\vphi}(i_1,\dots, i_n)\iff \sM\models\vphi(a_{i_1},\dots, a_{i_n})$. By construction, $(a_i)_{i\in \sI}$ is $\sJ$-indiscernible. 
            
            Furthermore, as in the previous implication, by $\sI$-indiscernibility of $(a_i)_{i\in \sI}$ and by $\aleph_0$-categoricity of $\sI$, for every $\lL_{\sM}$-formula $\vphi$, there are $\lL_{\sI}$-formulas $\psi_{\vphi,1},\dots, \psi_{p,k(\vphi)}$ such that:
            \[
            \sI\models \bigvee_{j=1}^{k(\vphi)} \psi_{\vphi,j}(i_1,\dots, i_n)\iff \sJ\models R_{\vphi}(i_1,\dots, i_n).
            \]
            So $\sJ\reduces\sI$. Therefore, by \Cref{item:uncollapse reduct}, $\sJ\reducesbothqf\sI$. Finally, let $i_1,\dots, i_n,j_1,\dots, j_n\in \sI$ be such that $\tp(a_{i_1},\dots, a_{i_n}) = \tp(a_{j_1},\dots, a_{j_n})$. 
            Then, by definition of $\sJ$, $\qftp_{\sJ}({i_1},\dots, {i_n}) = \qftp_{\sJ}({j_1},\dots, {j_n})$.  Since $\sI\reducesqf\sJ$, this implies that $\tp_{\sI}({i_1},\dots, {i_n}) = \tp_{\sI}({j_1},\dots, {j_n})$, and hence $\qftp_{\sI}({i_1},\dots, {i_n}) = \qftp_{\sI}({j_1},\dots, {j_n})$, as required.
        \end{itemize}
    \end{proof}
\end{lemma}

The following example shows that, in general, a collapsing $\sI$-indiscernible sequence may collapse only to a $\sJ$-indiscernible sequence where $\sJ$ is not a strict reduct of $I$, but a quantifier-free reduct. 

\begin{example}
    Let $\str{BG}^E=(G,R,E)$ be the countable random bipartite graph $\str{BG}=(G,R)$ equipped with a two-class equivalence relation for the partition. We may embed $\str{BG}$ (as a graph) into the countable random graph $\sG$. Then the sequence $(g)_{g\in \str{BG}^E}$ is not uncollapsing, for it is a $\str{BG}$-indiscernible sequence. One can see, by quantifier elimination in the random graph, that the sequence doesn't collapse further, that is $(g)_{g\in \str{BG}}$ is an uncollapsing $\str{BG}$-indiscernible sequence.
\end{example}

The following example shows that, in \Cref{uncollapsed finite iff uncollapsed}, one needs indeed to assume that the language is countable:
\begin{example}\label{example: Guingona equivalence fails in infinite language}
    Let $\sN$ be the structure on $\bN$ in the \emph{full set-theoretic language} $\lL_F$, i.e., the language which for every $A\subseteq \bN^{n}$, contains a relation symbol $R_A\in \Sig(\lL_F)$, naturally interpreted as $R_A(\sN) = A$. Notice that $\sN$ is an ultrahomogeneous $\aleph_0$-categorical structure. Let $\sN_S$ be the reduct of $\sN$ to the language $\lL_S\subset\lL_F$ consisting only of unary predicates for the singleton sets, i.e., $\lL_S:=\{R_{\{a\}} : a\in \bN \}$. Then $\sN$ codes any countable structure $\sM$, and it is in particular unstable.  
    On the other hand, $\sN_S$ is strongly minimal, and in particular, stable. So 
    $\sN_S\reduces\sN$ but $\sN\centernot\reduces\sN_S$ .
    \begin{enumerate}
        \item letting $a_i:=i$, $(a_i)_{i\in \sN}$ is a non-collapsing $\sN$-indiscernible sequence in $\mathcal{N}$, in the sense of \Cref{def:indiscernible collapsing qftp}.
        \item $\sN$ collapses $\sN$-indiscernibles to $\sN_S$ indiscernibles according to \Cref{def:indiscernible collapsing}. (In fact, every $\sN$-indiscernible sequence in any model of any theory is also an $\sN_S$-indiscernible sequence.) 
    \end{enumerate}
\end{example}

An analogue of Shelah's theorem characterising stable theories as those that collapse order-indiscernible sequences to indiscernible sets was proved by Scow for NIP in \cite{Scow2012} and her result was generalised in \cite{CPT19} to NIP$_n$, for $n>1$. 

\begin{fact}[{\cite[Theorem~5.4]{CPT19}}]\label{thm:nip-k-collapse}
    Let $\cla{OH}_{n+1}$ be the ordered random $(n+1)$-hypergraph. Then, the following are equivalent for a first-order theory $T$:
    \begin{enumerate}
        \item $T$ is NIP$_n$.
        \item $T$ collapses $\cla{OH}_{n+1}$-indiscernibles into order-indiscernibles.
    \end{enumerate}
\end{fact}

\subsection{\texorpdfstring{$\cK$}{K}-Configurations}\label{sec:K-configurations}

We will develop the study of transfer principles with respect to the notion of \emph{$\cK$-configurations}. Forbidding $\cK$-configurations offers a unified way of describing tameness conditions that arise from coding combinatorial configurations, such as the \emph{Order Property} and the \emph{Independence Property}, and provides an interesting scheme of dividing lines. The notion of $\cK$-configurations originates from the work of Guingona and Hill \cite{GH19}, and has been developed further in \cite{GPS21}. In this section, we will fix a relational first-order language $\lL_0$, and an arbitrary first-order language $\lL$. Throughout, we will use $\cK$ to denote a class of finite $\lL_0$-structures closed under isomorphism.

\begin{definition}[$\cK$-configuration, {\cite[Definition 5.1]{GPS21}}]\label{def: K-config}
    An $\lL$-structure $\sM$ \emph{admits (or codes) a $\cK$-~configuration} if there are an integer $n$, a function $I:\Sig(\lL_0) \rightarrow \lL$ and a sequence of functions $(f_A: A\in \cK)$ such that: 
    \begin{enumerate}
        \item for all $A\in \cK$, $f_A:A \rightarrow \sM^n$,
        \item for all $R\in \Sig(\lL_0)$, for all $A\in \cK$, for all $a\in A^{\mathsf{arity}(R)}$,
            \[
                A \models R(a) \Leftrightarrow \sM \models I(R)(f_A(a)).
            \]    
    \end{enumerate}
        
    A theory $T$ admits a $\cK$-configuration if some model $\sM\models T$ admits a $\cK$-configuration. We denote by $\codingcla{\cK}$ the class of theories which  admit a $\cK$-configuration, and by $\ncodingcla{\cK}$ the class of theories which \emph{do not} admit a $\cK$-configuration.
\end{definition}

\begin{example}
    Recall that $\cla{LO}$ and $\cla{CO}$ denote respectively the class of finite linear orders and finite cyclic orders, and $\cla{COD}$ denotes the class of finite cyclically ordered $D$-relation. One can observe that any structure admits an $\cla{CO}$-configuration if, and only if, it admits $\cla{LO}$-configuration. In particular $\ncodingcla{\cla{CO}}=\ncodingcla{\cla{LO}}$. This is of course due to the fact that, after naming one constant,  a dense cyclic order is inter-definable with a standard linear order. Similarly, any structure admits an $\cla{COD}$-configuration if and only if if admits $\cla{OC}$-configuration. Therefore $\ncodingcla{\cla{COD}}=\ncodingcla{\cla{OC}}$.    
\end{example}

\begin{proposition}\label{age coding iff structure coding}
    Let $\cK$ be a class of structures. Then 
        $\codingcla{\cK} = \codingcla{\age(\cK)} = \codingcla{\HC(\cK)}$. In particular, 
        $\codingcla{\sM}=\codingcla{\age(\sM)}$ for a structure $\sM$.
    \begin{proof} 
    Since $\age(\cK)=\age(\HC(\cK))$, it suffices to show the first equality (as we may replace $\cK$ with $\HC(\cK)$ to obtain the second). 

    For the inclusion $\codingcla{\cK}\subseteq \codingcla{\age(\cK)}$, suppose that $T \in \codingcla{\{\sM\}}$. Fix $I$ and $(f_{\sM})_{\sM\in \cK}$ as promised. For every $A\in \age(\cK)$ fix some $\sM_A\in \cK$ and embedding $e_{A}:A\hookrightarrow\sM_A$; then letting $f_{A}:=f_{\sM}\circ e_{A}$, we get an $\age({\sM})$-configuration. 
    
    For the inclusion $\codingcla{\age(\cK)}\subseteq\codingcla{\cK}$, suppose that $T\in \codingcla{\age(\cK)}$. Fix $I$ and $(f_{A} : A\in \age(\cK))$ as promised; then, for all $\sM\in \cK$, the type in $\vert \sM \vert$-variables
    \begin{align*}
        T &\cup \{\ \ I(R)(x_{a_1},\dots, x_{a_n}): R\in \lL(\cK),\  \sM\models\ \  R(a_1,\dots, a_n) \} \\
        & \cup \{\neg I(R)(x_{a_1},\dots, x_{a_n}) : R\in \lL(\cK),\  \sM\models\neg R(a_1,\dots, a_n) \} 
    \end{align*}is finitely satisfiable, and therefore satisfiable by compactness.
    \end{proof}
\end{proposition}

In light of \Cref{age coding iff structure coding} above, given a structure $\sM$, we overload notation, writing $\codingcla{\sM}$ for $\codingcla{\age(\sM)}=\codingcla{\{\sM\}}$, and $\ncodingcla{\sM}$ for $\ncodingcla{\age(\sM)}=\ncodingcla{\{\sM\}}$.

\begin{remark}\label{rmk:coding-qf-definitions}
    Let $\sM$ be an $\lL$-structure and assume that $\sN$ is an $\lL'$-structure which admits an $\sM$-configuration, witnessed by $I,f$. Given a quantifier-free $\lL$-formula $\vphi(\bar x) = \bigwedge_{i\in I}\bigvee_{j\in J}\psi_{i,j}(\bar x)$, where each $\psi_{i,j}(\bar x)$ is an atomic or negated atomic formula. In particular, for each $\psi_{i,j}(\bar x)$, we can define $I(\psi_{i,j})$ to be $I(\psi_{i,j})$ if $\psi_{i,j}$ is a relation symbol, and $\lnot I(\psi_{i,j})$ if $\psi_{i,j}$ is a negated relation symbol.  Then, by definition:
    \[
        \sM\vDash\vphi(a) \text{ if, and only if, } \sN\vDash \bigwedge_{i\in I}\bigvee_{j\in J} I(\psi_{i,j})(f(a)),
    \]
    for all $a\in \sM$.
\end{remark}

A somewhat different approach to coding $\cK$-configurations is the focus of \cite{Wal21}. There, the relevant notion is called \emph{trace definability} and it is seen as a weak form of interpretability. We will not be discussing trace definability in this paper. We will only take note that, from the point of view of trace definability, the following lemma is immediate. We include a proof here for coding $\cK$-configurations, for completeness.

\begin{proposition}\label{ncoding is closed under interpretation}
    Let $\sN$ be an $\lL_\sN$-structure. If an $\lL_\sM$-structure $\sM$ is interpretable in $\sN$ and $\sM$ admits a $\cK$-configuration, then $\sN$ also admits a $\cK$-configuration. In particular, the class $\ncodingcla{\cK}$ is closed under bi-interpretability. 
\end{proposition}
\begin{proof}
    Assume that $\sN$ interprets $\sM$. There is a definable set $D\subseteq \sN^m$, a definable equivalence relation $\sim$ on $D$ and a bijection $\sM \simeq D / \sim$ which interprets $\sM$. This induces a map $*: \lL_\sM \rightarrow \lL_\sN$ which associates to any $\lL_\sN$-formula $\vphi(x)$ an $\lL_\sM$-formula $\vphi^*(x^*)$ interpreting it (where $\vert x^* \vert= m\times \vert x\vert$). Let $s: \sM \rightarrow \sN^m$ a section of the natural projection $\mathcal{D}\rightarrow \mathcal{D}/\sim$ (where we identify $\sM$ with $D/\sim$). 
        
    If $\sM$ admits a $\cK$-configuration via $\left(I:\Sig(\lL_0) \rightarrow \lL_\sM,(f_A: A \rightarrow \sM^n : {A\in \cK})\right)$, then $\sN$ also admits a $\cK$-configuration, via $\left( * \circ I:\Sig(\lL_0) \rightarrow \lL_\sN ,(s\circ f_A:A \rightarrow \sM^{nm} :{A\in \cK})\right)$. Indeed, by unravelling the definitions, we have that:
    \[
        A \models R(a) \Leftrightarrow \sM \models I(R)(f_A(a))\Leftrightarrow  \sN \models  I(R)^*(s(f_A(a))),
    \]
    for all $R\in \Sig(\lL_0)$ for all $A\in \cK$ for all $a\in A^{\mathsf{arity}(R)}$.
\end{proof}
    
We immediately deduce that the class of dp-minimal structures, and in particular, the class of monadically NIP structures cannot be of the form $\ncodingcla{\cK}$, for any $\cK$, as neither is closed under bi-interpretatability. For instance, an infinite set $\sS$ in the language of pure equality is bi-interpretable with the full product $\sS\boxtimes \sS$ which has of dp-rank $2$. Similarly, it is well-known that distality is not preserved under taking reducts. Thus the class of distal structures is not of the form $\ncodingcla{\cK}$ for any $\cK$.

A version of the following \namecref{order on classes} appears in \cite[Proposition 5.4(4)]{GP23}, under the assumption the classes are strong amalgamation classes, and $\sM$ is a Fra\"iss\'e limit of such a class. For our purposes, we omit this requirement and introduce an elementary proof.
\begin{lemma}\label{order on classes}
    Let $\sM$ be a structure with quantifier elimination and $\cK$ a class of structures. Then the following are equivalent:
    \begin{enumerate}
        \item $\sM\in \codingcla{\cK}$.
        \item $\codingcla{\sM}\subseteq \codingcla{\cK}$.
    \end{enumerate} 
\end{lemma}

\begin{proof} 
    For $(1\Rightarrow2)$ suppose that $\sM\in\codingcla{\cK}$ and $\sN\in\codingcla{\sM}$. Let $I,f$ witness that $\sM$ admits a $\cK$-configuration, and $J,g$ witness that $\sN$ admits an $\sM$-configuration. By quantifier elimination, we may assume that $I(R)$ is quantifier-free for all $R\in\Sig(\cK)$. It is clear, using \Cref{rmk:coding-qf-definitions}, that $g\circ f$ and $J\circ I$ witness that $\sN$ admits a $\cK$-configuration. For $(2\Rightarrow1)$, observe that clearly $\sM\in \codingcla{\sM}$ and therefore, by assumption $\sM\in \codingcla{\cK}$.
\end{proof}

\Cref{fact: equality of closedness} below was originally stated for classes of \emph{finite} structures. However, notice that:
\begin{itemize}
    \item $\age(\cK_1\fullprod\cK_2) = \age(\cK_1)\fullprod\age(\cK_2)$.
    \item  $\age(\cK_1[\cK_2]) = \age(\cK_1)[\age(\cK_2)]$.
    \item $\age(\cK_1\sqcup\cK_2) = \age(\cK_1)\sqcup\age(\cK_2)$.
    \item If $\age(\cK_1)$ and $\age(\cK_2)$ contain structures of size $n$ for every $n\in \bN$, then \\
    $\age(\cK_1)\suppos\age(\cK_2) = \age(\cK_1\suppos\cK_2)$.
\end{itemize}

Therefore, by \Cref{age coding iff structure coding}, we may phrase \Cref{fact: equality of closedness} in the fullest generality:

\begin{fact}[{\cite[Theorem 4.27]{GPS21}}]\label{fact: equality of closedness}
    Let $\cK_1,\cK_2$ be classes of structures in respective finite disjoint relational language, $\lL_1,\lL_2$. Then
    \[
        \codingcla{\cK_1\fullprod\cK_2} = \codingcla{\cK_1[\cK_2]} = \codingcla{\cK_1\disjointunion\cK_2} = \codingcla{\cK_1}\cap \codingcla{\cK_2}.
    \]
    
    If, additionaly, $\age(\cK_1)$ and $\age(\cK_2)$ both contain structures of size $n$ for every $n\in \bN$, then all of the above are also equal to $\codingcla{\cK_1\suppos K_2}$.
\end{fact}

The \namecref{fact: equality of closedness} was originally phrased under a slightly weaker assumption that $\cK_1$ and $\cK_2$ both containing structures of size $n$, for every $n\in \bN$. We note that this requirement is only needed when considering free superpositions: for the full product and lexicographic product it is precisely \cite[Corollary 5.16]{GPS21} and \cite[Corollary 5.12]{GPS21}, respectively. The \namecref{fact: equality of closedness} above also did not include the disjoint union, however, this equality is almost trivial.

The following \namecref{coding reducts is coding join} can be seen as a generalisation of \Cref{fact: equality of closedness}.

\begin{theorem}\label{coding reducts is coding join}
    Let $\sN$ be a structure and let $\sN_1,\sN_2$ be {quantifier-free} reducts of $\sN$ such that $\sN = \sN_1\vee\sN_2$ in the lattice of quantifier-free reducts of $\sN$, up to interdefinability (i.e. $\aut(\sN) = \aut(\sN_1)\cap\aut(\sN_2)$). Then $\codingcla{\sN_1}\cap \codingcla{\sN_2} = \codingcla{\sN}$.
\end{theorem}
\begin{proof}
    The inclusion $\codingcla{\sN}\subseteq \codingcla{\sN_1}\cap \codingcla{\sN_2}$ follows almost immediately, from the fact that $\sN_i\reducesqf\sN$, for $i\in[2]$, and does not require the assumption that $\sN = \sN_1\vee\sN_2$. Suppose that $T\in\codingcla{\sN}$ and let $\sM\vDash T$ be a model admiting an $\sN$-configuration, witnessed by $I,f$. Since $\sN_1\reducesqf\sN$ for each $R\in\lL_{\sN_1}$, there is a quantifier-free $\lL_{\sN}$-formula $\vphi_R$ such that $\sN\vDash\vphi_R(a)$ if, and only if, $\sN_1\vDash R(a)$, for all $a\in R$. Let $I(\vphi_R)$ be as in \Cref{rmk:coding-qf-definitions} and define 
    \[
    \begin{aligned}
        J:\Sig(\lL_{\sN_1}) &\to        \lL_{\sM}\\
                R           &\mapsto     I(\vphi_R).
    \end{aligned}
    \]
    It is clear that $f,J$ witness that $\sM\in\codingcla{\sN_1}$. Similarly, $\sM\in\codingcla{\sN_2}$, and the result follows.

    For the other inclusion, observe that since $\sN = \sN_1\vee\sN_2$, there is a structure $\sN'$ which is interdefinable with $\sN$ such that $\lL_{\sN'} = \lL_{\sN_1}\cup\lL_{\sN_2}$. By \Cref{ncoding is closed under interpretation} we may assume that $\sN = \sN'$.
     Let $\sM_i\models T$ and let $I_i, f_i$ be a coding of $\sN_i$ in $\sM_i$, for $i\in\{1,2\}$. By taking a common elementary extension, we may assume $\sM_1=\sM_2=\sM$. Then $f_i:N_i\to \sM^{n_i}$ for some $n_1,n_2\in \bN$. Let $f:N\to \sM^{n_1+n_2}$ be defined as $f_1\frown f_2$, i.e., $f(a):=f_1(a)\frown f_2(a)$ for all $a\in \mathcal{N}$. Let $\pi_i$ be the obvious projection of $\sM^{n_1+n_2}$ onto $\sM^{n_i}$, for $i\in \{1,2\}$. Then, for every $i\in \{1,2\}$ and every $R\in \lL_i$, let $I(R):= I_i\circ \pi_i$. 
     By construction, $I,f$ is an $\sN$-configuration in $\sM$.
\end{proof}

\Cref{coding reducts is coding join} can be seen as a generalisation of \Cref{fact: equality of closedness} in the context of the following observation: 
\begin{observation}
Let $\cK$ is a class of structures, and let $\cla{E}$ be the class of all finite sets. Then: 
    \[
    \codingcla{\cK} = \codingcla{\cK\fullprod\cla{E}} = \codingcla{\cK[\cla{E}]} = \codingcla{\cla{E}[\cK]} = \codingcla{\cK\sqcup \cla{E}}.
    \]
    If, additionally, $\cK$ contains only finite structures, and it contains structures of size $n$ for every $n\in \bN$, then the above are also equal to $\codingcla{\cK\suppos\cla{E}}$. We can then deduce all the equalities in \Cref{fact: equality of closedness} using that, up to $\reducesbothqf$, we have in the lattice of quantifier-free reducts we have the following:
 \begin{itemize}
     \item $\sM[\cla{E}] \vee \cla{E}[\sN]= \sM[\sN]$,
     \item $\sM \boxtimes \cla{E} \vee \cla{E}\boxtimes \sN= \sM\boxtimes \sN$,
     \item $\sM \sqcup \cla{E} \vee \cla{E}\sqcup \sN= \sM\sqcup \sN$,
     \item  $\sM \suppos \cla{E} \vee \cla{E}\suppos \sN= \sM\suppos \sN$.
 \end{itemize}   
\end{observation}

\section{Structural Ramsey theory through collapsing indiscernibles}\label{sec: Collapsing indiscernible as a witness for the Ramsey property}

The main result of this \namecref{sec: Collapsing indiscernible as a witness for the Ramsey property} is \Cref{reduct of higher arity is not Ramsey}, which roughly says that (ages of) ``higher-arity'' reducts of an $\aleph_0$-categorical ultrahomogeneous structure cannot have the Ramsey property. The proof of the theorem is not combinatorial but uses exclusively model-theoretic tools, notably the notion of collapsing indiscernibles. 

Let us start by recalling that, under some mild hypotheses, the notions of $\ncodingcla{\cK}$ and collapsing generalised indiscernibles are closely related. More precisely: 

\begin{fact}[{\cite[Theorem 3.14]{GH19}}]\label{fact: noncoding iff collapsing Ramsey}
    Let $\sI$ be the Fraïssé limit of a Ramsey class with the strong amalgamation property, in a finite relational language. Then the following are equivalent for a theory $T$:
    \begin{enumerate}
        \item\label{item:T noncoding obsolete} $T\in \ncodingcla{\sI}$.
        \item\label{item:T collapses obsolete} $T$ collapses $\sI$-indiscernibles.
    \end{enumerate}
\end{fact}

We will use this second characterisation later, in order to prove our transfer principles. For instance, the following is immediate from \Cref{thm:nip-k-collapse} combined with \Cref{fact: noncoding iff collapsing Ramsey}:

\begin{fact}\label{fact:nip-n-nc-k}
    Let $\cla{H}_{n+1}$ the class of finite $(n+1)$-hypergraphs. Then the following are equivalent for a first-order theory $T$:
        \begin{enumerate}
            \item $T$ is NIP$_n$.
            \item $T$ is $\ncodingcla{\cla{H}_{n+1}}$.
        \end{enumerate}
\end{fact}

The reader will find a similar statement in \cite[Proposition 5.2]{CPT19}. 

In the context of this \namecref{sec: Collapsing indiscernible as a witness for the Ramsey property}, our approach is to use \Cref{fact: noncoding iff collapsing Ramsey} as a sort of criterion to witness the Ramsey property in a given class of structures. More precisely, if given an ultrahomogeneous Fraïssé limit $\sI$, if there is a theory $T$ which codes $\sI$ but collapses $\sI$-indiscernibles, then we must conclude that $\sI$ is not Ramsey.

Notice that the underlying assumption in \cite{GH19} is that all classes are in a \emph{finite relational language}. We give a more general version of \Cref{fact: noncoding iff collapsing Ramsey}, as we don't require that the language is finite and that the age of the structure $\sI$ has the \emph{strong} amalgamation property.

\begin{theorem}\label{thm:Ramsey class collapse}
    Let $\sI$ be an $\aleph_0$-categorical Fraïssé limit in a countable language. Then, for any theory $T$, $T\in \ncodingcla{\sI}$ implies that $T$ collapses $\sI$-indiscernibles.
    
    Moreover, if $\age(\sI)$ is Ramsey, these are equivalent.
\end{theorem}
 The theorem may fail when uncountable languages are involved. (See \Cref{example-Guingona equivalence fails in infinite language 2} below.) 

\begin{proof}
 Let $\lL_T, \lL_\sI$ be the respective languages of $T, \sI$.

For the first part of the theorem, we argue via contraposition. Suppose that $T$ does not collapse $\sI$-indiscernibles. Then, by definition there is some $\sM\vDash T$ and a non-collapsing $\sI$-indiscernible $(\bar a_i:i\in \sI)$, in $\sM$. Without loss of generality, we may assume that $\sM$ is Morleyised. We wish to show that $T\in\codingcla{\Ical}$, so we must find some $m\in\Nbb$, a function $J:\Sig(\Lcal_{\Ical})\to \Lcal$ and a function $f:I\to M^n$ such that:
    
        \[
            \sI \vDash R(a) \iff \sM\vDash J(R)(f(a)),
        \]
    
        for all $R\in\mathsf{Sig}(\lL_\sI)$ and all $a\in I^{\mathsf{arity}(R)}$.
        
        The idea is to code the configuration along the non-collapsing sequence $(\bar a_i:i\in\sI)$, that is, to use the function $J:i\mapsto \bar a_i$, for $i\in I$ (so $m=|\bar a_i|)$.
        
        To this end, let $R\in \lL_{\sI}$ be an $n$-ary relation symbol. Since $\sI$ is $\aleph_0$-categorical and finitely homogeneous it has quantifier-elimination, so there are finitely many complete quantifier-free types $p_1,\dots, p_k\in S_{n}^{\sI}(\emptyset)$, such that: 
    
        \begin{equation}\label{eq:isolating-types}
            \sI\vDash R(i_1,\dots, i_n)\iff \sI\vDash \bigvee_{j=1}^k p_j(i_1,\dots, i_n)
        \end{equation}
    
        for all $i_1,\dots, i_n\in \sI$.
        
        At this point, by \Cref{uncollapsed finite iff uncollapsed}, we get that $(\bar a_i:i\in \sI)$ is non-collapsing, in the sense of \Cref{def:indiscernible collapsing qftp}. Explicitly, we have that:
        
        \begin{equation}\label{eq:non-collapsing}
            \begin{aligned}
                \qftp(i_1,\dots, i_n)=\qftp&(j_1,\dots, j_n)\iff \\ &\tp(\bar a_{i_1},\dots, \bar a_{i_n})=\tp(\bar a_{j_1},\dots, \bar a_{j_n})   
            \end{aligned}
        \end{equation}
    
        for all $i_1,\dots, i_n, j_1,\dots, j_n\in \sI$. 
        
        The point is that for each quantifier-free type $p_j$ in \labelcref{eq:isolating-types} there is a complete type $q_j\in S_n^{\sM}(\emptyset)$ such that:
        
        \begin{equation}\label{eq:isolating-types-replaced}
            \sI \vDash R(i_1,\dots, i_n)\iff (\bar a_{i_1},\dots, \bar a_{i_n}) \vDash \bigvee_{j=1}^k q_j(\bar x_1,\dots,\bar x_n),
        \end{equation}
    
        for all $i_1,\dots, i_n\in \sI$.
    
        Let $\sM_\Ifrak$ be the structure $\sM$ expanded by an $m$-ary relation symbol $R$, where $m=|\bar a_i|$, naming precisely the tuples in the sequence $(\bar a_i:i\in\sI)$, and let $\lL_\Ifrak := \Lcal\cup\{R\}$. We start with the following claim:
        
        \begin{claim}\label{claim:naming-an-ind}
            The sequence $(\bar a_i:i\in\sI)$ remains indiscernible in the structure $\sM_\Ifrak$. In particular, it is a non-collapsing indiscernible.
        \end{claim}
    
        \begin{claimproof}
            Suppose not. Then there are $i_1,\dots,i_n,j_1,\dots,j_n\in\sI$ and an $\lL_\Ifrak$-formula $\phi(\bar x_1,\dots,\bar x_n)$ such that:
            \[
                \qftp(i_1,\dots, i_n)=\qftp(j_1,\dots, j_n)
            \]
            and
            \[
                \sM_\Ifrak\vDash\phi(\bar a_{i_1},\dots, \bar a_{i_n})\land\lnot\phi(\bar a_{j_1},\dots, \bar a_{j_n}). 
            \]
            We prove that this is impossible by induction on the complexity of $\phi(\bar x_1,\dots,\bar x_n)$. It is obviously impossible for quantifier-free formulas, and, in fact, it suffices to check that it is impossible when $\phi(\bar x_1,\dots,\bar x_n)$ is of the form:
            \[
                \exists \bar y(\bar y\in R\land \psi(\bar x_1,\dots,\bar x_n,y)).
            \]
            Suppose that there is such a formula and $i_1,\dots,i_n,j_1,\dots,j_n\in\Ical$ such that:
            \[
                \qftp(i_1,\dots, i_n)=\qftp(j_1,\dots, j_n)
            \]
            and:
            \[
                \sM_\Ifrak\vDash \exists \bar y(\bar y\in R\land \psi(\bar a_{i_1},\dots, \bar a_{i_n},y))\land\lnot\exists \bar y(\bar y\in R\land \psi(\bar a_{j_1},\dots,\bar  a_{j_n},y)) 
            \]
            In particular, there is some $k\in\Ical$ such that:
            \[
                \sM_\Ifrak\vDash \psi(\bar a_{i_1},\dots, \bar a_{i_n},\bar a_k).
            \]
            Since $\sI$ is ultrahomogeneous there is an automorphism $\sigma\in\Aut(\sI)$ sending $i_l$ to $j_l$ for all $l\in\{1,\dots,n\}$. Then:
            \[
                \qftp(i_1,\dots,i_n,k) = \qftp(j_1,\dots,j_n,\sigma(k)),
            \]
            and, by induction, we have that:
            \[
                \sM_\Ifrak\vDash\psi(\bar a_{j_1},\dots, \bar a_{j_n},\bar a_{\sigma(k)})), 
            \]
            contradicting the fact that:
            \[
                \sM_\Ifrak\vDash \lnot\exists \bar y(\bar y\in R\land \psi(a_{j_1},\dots, a_{j_n},y)).
            \]
            The claim then follows.
        \end{claimproof}
        
        In particular, by the claim above, \Cref{eq:non-collapsing,eq:isolating-types-replaced} hold in $\sM_\Ifrak$.
        
        Now, let $\Ifrak$ be the structure induced by $\sM_\Ifrak$ on the set $\{\bar a_i:i\in \Ical\}$. It is easy to check that $\Ifrak$ has quantifier-elimination.
        
        Since $\sI$ is $\aleph_0$-categorical, it has finitely many complete (quantifier-free) $n$-types, for all $n\in\Nbb$, and thus by \labelcref{eq:non-collapsing}, and since $\Ifrak$ has quantifier-elimination it follows that types in $\sM_\Ifrak$ determine (quantifier-free) types in $\Ifrak$. So $\Ifrak$ has finitely many complete $n$-types for all $n\in\Nbb$, and thus $\Ifrak$ is also $\aleph_0$-categorical. Therefore, by replacing the types $q_j$ in \labelcref{eq:isolating-types-replaced} by complete types in the induced structure $\Ifrak$, we can assume that each $q_j$ is isolated (along the sequence), by some $\Lcal$-formula in $\psi_{q_j}$. Replacing each $q_j$ with $\psi_{q_j}$ in \labelcref{eq:isolating-types-replaced} we get that:
        \[
            \sI\vDash R(i_1,\dots, i_n)\iff \Ifrak\vDash \bigvee_{j=1}^k q_j(a_{i_1},\dots, a_{i_n}),
        \]
        for all $i_1,\dots, i_n\in \Ical$.

        Since the domain of $\Ifrak$ is precisely the range of the function $f:\sI\to \sM^{|\bar a_i|}$, it follows that:
        \[
            \sI\vDash R(i_1,\dots, i_n)\iff \sM\vDash \bigvee_{j=1}^k q_j(f(i_1),\dots, f(i_n)),
        \]
        Thus, setting $J(R):= \bigvee_{j=1}^k \psi_{q_j}$, gives us that $\sM\in\codingcla{\sI}$, as required.\footnote{Observe that the Ramsey assumption was not used in this implication. It will, however, be used in the converse.}
    
    Now, for the ``moreover'' part, assume in addition that $\sI$ is Ramsey. Let $T\in\codingcla{\sI}$, let $\sM\vDash T$, and let $(a_i)_{i\in \sI}$ be an $\sI$-indexed sequence from $\sM$. Since $T\in\codingcla{\sI}$, let $I:\Sig(\lL_{\sI})\to \lL_T$, as in \Cref{def: K-config}. In particular, we have that
        \[
        \tp_{\lL_{T}}(a_{i_1},\dots, a_{i_n})=\tp_{\lL_{T}}(a_{j_1},\dots, a_{j_n})\implies \qftp_{\lL_{\sI}}(i_1,\dots, i_n)=\qftp_{\lL_{\sI}}(j_1,\dots, j_n),
        \]
        for all $i_1,\dots, i_n, j_1,\dots, j_n\in \sI$. 
        Notice that this implication is expressible in the language $\lL_T$ and is in the EM-type of the sequence $(a_i)_I$.
        Since $\sI$ is Ramsey, we can find, in some elementary extension $\sM'\succ\sM$, an $\sI$-indiscernible sequence $(b_i)_{i\in \sI}\in \sM'\succ \sM$ based on $(a_i)_{i\in \sI}$, by \Cref{rc mp}. This new sequence also satisfies the previous implication. By definition of $\sI$-indiscernibility, we also have that 
        \[
        \qftp_{\lL_{\sI}}(i_1,\dots, i_n)=\qftp_{\lL_{\sI}}(j_1,\dots, j_n) \implies
        \tp_{\lL_{T}}(b_{i_1},\dots, b_{i_n})=\tp_{\lL_{T}}(b_{j_1},\dots, b_{j_n})
        \]
        for all $i_1,\dots, i_n, j_1,\dots, j_n\in \sI$. In particular, combining the two implications, we see that the sequence $(b_i)_{i\in\sI}$ is non-collapsing, and hence, by \Cref{uncollapsed finite iff uncollapsed}, $T$ does not collapse $\sI$-indiscernibles.
\end{proof}

We revisit \Cref{example: Guingona equivalence fails in infinite language} to show an obstruction when the indexing structure is not Ramsey and the language is not countable. Of course, we will be primarily interested in Ramsey structures in the sequel, and thus this non-example should not be a cause for concern. Nonetheless, it illustrates the necessity of working in a countable language, for the coding/indexing structure.

\begin{example}\label{example-Guingona equivalence fails in infinite language 2}
    Let $\sN$ be the set $\mathbb{N}$ equipped with its full set-theoretic structure, and $\sN_S$ be its reduct to unary predicates for singleton. 
   We can see that $\sN_S$ doesn't collapse $\sN$-indiscernible (in the sense of \Cref{def:indiscernible collapsing qftp}) but we have clearly that $\sN_S \in \ncodingcla{\sN}$.     
\end{example}

We will observe now that an $n$-ary structure automatically collapses indiscernibles indexed by a higher arity structure. First, we recall the definitions:
\begin{definition}
Let $\sM$ be a relational $\lL$-structure. We say that $\sM$ is:
    \begin{enumerate}
        \item \emph{$n$-ary}, for $n\in\bN$, if it admits quantifier elimination in a relational language that consists only of relation symbols of arity at most $n$. 
        \item \emph{irreflexive} if for all relation symbols $P$ in $\lL$ and tuples $\bar{i}\in \sM^{\mathsf{arity}(P)}$, we have that $\sM\vDash\neg P(\bar{i})$ if an element occurs twice in $\bar{i}$.
    \end{enumerate}
\end{definition}

The following is clear: 
\begin{remark}
    Given a relational language $\lL'$ and an $\lL'$-structure $\sI$, we may always find a quantifier-free interdefinable language $\lL''$ such that, as an $\lL''$-structure, $\sI$ is irreflexive. Moreover it can be done in such a way that the notion of quantifier-free type remains the same: for all tuples $\bar{i},\bar{j}$ we have $\qftp^{\lL'}(\bar{i})=\qftp^{\lL'}(\bar{j})$ if, and only if, $\qftp^{\lL''}(\bar{i}) = \qftp^{\lL''}(\bar{j})$.
\end{remark}

\begin{proposition}\label{Prop:NaryCollapsingHigherArityStructure}
    Let $n$ be a positive integer and consider a relational language $\lL'$, and let $\lL_{\leq n}'$ be the sublanguage consisting of symbols of arity less or equal to $n$. 
    Let $\sI$ be an irreflexive $\lL'$-structure and  $\sJ$ the reduct to $\lL_{\leq n}'$. Then any $n$-ary structure $\sM$ collapses $\sI$-indiscernibles to $\sJ$-indiscernibles.
\end{proposition}

\begin{proof}
    Denote by $\lL$ a relational language with symbols of arity at most $n$, in which $\sM$ admits quantifier elimination. Let $(a_i)_{i\in I}$ be an $\sI$-indiscernible sequence of elements in $\sM$. We need to show that  $(a_i)_{i\in I}$ is $\sJ$-indiscernible. By quantifier elimination, it is enough to check that, for any relation $R\in \lL$, $(a_i)_{i\in I}$ is $\sJ$-indiscernible with respect to the formula $R(\bar{x})$. Consider two $\arity(R)$-tuples $\bar{i}$ and $\bar{j}$ such that $\qftp^{\lL_{\leq n}'}(i)=\qftp^{\lL_{\leq n}'}(j)$.
    By irreflexivity of $\sI$, and since $\bar{i}$ has fewer than $n$ elements, we have $\qftp^{\lL_{\leq n}'}(\bar{i})\vdash \qftp^{\lL'}(\bar{i})$, and similarly for $\bar{j}$. In particular, it follows that $\qftp^{\lL'}(\bar{i}) = \qftp^{\lL'}(\bar{j})$, and thus, by $\sI$-indiscernibility,  $\sM \models R(a_{\bar{i}}) \leftrightarrow R(a_{\bar{j}})$, and the result follows.
\end{proof}

\begin{example}\label{exa:BinaryStructureCollapseConvexilyOrderedCrelation}
    We illustrate \Cref{Prop:NaryCollapsingHigherArityStructure} with some useful examples:
    \begin{enumerate}
        \item\label{item: ordered C relation not coded by binary} Let $\sI=\mathrm{Flim}(\cla{OC})$ be the binary branching $C$-relation equipped with a convex order. Then any binary structure $\sM$ collapses $\sI$-indiscernibles to order-indiscernibles. Since $\cla{OC}$ is Ramsey, by \Cref{fact: noncoding iff collapsing Ramsey}, $\sM\in \ncodingcla{\cla{OC}}$.
        \item Let $\sI=\mathrm{Flim}(\cla{COD})$ be the binary branching $D$-relation equipped with a convex cyclic order. Then any ternary structure $\sM$ collapses $\sI$-indiscernible to cyclically ordered-indiscernibles.
    \end{enumerate}
\end{example}

\begin{theorem}\label{reduct of higher arity is not Ramsey}
    Let $\sJ$ be an $n$-ary $\aleph_0$-categorical ultrahomogeneous structure in a relational language $\lL$, and let $\sI$ be a non-$n$-ary reduct of $\sJ$ which is ultrahomogeneous in a finite relational language $\lL'$. Then $\age(\sI)$ is not a Ramsey class.
\end{theorem}
\begin{proof}
    As a reduct of $\sJ$, $\sI$ is also $\aleph_0$-categorical ultrahomogeneous. Assume $\sI$ is Ramsey. Let $\lL_{\leq n}'$ be the sublanguage of $\lL'$ consisting of symbols of arity less or equal to $n$. Since $\sI\reduces\sJ$, clearly $\sJ\in \codingcla{\sI}$, therefore, by \Cref{fact: noncoding iff collapsing Ramsey}, $\sJ$ does not collapse $\sI$-indiscernibles. On the other hand, by \Cref{Prop:NaryCollapsingHigherArityStructure}, $\sJ$ collapses $\sI$-indiscernibles to $\restriction{\sI}{\lL'_{\leq n}}$. By assumption, since $\sI$ is not $n$-ary, $\restriction{\sI}{\lL'_{\leq n}}\reducesneq \sI$, and thus we have a contradiction.
\end{proof}

For the remainder of this section, for $n\in\Nbb$ let $\lL_{\sH_n^o}=\{R,<\}$ be a language with a binary relation symbol $<$ and and an $n$-ary relation symbol $R$. By an $\lL_{\sH_n^o}$-structure we shall always mean a structure in which $<$ is a (total) linear order and $R$ is a uniform symmetric $n$-ary relation (i.e. $\lL_{\sH_n^o}$ structures are ordered $n$-uniform hypergraphs, in the sense of \Cref{sub:standard-structs}).

\begin{proposition}\label{reducts of ordered hypergraph are not n-ary}
    Let $n\geq 2$ and let $\sM$ be an $\lL_{\sH_n^o}$-structure and $\sN$ be a reduct of $\sM$. If $\sM\upharpoonright\Set{<}\reducesneq\sN\reducesneq\sM$ then $\sN$ is not $n$-ary.
\end{proposition}

\begin{proof}
    Since $\sN$ is a proper reduct of $\sM$ and $\aut(\sN)$ preserves $<$, we must have that $\aut(\sN)$ does not preserve $R$. Therefore, there are $a_1,\dots,a_n\in \sM$ and $g\in \aut(\sN)$ such that 
    \[
        \sM\models R(a_1,\dots, a_n)\land \neg R(g(a_1),\dots, g(a_n)).
    \]
    By symmetry, we may assume that $a_1<\dots<a_n$. Since $g$ preserves $<$ we have that $g(a_1)<\dots <g(a_n)$. Observe that for all $x_1,\dots, x_n;y_1,\dots, y_n\in \sM$ such that $x_1<\dots<x_n$ and $y_1<\dots y_n$, there is some $f\in \aut(\sN)$ such that $y_i = f(x_i)$ for all $1\leq i\leq n$. Indeed, letting $f_1, f_2\in \aut(\sM)$ such that $f_1(x_i)=a_i$ and $f_2(g(a_i))=y_i$ for all $1\leq i\leq n$, we have that $f=f_2\circ g\circ f_1$ is as stated. So the only non-trivial definable relation in $\sN$ of arity $\leq n$ is $<$. As $\sN\centernot\reducesboth\sM\upharpoonright\Set{<}$, this implies $\sN$ is not $n$-ary.
\end{proof}

\begin{theorem}\label{reducts of ordered hypergraph are not Ramsey}
       Let $n\in\Nbb$ and $\sM$ be a non-trivial reduct of $\cla{OH}_n$, other than DLO. Then $\sM$ is not Ramsey.
\end{theorem}

The proof of \Cref{reducts of ordered hypergraph are not Ramsey} will make use of the following \namecref{rc order mp}:

\begin{fact}[{\cite[Theorem B]{MP22}}]\label{rc order mp}
    Let $\cla{C}$ be a Ramsey class of $\lL^\prime$-structures and $\Ncal$ an $\lL^\prime$-structure such that $\age(\Ncal) = \cla{C}$. Then there is an $\Aut(\Ncal)$-invariant linear order on $\Ncal$ which is the union of quantifier-free types. 
    More explicitly, there is a (possibly infinite) Boolean combination of atomic and negated atomic $\lL^\prime$-formulas $\Phi(x,y):=\bigvee_{i\in I}\bigwedge_{j\in J_i} \vphi_{j_i}^{(-1)^{n_{j_i}}}(x,y)$, such that $\Phi$ is a linear order for every structure in $\cla{C}$.
\end{fact}

\begin{proof}[Proof of \Cref{reducts of ordered hypergraph are not Ramsey}]
    Assume towards contradiction $\sN$ is Ramsey. Then, by \Cref{rc order mp}, there is a linear order $\triangleleft$ on $\sN$ which is a union of quantifier-free types in $\sN$. As $\sN$ is a quantifier-free reduct of $\sM$, we see that $\triangleleft$ is a union of quantifier-free types in $\sM$.

    \begin{claim}\label{claim:order is order} 
        The order $\triangleleft$ is either $<$ or $>$.
    \end{claim}       
    \begin{claimproof}
        This is clear when $n\geq 3$, since $R$ is uniform and thus the only quantifier-free types in two variables in $\sM$ are $x<y$ and $x>y$. As $\triangleleft$ is antisymmetric, it cannot be equivalent to $x<y\lor x>y$.

        We now discuss the case where $n=2$.  In this case, the only quantifier-free types in $\sM$ are:
                \begin{enumerate}
                    \item $x<y \land R(x,y)$.
                    \item $x>y \land R(x,y)$.
                    \item $x<y \land \neg R(x,y)$.
                    \item $x>y \land \neg R(x,y)$.
                \end{enumerate}
                It is left as an exercise to the reader to verify that the only unions of the types above that define a linear order in the random graph are the following:
                \[
                    (x<y \land R(x,y))\lor (x<y \land \neg R(x,y))\text{ and } (x>y \land R(x,y))\lor (x>y \land \neg R(x,y)),
                \]
                again, it follows that $\triangleleft$ is either $<$ or $>$.
\end{claimproof}
     By \Cref{claim:order is order}, we have that $\sM\upharpoonright \Set{<}\reducesneq \sN\reducesneq  \sM$. Therefore, by \Cref{reducts of ordered hypergraph are not n-ary}, $\sN$ is not $n$-ary. By \cite[Theorem~2.7]{Tho96}, we know that $\sN$ is finitely homogeneous, therefore by \Cref{reduct of higher arity is not Ramsey}, $\sN$ is not Ramsey.
\end{proof}

\begin{remark}
    In fact, a slightly more careful analysis of the proof of \Cref{reducts of ordered hypergraph are not Ramsey} allows us to deduce a slightly more general result. We state it, and explain how to deduce it from the proof of \Cref{reducts of ordered hypergraph are not Ramsey} below.
\end{remark}

\begin{corollary}
     Let $n\geq 2$ and let $\sM$ be an $\aleph_0$-categorical ultrahomogeneous $\lL_{\sH_n^o}$-structure. If $\sM$ is not interdefinable with $(M,<_1,<_2)$, where $<_1$ and $<_2$ are two (independent) linear orders, then every proper quantifier-free reduct $\sN$ of $\sM$, which is not interdefinable with $\sM\upharpoonright\Set{<}$ is not Ramsey.
\end{corollary}
\begin{proof}
    In the case $n=2$, in the proof of \Cref{claim:order is order}, the only other unions of quantifier-free types that we would need to consider are:
    \[
        (x<y\land R(x,y))\lor (x>y\land \lnot R(x,y))
    \text{ and }
        (x>y\land R(x,y))\lor (x<y\land \lnot R(x,y)).
    \]
    But, if either of the formulas above defines a linear order on $\sM$, then $\sM$ is interdefinable with a structure $(M,<_1,<_2)$, where $<_1$ and $<_2$ are independent linear orders, which is a contradiction.
\end{proof}

By \cite[Proposition 2.23]{Bodirsky2015}, the class of finite cyclically ordered binary branching $D$-relations $\cla{COD}$ is not Ramsey, since no total order is definable in $\mathsf{FLim}(\cla{COD})$. We can see that, adding a generic order will not suffice to turn the class into a Ramsey class:
\begin{example}\label{Example:CODisnotRamsey}
    The free superposition $\cla{LO}*\cla{COD}$ is an amalgamation class, but is not a Ramsey class.  
    Indeed,  $\mathsf{FLim}(\cla{LO}*\cla{COD})$ is a reduct of $\mathsf{FLim}(\cla{LO}*\cla{OC})$. The latter is ternary, while the former is not. Therefore, by \Cref{reduct of higher arity is not Ramsey}, the former is not Ramsey.
\end{example}

\section{Transfer principles}\label{sec:transfer}
We study transfer principles for products, full and lexicographic, with respect to monadic NIP, distality, and $\ncodingcla{\cK}$. The common point of these properties is that they admit characterisations involving indiscernibility (and potentially some other notions). Our main tool we will thus be a description of indiscernible sequences, both for full and lexicographic product.

\subsection{Transfers in full products}

In this section, we will observe that the full product behaves nicely with respect to the notion collapsing indiscernible to a specified reduct, but not necessarily for the dividing line arising from coding structures. It is fairly easy to describe indiscernible sequences in a full product $\sM_1 \boxtimes \sM_2$ (see next Proposition \ref{PropositionCharacterisationIndiscernibleFullProduct}). This characterisation has some interesting consequences that we will collect here for the sake of completeness.

We first observe that the full product of structures $\sM_1\boxtimes \sM_2$ is almost never monadically NIP (or, for that matter, monadically anything). 

\begin{proposition}
    Assume $\sM_i$ is NIP for $i\in \{1,2\}$. We have $\sM_1 \boxtimes \sM_2$ is NIP of dp-rank $\dprk(\sM_1)+\dprk(\sM_2)$.
\end{proposition}
One can find a proof in \cite[Proposition 1.24]{Tou23}. We deduce immediately the following remark:

\begin{corollary}
    A full product $\sM_1\boxtimes \sM_2$ of infinite structures is never monadically NIP, as it has burden at least $2$ by the previous proposition. Indeed, $\sM_1\boxtimes \sM_2$ is monadically NIP precisely when one of the structures is finite and the other is monadically NIP. 
\end{corollary}

Another negative result is that $\ncodingcla{\cK}$ is not always stable under full products. We illustrate this in the following \namecref{example-full-product-bad}:

\begin{example}\label{example-full-product-bad}
     Recall that we denote by $\cla{OG}$ the class of finite ordered graphs, and by $\cla{OC}$ the class of finite convexly ordered binary branching $C$-relations. Let us consider the product $\cla{OG}\boxtimes \cla{OC}$ of the two classes. The ordered random graph $\sM$ is $\ncodingcla{\cla{OC}}$ because the random graph is a binary structure, and therefore collapses $\cla{OC}$-indiscernibles to order-indiscernibles (see \Cref{exa:BinaryStructureCollapseConvexilyOrderedCrelation}). In particular, $\sM\in \ncodingcla{\cla{OG}\boxtimes \cla{OC}}$. The generic $C$-relation $\sN$ induced by a binary tree is $\ncodingcla{\cla{OG}}$ as it is NIP, so again, $\sN\in \ncodingcla{\cla{OG}\boxtimes \cla{OC}}$. However, one can observe that $\sM\boxtimes\sN\notin\ncodingcla{\cla{OG}\boxtimes \cla{OC}}$, since $\sM\boxtimes\sN$ is the Fraïssé limit of $\cla{OG}\boxtimes \cla{OC}$.
     
    Intuitively, $\ncodingcla{\cla{OG}\boxtimes \cla{OC}}$ fails to be stable by full products because $\cla{OG}\boxtimes \cla{OC}$-indexed sequences can collapse to two orthogonal notions of indiscerniblity, namely $\sN$-indiscernibles and $\sM$-indiscernibles. 
\end{example}

\Cref{example-full-product-bad} above shows in particular that the hierarchy of $\codingcla{\cK}$-configuration is not linearly ordered. By the above, we have indeed that  $\cla{OG} \notin \codingcla{\cla{OC}}$ and $\cla{OC} \notin \codingcla{\cla{OG}}$. Using \Cref{order on classes}, we obtain:

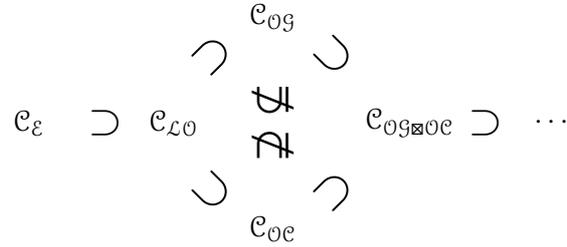
\begin{center}
    
\begin{tikzpicture}
    \draw (-1.5,0) node{$\codingcla{\cla{E}}$};
    \draw (-0.5,0) node{\scalebox{1.5}{$\supset$}};
    \draw (0.4,0) node{$\codingcla{\cla{LO}}$};
    \draw (0.9,0.9) node
        {\scalebox{1.7}{$\mathrel{\rotatebox[origin=c]{45}{$\supset$}}$}};
    \draw (1.7,1.4) node{$\codingcla{\cla{OG}}$};
    \draw (0.9,-0.9) node
    {\scalebox{1.7}{$\mathrel{\rotatebox[origin=c]{-45}{$\supset$}}$}};
    \draw (1.7,-1.4) node{$\codingcla{\cla{OC}}$};
    \draw (2.5,0.9) node
    {\scalebox{1.7}{$\mathrel{\rotatebox[origin=c]{-45}{$\supset$}}$}};
    \draw (2.5,-0.9) node
    {\scalebox{1.7}{$\mathrel{\rotatebox[origin=c]{45}{$\supset$}}$}};
    \draw (3.5,0) node{$\codingcla{\cla{OG}\boxtimes \cla{OC}}$};
    \draw (4.5,0) node{\scalebox{1.5}{$\supset$}};
    \draw (5.4,0) node{$\cdots$};

    \draw (1.7,0.3) node{\scalebox{1.7}{$\mathrel{\rotatebox[origin=c]{90}{$\nsubseteq$}}$}};
    \draw (1.7,-0.3) node{\scalebox{1.7}{$\mathrel{\rotatebox[origin=c]{90}{$\nsupseteq$}}$}};

\end{tikzpicture}

\end{center}

Above $\cla{LO}$ is the class of linear orders, and $\cla{E}$ is the class of finite (unstructured) sets. The key is to observe that:
\begin{itemize}
    \item $\codingcla{\cla{OG}}\not\subseteq\codingcla{\cla{OC}}$, which follows from \Cref{item: ordered C relation not coded by binary} of \Cref{exa:BinaryStructureCollapseConvexilyOrderedCrelation} together with \Cref{order on classes}.
    \item $\codingcla{\cla{OC}}\not\subseteq\codingcla{\cla{OG}}$ which follows from the fact that, as discussed above, the generic $C$-relation is NIP, so, in particular it belongs to $\ncodingcla{\cla{OG}}$, but of course it also belongs to $\codingcla{\cla{OC}}$.
\end{itemize}
This provides a negative answer to the question asked in  \cite[Section 7]{GP23} and \cite[Question 5.6]{GPS21} about linearity of the hierarchy.

Now, we will obtain positive results, using the following proposition:

\begin{proposition}\label{PropositionCharacterisationIndiscernibleFullProduct}
    A sequence $(\bar{c}_i)_{i\in I }$ in $\sM_1 \boxtimes \sM_2$ is $\sI$-indexed  indiscernible if and only if $(\pi_{M_1}(\bar{c}_i))_{i\in I}$ is an $\sI$-indiscernible sequence in $\sM_1$ and $(\pi_{M_2}(\bar{c}_i))_{i\in I}$ is an $\sI$-indiscernible sequence in $\sM_2$.
\end{proposition}

To clarify the notation, if $\Bar{c}=c_1,\dots,c_n$ is a tuple of $\sM_1 \boxtimes \sM_2$, the components can be in any of the three sorts, and we denote by $\pi_{M_1}(\Bar{c})$ the tuple of size at most $n$ consisting of the projections $\pi_{M_1}(c_i)$ if $c_i\in \sM_1 \boxtimes \sM_2$, and of the elements $c_i$ if $c_i\in \sM_1$. In particular, we removed the components $c_i\in \sM_2$. 

 \begin{proof}
     By relative quantifier elimination, a formula $\vphi(\bar{x})$ in $\lL_{\sM_1\boxtimes\sM_2}$ is equivalent modulo the theory of $\sM_1\boxtimes\sM_2$ to a union of formulas of the form:
     \[\vphi_1(\pi_{M_1}(\bar{x}))\wedge \vphi_2(\pi_{M_2}(\bar{x})).\]
     where $\vphi_1(\bar{x}_1)$ is an $\lL_{\sM_1}$-formula and $\vphi_2(\bar{x}_2)$ is an  $\lL_{\sM_2}$-formula. Let $(\bar{c}_i)_{i\in I}$ be a sequence in $\sM_1\boxtimes\sM_2$, where $I$ is an ordered set. To check whether $(\bar{c}_i)_{i\in I}$ is an $\sI$-indexed  indiscernible sequence, it is necessary and sufficient to check if $(\bar{c}_i)_{i \in I}$ is $\sI$-indiscernible with respect to the formulas $\vphi_i(\pi_{M_i}(x))$, $i\in \{1,2\}$. But this is exactly to check whether the sequence $(\pi_{M_1}(\bar{c}_i))_i$ is an $\sI$-indexed  indiscernible sequence in  $\sM_1$ and the sequence $(\pi_{M_2}(\bar{c}_i))_i$ is an $\sI$-indexed  indiscernible sequence in $\sM_2$.  

 \end{proof}
 
 We immediately deduce transfer principles for full products with respect to some properties which can easily be described with indiscernibles.  
 
\begin{corollary}\label{cor:collapse-indiscernibles-full-product}
        Consider $\sI$ and $\sJ$ two structures such that $\sJ$ is a reduct of $\sI$. The full product ${\sM_1 \boxtimes \sM_2}$ collapses $\sI$-indiscernibles to $\sJ$-indiscernibles if and only if $\sM_1$ and $\sM_2$ collapse $\sI$-indiscernibles to $\sJ$-indiscernibles.
\end{corollary}

The following \namecref{cor:nip-n-transfer-full} is immediate from \Cref{cor:collapse-indiscernibles-full-product} together with the characterisation of NIP$_n$ via an indiscernible collapse (\Cref{thm:nip-k-collapse}).

\begin{corollary}[NIP$_n$ transfer for full products]\label{cor:nip-n-transfer-full}
   For all $n\in\Nbb$, the full product $\sM_1 \boxtimes \sM_2$ is NIP$_n$ if, and only if, both $\sM_1$ and $\sM_2 $ are NIP$_n$.
\end{corollary}
    
\begin{proposition}
    The full product $\sM_1 \boxtimes \sM_2$ is distal if, and only if, $\sM_1 $ is distal and $\sM_2 $ is distal.
\end{proposition}

\begin{proof}
     If $\sM_1 \boxtimes \sM_2$ is distal, we see easily that $\sM_1 $ and $\sM_2 $ are distal. To prove the right-to-left implication, assume that $\sM_1$ and $\sM_2$ are both distal, and let $((a_i,b_i))_{i\in \mathbb{Q}}$ be a sequence in $\sM_1 \boxtimes \sM_2$, and $B$ a set of parameters in $\sM_1 \boxtimes \sM_2$ such that:
     \begin{enumerate}
         \item $((a_i,b_i))_{i\in \mathbb{Q}\setminus \{0\} }$ is indiscernible over $B$,
         \item $((a_i,b_i))_{i\in \mathbb{Q}}$ is indiscernible.
     \end{enumerate}
     By the characterisation in Proposition \ref{PropositionCharacterisationIndiscernibleFullProduct}, we have 
\begin{enumerate}
         \item $(a_i)_{i\in \mathbb{Q}\setminus \{0\} }$ is indiscernible over $\pi_{\sM_1}(B)$,
         \item $(a_i)_{i\in \mathbb{Q}}$ is indiscernible,
         \item $(b_i)_{i\in \mathbb{Q}\setminus \{0\} }$ is indiscernible over $\pi_{\sM_2}(B)$,
         \item $(b_i)_{i\in \mathbb{Q}}$ is indiscernible.
     \end{enumerate}
     By distality in $\sM_1$ and $\sM_2$, we have:
     \begin{enumerate}
         \item $(a_i)_{i\in}$ is indiscernible over $\pi_{\sM_1}(B)$,
        \item $(b_i)_{i\in \mathbb{Q} }$ is indiscernible over $\pi_{\sM_2}(B)$,
     \end{enumerate}
    Again by the characterisation, $((a_i,b_i))_{i\in \mathbb{Q}}$ is indiscernible over $B$. This shows that  $\sM_1 \boxtimes \sM_2$ is distal and concludes the proof.
\end{proof}
Similarly one can show more generally that $m$-distality also transfers to the full product. 
\begin{corollary}\label{cor:tranfer-n-distal-full}
    The full product $\sM_1 \boxtimes \sM_2$ is $n$-distal if, and only if, $\sM_1 $ is $n$-distal and $\sM_2 $ is $n$-distal.
\end{corollary}
We leave the proof to the reader. We will discuss again the notion of $m$-distality in \Cref{prop:m-distality-lex-transfer}, where we show the transfer for lexicographic products.

\begin{proposition}\label{cor:transfer-ind-triv-full}
    The full product $\sM_1 \boxtimes \sM_2$ has indiscernible-triviality if, and only if, $\sM_1 $ has indiscernible-triviality and $\sM_2 $ has indiscernible-triviality.
\end{proposition}
\begin{proof}
    Remark that any model of the theory of $\sM_1 \boxtimes \sM_2$ is a full product. We may thus assume that $\sM_1 \boxtimes \sM_2$ is a monster model. The left-to-right implication follows from the definition: an indiscernible sequence of $\sM_1$ (or of $\sM_2$) is in particular an indiscernible sequence of $\sM_1 \boxtimes \sM_2$.\\
    For the right-to-left implication, let $(c_i)$ be a sequence of tuples in $\sM_1\boxtimes \sM_2$ which is indiscernible over a set of parameters $B$ and over another set of parameters $C$. Since the projection are in the language, we may assume without loss that $B=B_1\times B_2 \subset \mathcal{M}_1 \times \mathcal{M}_2$ and $C=C_1\times C_2 \subset \mathcal{M}_1 \times \mathcal{M}_2$. By the characterisation, $(\pi_{M_1}(c_i)\frown B_1)$ and $(\pi_{M_1}(c_i)\frown C_1)$ are indiscernible in $\sM_1$. By indiscernible-triviality of $\sM_1$, $(\pi_{M_1}(c_i)\frown B_1\frown C_1)$ is indiscernible in $\sM_1$. Similarly, $(\pi_{M_2}(c_i)\frown B_2\frown C_2)$ is indiscernible in $\sM_2$ and by the characterisation, $(c_i \frown B \frown C)_{i\in I}$ is indiscernible in $\sM_1 \boxtimes \sM_2$. This shows indiscernible-triviality of $\sM_1 \boxtimes \sM_2$.
\end{proof}

\subsection{Transfers in lexicographic sums}
Through this subsection, we consider the lexicographic sum $\sS$ of an $\lL_\sM$-structure $\sM$ and a class of $\lL_\mathfrak{N}$-structures $\mathfrak{N}:= \{\sN_a\}_{a\in \sM}$. We will always assume the following:
\begin{align}
     \tag{$*$} \text{For every sentence $\vphi \in \lL_\mathfrak{N}$, the set $\{a\in \sM \ \vert \ \sN_a \models \vphi \}$ is $\emptyset$-definable in $\sM$.}
\end{align} 
This technical assumption ensures that the induced structure on the sort for $M$ in $\sS$ is exactly the $\lL_\sM$-structure $\sM$ (this is a consequence of Theorem \ref{thm:QuantifierEliminationLexSum}). Concretely, this means that we will often work with more colours on the structure $\sM$ than initially intended. Since most of the notions that we will consider in our application are preserved under taking reducts, this will be a harmless assumption.
  
We should, as for the full product, describe the generalised indiscernible sequences of $\sM[\Nfrak]$ in terms of generalised indiscernibility in $\sM$ and $\sN_i$. But unlike the full product, some assumptions will have to be made on the kinds of indexing structures we consider. We start with the following motivating example:

\begin{example}\label{ex:lex-product-bad}
    Let $\sM$ be an infinite set, in the language of pure equality, say $\sM=\omega$ and $\sN$ be an infinite set equipped with an equivalence relation with two infinite equivalence classes, say $\sN = R\sqcup B$, where $R$ and $B$ are disjoint copies of $\omega$. The lexicographic product $\sM[\sN]$ is precisely an infinite set equipped with two equivalence relations $E_1,E_2$ such that all $E_i$-classes for $i=1,2$ are infinite, $E_2$ refines $E_1$ and each $E_1$-class is refined by exactly two $E_2$-classes. To fix notation, we may write
    \[
    \sM[\sN] = \bigsqcup(\{i\}\times(\{r_i^j:j\in\omega\}\sqcup\{b_i^j:j\in\omega\})),
    \]
    and for $(w,x),(y,z)\in \sM[\sN]$ we have:
    \begin{itemize}
        \item $(w,x)E_1(y,z)$ if, and only if, $w=y$.
        \item $(w,x)E_2(y,z)$ if, and only if $w=y$ and either $(x=r_w^j)\land(z=r_w^k)$ or $(x=r_w^j)\land(z=r_w^k)$, for some $j,k\in\omega$.
    \end{itemize} 
    
    Let $\sJ = \sM[\sN]$ and $\sI$ be the reduct of $\sM[\sN]$ to $\lL_\sJ = \{E_1\}$, i.e. an infinite set equipped with an equivalence relation with infinitely many infinite equivalence classes. 
    
    Consider the sequence: 
    \[
    A = (((i,r_i^j):j\in\omega)\frown((i,b_i^j) :j\in\omega ):i\in\omega).
    \] 
    We may view this as a $\sJ$-indexed sequence in $\sM[\sN]$, where $\sJ=\sM[\sN]$ is taken with the same enumeration as $A$. By quantifier elimination in $\sM[\sN]$, this sequence is $\sJ$-indiscernible. On the other hand, it is easy to see that if we view this sequence as an $\sI$-indexed sequence, with the same enumeration, it is not $\sI$-indiscernible.

    Let us consider the ``factors'' of the sequence $A$, that is, the sequences:
    \begin{itemize}
        \item  $v(A)$, consisting of the first coordinates of the elements appearing in $A$, with the same enumeration, so:
        \[
            v(A) = ((i:j\in\omega)\frown(i:j\in\omega ):i\in\omega).
        \]
        \item $A_i$, for $i\in\omega$, consisting of the second coordinantes of elements in $A$ whose first coordinate is $i$, so:
        \[
            A_i = (r_i^j:j\in\omega)\frown(b_i^j :j\in\omega)
        \]
    \end{itemize}
    We can view both sequences above as $\sI$-indexed sequences in $\sM$ and $\sN$, respectively. In both cases we can make sure that the sequences are $\sI$-indiscernible (for $v(A)$ we make sure that whenever two elements are equal they are in the same $E_1$-class, and for $A_i$ we make sure that the $E_1$-classes refine the equivalence relation of $\sN$).
\end{example}

The upshot of the previous example is that it is possible to have a $\sJ$-indexed sequence in a lexicographic product whose components are $\sJ$-indiscernible, but which is not $\sJ$-indiscerinble. We investigate conditions on $\sI$ that allow us to formulate a relatively simple characterisation of $\sI$-indiscernibles in a lexicographic sum in terms of $\sI$-indiscernibles in the summands. 

First, to fix some terminology, given a structure $\sI$ We define the following graphs on $\{0,1\}\times \sI$: 
\[
    C:= \{((\epsilon,i),(1-\epsilon,i))  : i\in \sI, \epsilon\in \{0,1\}\},
\]
and
\[
    D:= \{((\epsilon,i),(\epsilon,j))  : i,j\in \sI, \epsilon \in\{0,1\}\}.
\]
These are shown below:
\begin{figure}[ht]
\begin{center}\begin{tikzpicture}
\draw (7.2,1) node {$\{1\}\times \sI$};
\draw (7.2,0) node {$\{0\}\times \sI$};
\foreach \i in {0,...,5}{
    \fill (8+\i,1) circle (2pt);
    \fill (8+\i,0) circle (2pt);
    }
\foreach \i in {0,...,5}{
    \draw (8+\i,1) -- (8+\i,0);
    }
    \draw (7.2,-0.66) node {$\phantom{a}$};
\end{tikzpicture}
\hspace{1cm}
\begin{tikzpicture}
        \draw (-0.7,1) node {$\{1\}\times \sI$};
        \draw (-0.7,0) node {$\{0\}\times \sI$};
        \draw (5.3,0) node {$\cdots$};
        \draw (5.3,1) node {$\cdots$};

        \foreach \i in {0,...,4}{
            \fill (\i,1) circle (2pt);
            \fill (\i,0) circle (2pt);
            \foreach \j in {\i,...,5}{
                \draw (\i,0) .. controls (\i/2+\j/2, \i/6-\j/6) .. (\j,0);
                \draw (\i,1) .. controls (\i/2+\j/2, \j/6-\i/6+1) .. (\j,1);
            }
        }
        \draw (6,0.5);
    \end{tikzpicture}
\end{center}
\caption{The graphs $C$ and $D$}
\end{figure}
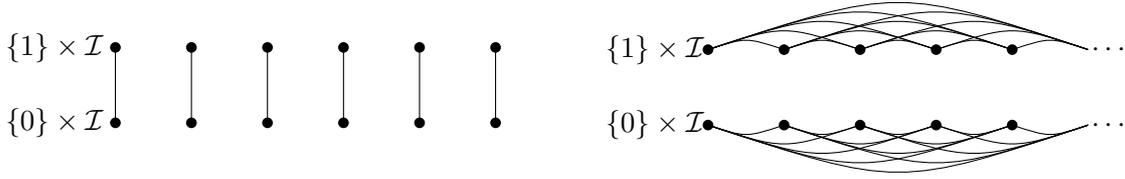

\begin{definition}\label{DefinitionReasonableIndexingStructures}
Let $\sI$ be a structure. 
The structure $\sI$ is called \emph{a reasonable indexing structure} or simply \emph{reasonable} if 
for every graph $R$ on $\{0,1\}\times I$ such that $R((\epsilon_0,i),(\epsilon_1,j))$ depends only on $\qftp_{\sI}(i,j)$ and on $\epsilon_0,\epsilon_1\in \{0,1\}$, one of the following holds:
\begin{enumerate}
    \item\label{1.} $R$ is connected (i.e. there is a path joining any two vertices of $\{0,1\}\times I$),
    \item\label{2.} $R$ is contained in $C$ (i.e. it is a subgraph of $(\{0,1\}\times I,C)$), 
    \item\label{3.} $R$ is contained in $D$ (i.e. it is a subgraph of $(\{0,1\}\times I,D)$).
\end{enumerate}
\end{definition}
    
It is clear that the structure $\sM$ from \Cref{ex:lex-product-bad} is not reasonable. Of course, that structure does not have the modelling property. The next example shows that the modelling property on $\sI$ does not imply that $\sI$ is a reasonable indexing structure. 

\begin{example}
    An infinite linear order $(\sM,<)$ equipped with an equivalence relation $E$ with infinitely many, infinite, convex equivalence classes is not a reasonable indexing structure. Indeed, consider the graph $(\{0,1\}\times M,R)$ where for $a,b\in M$ and $\epsilon_0,\epsilon_1\in \{0,1\}$, $R((\epsilon_0,a),(\epsilon_1,b))$ if, and only if, $aEb$ and $\epsilon_0\neq\epsilon_1$. 

    
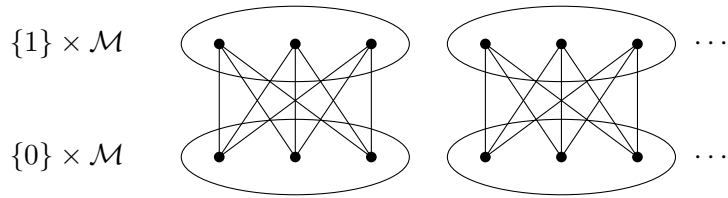
\begin{center}
    
\begin{tikzpicture}[rotate=90]

\begin{scope}
    \draw (0,0)  ellipse (0.5 and 1.5) ;
    \fill (0,1) circle (2pt);
    \fill (0,0) circle (2pt);
    \fill (0,-1) circle (2pt);
\end{scope}

\draw (0,1) -- (1.5,0);
\draw (0,0) -- (1.5,0);
\draw (0,-1) -- (1.5,0);
\draw (0,1) -- (1.5,1);
\draw (0,0) -- (1.5,1);
\draw (0,-1) -- (1.5,1);
\draw (0,1) -- (1.5,-1);
\draw (0,0) -- (1.5,-1);
\draw (0,-1) -- (1.5,-1);
\begin{scope}[xshift=1.5cm]
    \draw (0,0)  ellipse (0.5 and 1.5) ;
    \fill (0,1) circle (2pt);
    \fill (0,0) circle (2pt);
    \fill (0,-1) circle (2pt);
\end{scope}

\begin{scope}[yshift=-3.5cm]
\begin{scope}
    \draw (0,0)  ellipse (0.5 and 1.5) ;
    \fill (0,1) circle (2pt);
    \fill (0,0) circle (2pt);
    \fill (0,-1) circle (2pt);
\end{scope}

\draw (0,1) -- (1.5,0);
\draw (0,0) -- (1.5,0);
\draw (0,-1) -- (1.5,0);
\draw (0,1) -- (1.5,1);
\draw (0,0) -- (1.5,1);
\draw (0,-1) -- (1.5,1);
\draw (0,1) -- (1.5,-1);
\draw (0,0) -- (1.5,-1);
\draw (0,-1) -- (1.5,-1);
\begin{scope}[xshift=1.5cm]
    \draw (0,0)  ellipse (0.5 and 1.5) ;
    \fill (0,1) circle (2pt);
    \fill (0,0) circle (2pt);
    \fill (0,-1) circle (2pt);
\end{scope}
\end{scope}
\draw (0,3) node {$\{0\}\times \sM$};
\draw (1.5,3) node {$\{1\}\times \sM$};
\draw (0,-5.5) node {$\cdots$};
\draw (1.5,-5.5) node {$\cdots$};
\end{tikzpicture}

\end{center}

One sees by definition that the edge relation of this graph depends only on the quantifier-free type of the nodes, but this graph does not satisfy any of the conditions \ref{1.},\ref{2.} and \ref{3.} of \Cref{DefinitionReasonableIndexingStructures}. 
\end{example}

The definition of reasonability may seem rather ad hoc, at first sight. We will show that modulo some mild assumptions, $\sI$ is reasonable precisely when it is \emph{primitive}. To this end, let us recall some terminology from permutation group theory. We say that a permutation group $G$ acting on a set $\Omega$ acts \emph{primitively} if the only $G$-invariant equivalence relations on $\Omega$ are trivial (i.e.\ \emph{equality} and \emph{universality}). In the context of first-order structures, we have the following definition:

\begin{definition}[Primitivity]
    We say that an $\lL$-structure $\sM$ is \emph{primitive} if $\Aut(\sM)$ acts primitively on $\sM$. Explicitly, $\sM$ is primitive if, and only if, the only $\Aut(\sM)$-invariant equivalence relations on $\sM$ are trivial.
\end{definition}

\begin{proposition}\label{prop:reasonable-iff-primitive}
Let $\sI$ be an infinite ultrahomogeneous structure in a finite relational language. Then, the following are equivalent:
    \begin{enumerate}
        \item $\sI$ is reasonable,
        \item $\sI$ is primitive.
    \end{enumerate}
\end{proposition}

\begin{proof} \,
\begin{itemize}
    \item[$(1 \Rightarrow 2)$] Assume $\sI$ is reasonable, and let $E$ be a non-trivial $\Aut(\sI)$-invariant equivalence relation on $\sI$. Since $\sI$ is ultrahomogeneous, in a finite relational language, $E$ must be quantifier-free definable. Now, we may define the graph $R$ on $\{0,1\}\times M$ as follows: For $a,b\in \sI$ and $\epsilon_0,\epsilon_1 \in \{0,1\}$, 
    \[
        R((\epsilon_0,a),(\epsilon_1,b)) \text{ if and only if } aEb.
    \]
    
    By definition, the edges between the nodes $(\epsilon_0,a),(\epsilon_1,b)$ depend only on the quantifier-free type of $a,b$, and the $\epsilon_i$. Since $E$ is non-trivial there is a class with at least two elements; thus, the graph is not contained in the graphs $C,D$ from \Cref{DefinitionReasonableIndexingStructures}. It follows that the graph must be connected. But then, by transitivity of $E$, we must have that $aEb$ for all $a,b\in I$. Thus, $E$ is the universality equivalence relation, contradicting our initial assumption.
    
    \item[$(2 \Rightarrow  1)$] Assume $\sI$ is not reasonable. Then, by definition, we can find a graph $R$ on $\{0,1\}\times \sI$ such that $R((\epsilon_0,a),(\epsilon_1,b))$ depends only on $\epsilon_0,\epsilon_1$ and on the quantifier-free type of $a,b$, which is not connected, and not included in the graphs $C$ and $D$ from \Cref{DefinitionReasonableIndexingStructures}. We can define two equivalence relations $E_\epsilon$, for $\epsilon\in \{0,1\}$, by setting $aE_{\epsilon}b$ if, and only if, $(\epsilon, a)$ and $(\epsilon, b)$ are connected. By assumption on $R$, it follows that $E_\epsilon$ is $\Aut(\sI)$-invariant. If either $E_0$ or $E_1$ is not trivial, then we are done. Thus, we may assume, without loss of generality, that they both are trivial.
    
    We first show that $E_1$ and $E_2$ cannot both have a single equivalence class. Suppose not, then there are no edges of the form $R((0,a),(1,b))$, for otherwise the graph would be connected. But then, the graph would be included in $D$, which is absurd. 
    
    Assume now that $E_1$ has one equivalence class and $E_0$ has only singleton classes, it is easy to deduce from the conditions on $R$ that there must be exactly one element $(0,i)$ which connects with some elements $(1,j)$ on the top row. 
    
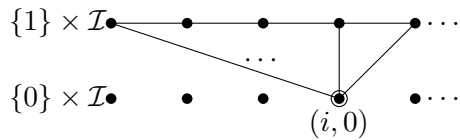
\begin{center}
    \begin{tikzpicture}
        \draw (-0.7,1) node {$\{1\}\times \sI$};
        \draw (-0.7,0) node {$\{0\}\times \sI$};
        \draw (4.4,0) node {$\dots$};
        \draw (0,1) -- (3,0);
        \draw (2,0.5) node {$\cdots$};
        \draw (3,1) -- (3,0);
        \draw (4,1) -- (3,0);
        \draw (4.4,1) node {$\dots$};
        \foreach \i in {0,...,4}{
            \fill (\i,1) circle (2pt);
            \fill (\i,0) circle (2pt);
        }
        \foreach \j in {0,...,3}{
            \draw (\j,1) -- (\j+1,1);  
        };
        \draw (3,0) circle (3pt);
        \draw (3,0) node[below] {$(i,0)$} ;
    \end{tikzpicture}   
    \end{center}
    Then, we have a non trivial  $\Aut(\sI)$-equivalence relation $E$ with given by 
    \[
        E(j',j) \  \text{ if and only if } j=j'=i \text{ or } (j\neq i \text{ and } j'\neq i).
    \]
    By symmetry, we may assume now that both $E_0$ and $E_1$ are trivial with singleton classes. In particular, this means that $R$ is the edge relation of a bipartite graph. 
    
    Let $a,b\in\sI$. We say there is a \emph{jump} from $a$ to $b$ if $R((0,a),(1,b))$.  Observe that, by assumption on $R$, for any  $\sigma\in\Aut(\sI)$ and any $a,b\in\sI$, if there is a jump from $a$ to $b$ then there will also be a jump from $\sigma(a)$ to $\sigma(b)$.

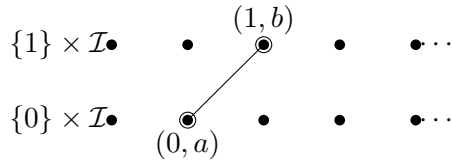
    \begin{center}
    \begin{tikzpicture}
        \draw (-0.7,1) node {$\{1\}\times \sI$};
        \draw (-0.7,0) node {$\{0\}\times \sI$};
        \draw (4.3,0) node {$\cdots$};
        \draw (1,0) -- (2,1);
        \draw (4.3,1) node {$\cdots$};
        \foreach \i in {0,...,4}{
            \fill (\i,1) circle (2pt);
            \fill (\i,0) circle (2pt); 
        };
        \draw (1,0) circle (3pt);
        \draw (1,0) node[below] {$(0,a)$};
        \draw (2,1) circle (3pt);
        \draw (2,1) node[above] {$(1,b)$};
    \end{tikzpicture}   
    \end{center}  
    
    We see that, given that $E_1$ is trivial with singleton classes, an element can jump to at most one element. Indeed, suppose not then there is some vertex with two different edges to $\{1\}\times \sI$, say $R((0,a),(1,b))$ and $R((0,a),(1,b'))$, but then $(1,b)$ and $(1,b')$ lie on the same path, which is a contradiction.

    
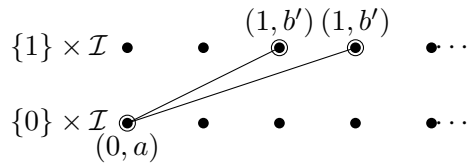
\begin{center}
    \begin{tikzpicture}
        \draw (-0.9,1) node {$\{1\}\times \sI$};
        \draw (-0.9,0) node {$\{0\}\times \sI$};
        \draw (4.3,0) node {$\cdots$};
        \draw (0,0) -- (2,1);
        \draw (0,0) -- (3,1);
        \draw (4.3,1) node {$\cdots$};
        \foreach \i in {0,...,4}{
            \fill (\i,1) circle (2pt);
            \fill (\i,0) circle (2pt); 
        };
        \draw (0,0) circle (3pt);
        \draw (2,1) node[above] {$(1,b')$};
        \draw (2,1) circle (3pt);
        \draw (3,1) node[above] {$(1,b')$};
        \draw (3,1) circle (3pt);
        \draw (0,0) node[below] {$(0,a)$};
    \end{tikzpicture}   
    \end{center}
    
    If an element does not jump to any other element, we can define a two-class $\Aut(\sI)$-invariant equivalence relation as follows: one class for jumping elements and one class for non-jumping elements. Similarly, if there is some $b\in\sI$ such that no element jumps to $b$, we can again define an analogous two-class $\Aut(\sI)$-invariant equivalent relation.
    
    Thus, without loss of generality, we may assume that we have a well-defined bijection $j: \sI \rightarrow \sI$ given by $a\mapsto b$, when $a$ jumps to $b$. We have the following $\Aut(\sI)$-invariant equivalence relation $\sim_j$ given by: 
    \[
        a\sim_j b  \text{ if, and only if, there is an integer } n\in \mathbb{Z} \text{ such that } j^n(a)=b. 
    \]

    Clearly, there must exist elements $a\neq b$ such that $a$ jumps to $b$, otherwise $R$ would be contained in the graph $C$ from \Cref{DefinitionReasonableIndexingStructures}. Thus, the equivalence relation $\sim_j$ has a non-singleton class. If $\sim_j$ is not universal, i.e.  has at least two classes, then we are done. Otherwise, there is exactly one class, and we may consider the equivalence relation $\sim_{j^2}$ given by the squared function $j^2$, and which has exactly 2 disjoint equivalence classes. We showed that $\sI$ is not primitive.
\end{itemize}
\end{proof}

\begin{lemma}\label{LemmaIIndiscernible}
    If $\sI$ is a transitive Fra\"iss\'e limit of a free amalgamation class $\cC$, possibly endowed with a generic order, then $\sI$ is a reasonable indexing structure. 
\end{lemma}

\begin{proof}
Let $G$ be a graph on $\{0,1\}\times \sI$ with edge relation $R((\epsilon_0,i),(\epsilon_1,j))$ which only depends on the quantifier-free type of $(i,j)$ and $\epsilon_0,\epsilon_1\in \{0,1\}$. Assume it is not included in $C$ nor $D$. We need to show that $G$ is connected. Let $i_0,i_0' \in \sI$, we will show that there is a path between $(0,i_0)$ and $(0,i_0')$. The other paths from $(\epsilon_0,i_0)$ and $(\epsilon_1,i_0')$ for $\epsilon_0,\epsilon_1\in \{0,1\}$ can be deduced similarly. Since $G$ is not included in $C$ nor $D$, there exists some $i\neq i_1\in I$ such that there is a path of length at most $2$ from $(0,i)$ $(1,i_1)$. 

We claim that such a path could have one of the following two forms:

\begin{center}
    \begin{tikzpicture}
        \foreach \i in {-1,0,1,2,4,5,6,7}{
            \fill (\i,1) circle (2pt);
            \fill (\i,0) circle (2pt);
        };
        \draw (1,0) node[below] {$(1,i_1)$};
        \draw (0,1) node[above] {$(0,i)$};        
        \draw (5,1) node[above] {$(0,i)$};
        \draw (6,0) node[below] {$(1,i_1)$};
        
        \draw (3,0.5) node {or};
        \draw (0,1)--(1,0);
        \draw (5,1)--(5,0)--(6,0);
        \draw (1,-1) node{(Case 1)};
        \draw (6,-1) node{(Case 2)};
        \end{tikzpicture}   
\end{center}  

Suppose we are not in the first case. Then $G$ contains no edges of the form $R((0,i),(1,j))$, for $i\neq j$. Since $G$ is not contained in either $C$ nor $D$ it must contain at least one edge of the form $R((0,i),(0,j))$ and at least one edge of the form $R((0,i),(1,i))$. But in this case, it must contain all edges of the form $R((0,i),(1,i))$, since by transitivity $\qftp_{\sI}(i)$ is constant, for all $i\in\sI$, and thus it contains the second path.
 
We assume that we are in the first case since the argument below will work identically for the second. By transitivity, we may assume that $i_0 = i$. By transitivity, again, there is some $i_1'$ with $\qftp_{\sI}(i_0,i_1)=\qftp_{\sI}(i_0',i_1')$. Without loss of generality, in what follows we assume that $i_1<i_0$, in $\sI$, and the case $i_0<i_1$ follows by an almost identical argument. 

By freely amalgamating $\{i_0,i_0',i_1\}$ and $\{i_0,i_0',i_1'\}$ over $\{i_0,i_0'\}$ we can find a structure $\{k_0,k_0^\prime,k_1,k_1^\prime\}$ in $\cla{C}$ such that $\qftp_\lL(i_0,i_0^\prime) = \qftp_\lL(k_0,k_0^\prime)$ and $\qftp_{\lL}(i_0,i_1)=\qftp_{\lL}(k_0, k_1) = \qftp_{\lL}(k_0',k_1')$ and there are no relations between $k_1$ and $k_1^\prime$. Now, since $\sI$ is ultrahomogeneous, there is some $\sigma\in\Aut(\sI\vert_\lL)$ such that $\sigma(k_0) = i_0$ and $\sigma(k_0^\prime) = i_0^\prime$. We let $i_2:=\sigma(k_1)$ and $i_2^\prime:=\sigma(k_1^\prime)$. It is now immediate, by construction, that $\qftp_{\sI\vert_\lL}(i_0,i_1)=\qftp_{\sI\vert_\lL}(i_0, i_2) = \qftp_{\sI\vert_\lL}(i_0',i_2')$, and that there are no relations between $i_2$ and $i_2^\prime$. Since the order in $\sI$ is generic, we can choose, in $\sI$, $i_2,i_2'$ so that $i_2<i_0$ and $i_2'<i_0'$. In particular, this means that $\qftp_{\sI}(i_0,i_1)=\qftp_{\sI}(i_0, i_2) = \qftp_{\sI}(i_0',i_2')$. So, by assumption on $R$, we have that $R((0,i_0),(1,i_2))$ and that $R((0,i_0'),(1,i_2'))$.

We now use free amalgamation again, this time amalgamating $\{i_0,i_2\}$ and $\{i_0',i_2'\}$ over $i_0$, which we identify with $i_0^\prime$, since they have the same quantifier-free type, and we obtain a structure $\{p_0,p_2,p_2^\prime\}\in\cla{C}$ where $\qftp_\lL(p_0,p_2) = \qftp_\lL(i_0,i_2)$, $\qftp_\lL(p_0,p_2^\prime) = \qftp_\lL(i_0,i_2^\prime)$ and there are no relations between $p_2$ and $p_2^\prime$. But observe that this means that $\qftp_{\lL}(i_2,i_2^\prime) = \qftp_{\lL}(p_2,p_2^\prime)$, so there is an automorphism $\sigma\in\Aut(\sI\vert_\lL)$ such that $\sigma(p_2) = i_2$ and $\sigma(p_2^\prime) = i_2^\prime$. Let $j_0 := \sigma(p_0)$. 
\begin{center}
    \begin{tikzpicture}
        \foreach \i in {0,1,2,4,6,7,8}{
            \fill (\i,1) circle (2pt);
            \fill (\i,0) circle (2pt);
        };
        \draw (1,1) node[above] {$(1,i_1)$};
        \draw (0,0) node[below] {$(0,i_0)$};        
        \draw (2,1) node[above] {$(1,i_2)$};
        
        \draw (0,0)--(1,1);
        \draw (0,0)--(2,1);

        \draw (7,1) node[above] {$(1,i_1')$};
        \draw (6,0) node[below] {$(0,i_0')$};        
        \draw (8,1) node[above] {$(1,i_2')$};
        
        \draw (6,0)--(7,1);
        \draw (6,0)--(8,1);

        \draw (4,0) node[below] {$(0,j)$};

        \draw[dotted] (2,1)--(4,0);
        \draw[dotted] (8,1)--(4,0);
        
        \draw (3,0.5) node {\dots};
        \draw (5,0.5) node {\dots};
        \draw (-1,0) node{$\{0\}\times\sI$};
        \draw (-1,1) node{$\{1\}\times\sI$};
        \end{tikzpicture}   
\end{center}  

Then we have that that $\qftp_{\sI\vert_\lL}(i_0,i_1)=\qftp_{\sI\vert_\lL}(j_0, i_2) = \qftp_{\sI\vert_\lL}(j_0,i_2')$, and, again, since the order is generic, we may assume that $j_0>i_2,i_2'$. It follows that $\qftp_{\sI}(i_0,i_1)=\qftp_{\sI}(j_0, i_2) = \qftp_{\sI}(j_0,i_2')$. Thus $G$ contains the path:
$R((0,i_0),(1,i_2))$, $R((1,i_2),(0,j_0))$, $R((0,j_0),(1,i_2'))$, and $R((1,i_2'),(0,i_0'))$ as required.
    \end{proof}

\begin{remark}
    In particular, combining \Cref{prop:reasonable-iff-primitive} and \Cref{LemmaIIndiscernible} we extend \cite[Lemma~2.6]{MT2011}, which states that transitive Fra\"iss\'e limits of free amalgamation classes are primitive, to transitive Fra\"iss\'e limits of free amalgamation classes with a generic order.
\end{remark}

Under this assumption of reasonable index sequence, we can describe the $\sI$-indiscernible sequences in the lexicographic sum $\sS$:
\begin{proposition}\label{PropositionCharacterisationIndiscernibleLexicographicProduct}
    Let $\sM$ be an $\lL_\sM$-structure,  $\mathfrak{N}=\{\sN_a\}_{a\in \sM}$ be a collection $\mathfrak{N}$ of $\lL_\mathfrak{N}$-structures indexed by $\sM$ and let $\sS = \sM[\mathfrak{N}]$. Let $\sI$ be a reasonable indexing structure, and let $(c_i)_{i\in \sI}$ be a sequence of tuples of size $\lambda$ (possibly infinite) in $\sS$. For $a\in \sM$, denote by $(c^a_i)_{i\in I}$ the sequence of subtuples $c^a_i:=(c_{\kappa_0,i},c_{\kappa_1,i}, \cdots)\subseteq c_i$ consisting of all the elements $c_{\kappa_j,i}$ of $c_i$ such that $v(c_{\kappa_j,i})$ is equal to $a$. Then:
    \begin{itemize}
        \item[(A)] For all $\kappa<\lambda$ and all distinct $i,j\in \sI$ if $v(c_{\kappa,i}) = v(c_{\kappa,j})$, then the sequence $(v(c_{\kappa,i}))_{i\in \sI}$ is constant. 
    \end{itemize}
   We denote by $A \subseteq \sM$ the subset of elements $a\in \sM$ such that for some $\kappa<\lambda$ and all $i\in \sI$, $v(c_{\kappa,i})=a$. In this notation we have that:
   \begin{itemize}
       \item[(B)] The following are equivalent:
    \begin{enumerate}
        \item The sequence $(c_i)_{i\in \sI}$ is $\sI$-indiscernible over $B\subseteq \sS$.
        \item The following conditions hold:
        \begin{enumerate}
             \item For all $a\in A$, the sequence $(c_i^a)_{i\in I}$ is $\sI$-indiscernible in $\sN_a$ over $B\cap \sN_a$.
            \item The sequence $(v(c_i))_{i\in I}$, where $v(c_i)$ denote the tuple $v(c_{\kappa,i})_{\kappa<\lambda}$, is $\sI$-indiscernible over $v(B)$ in $\sM$ and $\tp(c_i)$ is constant.
    \end{enumerate}
    \end{enumerate}
   \end{itemize} 
\end{proposition}
    
\begin{proof}\,
\begin{itemize}
    \item[$(A)$] Consider the graph $G$ on $\{0,1\}\times \sI$ whose edge relation is defined as follows:
    \[
        R((\epsilon_0,i)(\epsilon_1,j)) \text{ if, and only if, } v(c_i) = v(c_j). 
    \]
    Clearly, in $G$ whether $R((\epsilon_0,i),(\epsilon_1,j))$ holds depends only on $\qftp(i,j)$ by indiscernibility of $(c_i)_{i\in \sI}$. Thus, since $\sI$ is reasonable if this graph contains at least one edge of the form $R((0,i),(0,j))$, for some $i,j\in\sI$ then it must be connected and by construction, this means that $(v(c_{\kappa,i}))_{i\in \sI}$ is constant. 
   \item[$(B)$] It is easy to verify $1\implies 2$, and the proof is left to the reader.
   
   We show that $2\implies 1$. Let $(c_i)_{i\in I}$ be a sequence of tuples of size $\lambda$ satisfying $2$. Let $\vphi(\textbf{x}_{0},\dots, \textbf{x}_{n})$ be an $\lL_{\sM[\mathfrak{N}]}$-formula, where $\vert \textbf{x}_0 \vert = \cdots = \vert \textbf{x}_{n-1}\vert = \lambda$. We denote by $\textbf{x}$ the tuple $\textbf{x}_0,\dots, \textbf{x}_{n-1}$. By relative quantifier elimination (\Cref{thm:QuantifierEliminationLexSum}), we may Morleyise the structures $\sM$ and $\sN_{\rho}$ for $\rho\in S_1^{\Th(\sM)}$ and assume that $\vphi(\textbf{x}_1,\dots, \textbf{x}_n)$ is equivalent to a Boolean combination of 
    \begin{itemize}
        \item $v(\mathbf{x}_{\kappa,l})=v(\mathbf{x}_{\kappa',m})$ for some $m<l<n, \kappa,\kappa'<\lambda$, 
        \item $P_\bullet(\hat{\mathbf{x}})$, for $P\in \lL_\mathfrak{N}$ and subtuples $\hat{\mathbf{x}}$ of $\mathbf{x}$,
        \item $P(v(\hat{\mathbf{x}}))$, for $P\in \lL_\sM$ and subtuples $\hat{\mathbf{x}}$ of $\mathbf{x}$.
    \end{itemize}
    So it suffices to check indiscernibility with respect to each formula $\psi(\textbf{x}_1,\dots,\textbf{x}_n)$ as above. More precisely, for all $i_0,\dots,i_n\in\sI$ and all $i'_0,\dots,i_n'\in\sI$ such that $\qftp(\bar i) =\qftp(\bar{i'})$ we have to show that $\sS\models\psi(\bar c_{\bar i})\leftrightarrow\psi(\bar c_{\bar i'})$.
            
    We start with the following, easy, claim:    
    \begin{claim}
        Fix $\kappa_0,\dots,\kappa_n <\lambda$. Let $i_0,\dots,i_n\in I$ not all equal. If $v(c_{\kappa_0,i_0})=\cdots=v(c_{\kappa_n,i_n})=a$ for some $a\in \sM$, then for all $i\in I$,  $v(c_{\kappa_0,i})=\cdots=v(c_{\kappa_n,i})=a$.
    \end{claim}
    \begin{claimproof}
        For $n=1$, the result follows immediately from Definition \ref{DefinitionReasonableIndexingStructures} applied to the graph $G$ on $\{0,1\}\times I$, where $R((\epsilon_0,i),(\epsilon_1,j))$ if, and only if, $v(c_{\kappa_{\epsilon_0},i}) = v(c_{\kappa_{\epsilon_1},j})$. Clearly, in $G$ whether $R((\epsilon_0,i),(\epsilon_1,j))$ holds depends only on $\qftp(i,j)$ and $\epsilon_0,\epsilon_1$, by indiscernibility of $(c_i)_{i\in I}$. The statement for $n>1$ follows easily by induction.
    \end{claimproof}
    
    If $i_0,\dots,i_n$ are all equal to some $i$, then, since $(\tp(c_i))_{i\in\sI}$ is constant, for any formula $\psi(\mathbf{x}_0,\dots,\mathbf{x}_n)$, $\psi(c_{\kappa_0,i},\dots, c_{\kappa_n,i})$ holds if, and only if, $\psi(c_{\kappa_0,i'},\dots, c_{\kappa_n,i'})$ holds for all $i'\in I$. We may assume now that $i_0,\dots,i_n$ are not all equal.
    
    First, we check $\sI$-indiscernibility with respect to a formula $v(\mathbf{x}_{\kappa_0,i_0})=v(\mathbf{x}_{\kappa_1,i_1})$. We assumed that $i_0\neq i_1$ and by the previous claim, for any $\kappa_0,\kappa_1<\lambda$ we have that $v(c_{\kappa_0,i_0})=v(c_{\kappa_1,i_1})=a$ for some $a\in \sM$ if, and only if, for all $i'\in I$ ,  $v(c_{\kappa_0,i'})=v(c_{\kappa_1,i'})=a$. Thus, this case follows.
    
    We check now $\sI$-indiscernibility with respect to a formula $P_\bullet(c_{\kappa_0,i_0}, \dots ,c_{\kappa_n,i_n})$. If $v(c_{\kappa_0,i_0}),\dots,v(c_{\kappa_{n},i_{n}})$ are not all equal, then by the previous claim, for all $i_0',\dots,i_n'\in I$ with same quantifier-free type in $\sI$, $v(c_{\kappa_0,i_0'}), \dots, v(c_{\kappa_{n},i_{n}'})$ are not all equal. In particular, for any $i_0',\cdots,i_n'$ with the same quantifier-free type as $i_0,\cdots,i_n$ ,  $P_\bullet(c_{\kappa_0,i_0'}, \dots ,c_{\kappa_n,i_n'})$ doesn't hold. 
    
    Thus, we may assume that there is $a\in \sM$ such that $v(c_{\kappa_1,i})=\cdots=v(c_{\kappa_n,i})=a$ for all $i\in I$. This means that $c_{\kappa_1,i},\dots,c_{\kappa_n,i} \in c_i^a$. Then, since, by assumption, $(c_i^a)$ is $\sI$-indiscernible in $\sN_a$, we have that:
    \[
    \sN_a\models P(c_{\kappa_0,i_0}, \dots ,c_{\kappa_n,i_n})\text{ if, and only if, } \sN_a\models P(c_{\kappa_0,i'_0}, \dots ,c_{\kappa_n,i'_n}),
    \]
    for any $i_0',i_1',\dots,i_n'$ with same quantifier-free type in $\sI$. This shows that $(c_i)_i$ is $\sI$-indiscernible with respect to the formula $P_\bullet(c_{\kappa_0,i_0}, \dots ,c_{\kappa_n,i_n})$, as required.

    Finally, $\sI$-indiscernibility with respect to the formula $P(v(c_{\kappa_0,i_0}),\cdots, v(c_{\kappa_n,i_n}))$ where $P\in \lL_{\sM}$ is clear by the assumption that $(v(c_{i_1}))_{i\in \sI}$ is $\sI$-indiscernible in $\sM$. This concludes our proof.
\end{itemize}
\end{proof}

    In the literature, it is sometimes the case that indiscernibles indexed by non-primitive (i.e. non-reasonable) structures are considered (see for instance \cite[Theorem 5.8]{GHS17} for a nice characterisation of NTP$_2$ theories). We have obtained partial results in this direction and leave a full characterisation for future work.

From the \namecref{PropositionCharacterisationIndiscernibleLexicographicProduct}, we obtain the following easy corollary.

\begin{corollary}\label{CorollaryTransferCollaspingIndiscernibleForLexicographicProduct}
    Let $\sM$ be an $\lL_\sM$-structure,  $\mathfrak{N}=\{\sN_a\}_{a\in \sM}$ be a collection $\mathfrak{N}$ of $\lL_\mathfrak{N}$-structures indexed by $\sM$ and let $\sS = \sM[\mathfrak{N}]$. Consider $\sI$ and $\sJ$ two reasonable indexing structures, such that $\sJ$ is a reduct of $\sI$. Then, the following are equivalent:
    \begin{enumerate}
        \item $\sS$ collapses $\sI$-indiscernibles to $\sJ$-indiscernibles.
        \item For all $\rho\in S_1^{\Th(\sM)}$,  $\sN_\rho$ collapses $\sI$-indiscernibles to $\sJ$-indiscernibles and $\sM$ collapses $\sI$-indiscernibles to $\sJ$-indiscernibles.
    \end{enumerate}

\end{corollary}
\begin{proof}
As a monster model of $\Th(\sS)$ is still a lexicographic sum of models of $\Th(\sN_\rho)$ where $\rho\in S_1^{\Th(\sM)}$, we may assume that $\sS=\sM[\mathfrak{N}]$ is a monster model. Let $(c_i)_{i}$ be an $\sI$-indiscernible sequence in $\sS$. By Proposition \ref{PropositionCharacterisationIndiscernibleLexicographicProduct}, and using the same notation, we have that:
    \begin{enumerate}
        \item For all $a\in A$, the sequence $(c_i^a)_{i\in I}$ is $\sI$-indiscernible in $\sN_a$ over $B\cap \sN_a$. 
        \item The sequence $(v(c_i))_{i\in I}$, where $v(c_i)$ denote the tuple $v(c_{\kappa,i})_{\kappa<\lambda}$, is $\sI$-indiscernible over $v(B)$ in $\sM$ and $\tp(c_i)$ is constant.
    \end{enumerate}

By assumption, $\sM$ and $\sN_a$ for all $a\in \sM$ collapse $\sI$-indiscernible sequences to $\sJ$-indiscernible sequences. In particular, we have
    \begin{enumerate}
        \item For all $a\in A$, $(c_i^a)_{i\in I}$ is $\sJ$-indiscernible in $\sN_a$ over $B\cap \sN_a$. 
        \item The sequence $(v(c_i))_{i\in I}$ is $\sJ$-indiscernible over $v(B)$ in $\sM$ and $\tp(c_i)$ is constant.
    \end{enumerate}
    By Proposition \ref{PropositionCharacterisationIndiscernibleLexicographicProduct}, $(c_i)_{i}$ is a $\sJ$-indiscernible sequence in $\sS$. This concludes our proof.
\end{proof}

Observe that $\ncodingcla{\cK}$ is stable under reducts and thus \Cref{example-full-product-bad} also applies to lexicographic products. More precisely:
    
\begin{remark}
    Let $\sM$ be the ordered random graph and $\sN$ the convexly ordered $C$-relation, as in \Cref{example-full-product-bad}. Then $\sM[\sN]$ is $\codingcla{\cla{OG}\boxtimes\cla{OC}}\overset{\Cref{fact: equality of closedness}}{=}\codingcla{\cla{OG}[\cla{OC}]}$, but, as we have already seen $\sM$ is $\ncodingcla{\cla{OC}}\subset \ncodingcla{\cla{OG}\boxtimes\cla{OC}}$ and $\sN$ is $\ncodingcla{\cla{OG}}\subset \ncodingcla{\cla{OG}\boxtimes\cla{OC}}$.
\end{remark}

We can also apply Corollary \ref{CorollaryTransferCollaspingIndiscernibleForLexicographicProduct} to the specific case of $\cla{H}_{n+1}$-indiscernibles where $\cla{H}_{n+1}$ is the ordered random $(n+1)$-hypergraph since the latter is primitive. We obtain:

\begin{corollary}[$n$-NIP transfer for lexicographic sum]\label{cor:nip-n-transfer-lex}
    Let $\sM$ be an $\lL_\sM$-structure,  $\mathfrak{N}=\{\sN_a\}_{a\in \sM}$ be a collection $\mathfrak{N}$ of $\lL_\mathfrak{N}$-structures indexed by $\sM$ and let $\sS = \sM[\mathfrak{N}]$. The lexicographic product  $\sS$ is $n$-NIP if, and only if, for all $\rho\in S_1^{\Th(\sM)}$,  $\sN_\rho$ is $n$-NIP and $\sM$ is $n$-NIP.
\end{corollary}
    
Since for positive integer $k>k'$, $k'$-NIP structures are $k$-NIP, we have that a lexicographic product $\sM[\sN]$ of an $m$-NIP structure $\sM$ with an $n$-NIP structure $\sN$ is $\max(n,m)$-NIP.

\begin{proposition}\label{prop:TransferTrivialIndiscernibilityLexicographicSum}
    Assume that $\sM$ has indiscernible-triviality and that for all $\rho\in S_1^{\Th(\sM)}$, $\sN_\rho$ has indiscernible-triviality. Then so does $\sS$.
\end{proposition}
\begin{proof}
    We may assume that $\sS$ is $\vert \lL_{\sM} \vert^+$-saturated. Let $(c_i)_{i\in I}$ be an indiscernible sequence in $\sS$. Then $(c_i)_i$ is indiscernible over an element $a$ if and only if $(c_i\frown a)$ is indiscernible. 
    
    Let $d,d'\in \sS$ and assume that $(c_i\frown d)$ and $(c_i\frown d')$ are indiscernible. For $a\in \sM$ denote by $d^a$ the singleton $d$ if $\val(d)=a$ and the empty word $\emptyset$ if not. By Proposition \ref{PropositionCharacterisationIndiscernibleLexicographicProduct}, and using the same notation, we have:
    \begin{enumerate}
        \item for all $a\in A$, $(c_i^a\frown d^a)_{i\in I}$ $(c_i^a\frown {d'}^a )_{i\in I}$ are indiscernible in $\sN_a$ . 
         \item $(v(c_i)\frown v(d))_{i\in I}$ and $(v(c_i)\frown v(d'))_{i\in I}$ are indiscernible and $(\tp(c_i\frown d))_{i\in \sI}$ and $(\tp(c_i\frown d'))_{i\in\sI}$ are constant.
    \end{enumerate}
    By assumption, the $\sN_a$'s and $\sM$ have trivial indiscernibility. It follows that:
        \begin{enumerate}
             \item for all $a\in A$, $(c_i^a\frown d^a \frown {d'}^a)_{i\in I}$ is indiscernible in $\sN_a$.
             \item $(v(c_i)\frown v(d) \frown v(d'))_{i\in I}$ is indiscernible.
         \end{enumerate}
    It is also easy to check, by quantifier elimination and trivial-indiscernibility, that $\tp(c_i\frown d \frown d')$ is constant. By the same characterisation, we get that $(c_i\frown d \frown d')_i$ is indiscernible, in other words that $(c_i)_i$ is indiscernible over $d$ and $d'$. This concludes our proof.
\end{proof}

Combining the previous propositions, we will conclude a transfer principle for monadic NIP structures. We need one more ingredient to ensure that the lexicographic sum has dp-rank one:
\begin{fact}\cite[Theorem 2.16]{Tou21} \label{Fact:BurdenLexicographicProduct}
    We have:
        \[
        \bdn(\sS)=\sup\left(\bdn(\sM), \bdn(\sN_\rho), \rho \in S_1^{\Th(\sM)}\right).
        \]
    In particular, $\sS$ is inp-minimal if, and only if, $\sM$ is inp-minimal and $\sN_\rho$ is inp-minimal for every $\rho \in S_1^{\Th(\sM)}$.
    \end{fact}
    Then, we have by \Cref{thm:characterisations-of-monadic-NIP} and \Cref{prop:TransferTrivialIndiscernibilityLexicographicSum}:
    
\begin{corollary}[Transfer Principle For Monadic NIP Structures] \label{cor:TransferPrincipleMonadicNIP}
    Assume that $\sM$ is monadicaly NIP\footnote{Notice that the assumption (*) is not required here, as the expansion of $\sM$ by an unary predicates for the set $\set{ a\in \sM | \sN_a\models \vphi}$ will be automatically monadic NIP, by definition.} and that for all $\rho\in S_1^{\Th(\sM)}$, $\sN_\rho$ is monadicaly NIP. Then so is $\sM[\mathfrak{N}]$.
\end{corollary}
    
\begin{example}
    Given two meet-trees $T_1$ and $T_2$, the lexicographic product $T_1[T_2]$ is a meet-tree with a convex equivalence relation. By Theorem \ref{cor:TransferPrincipleMonadicNIP} and Example \ref{ExampleMonadicNIP}, it is monadically NIP. More generally, any meet-trees with finitely many convex equivalence relations $E_0, \dots, E_n$ such that $E_{s+1}$ refines $E_s$ for $s<n$ is monadically NIP.
\end{example}

\begin{proposition}[$m$-Distality transfer for lexicographic product]\label{prop:m-distality-lex-transfer}
    Assume that $\sM$ is $m$-distal and that for all $\rho\in S_1^{\Th(\sM)}$, $\sN_\rho$ are $m$-distal. Then $\sM[\mathfrak{N}]$ is $m$-distal.
\end{proposition}

\begin{proof}
    Without loss, we may assume that  $\sS=\sM[\mathfrak{N}]$ is a monster model.  Let $\sI= \sI_0+\cdots+\sI_{m+1}$ be an indiscernible sequence partitioned in $m+2$ subsequences. Let $(c_i)_{i\leq m}$ be a sequence of $m+1$ elements which does not insert indiscernibly into $\sI$. By definition, the sequence
    \[
    \sI' = \sI_0+ c_0  + \cdots+ c_{m}+\sI_{m+1}
    \]
    is not indiscernible in $\sS$. By Proposition \ref{PropositionCharacterisationIndiscernibleLexicographicProduct}, and using the same notation, we have the following cases to consider:
    \begin{itemize}
        \item[]\emph{Case 1}. For some $a\in A$, ${(\sI')}^a$ is not indiscernible in $\sN_a$. Then, since $\sN_a$ is $m$-distal, there is a $m$-subtuple of $(c_i)_{i\leq m}$ which does no insert in $\sI^a$ (as indiscernible sequence in $\sN_a$). 
   
        \item[]\emph{Case 2}. The sequence $v(\sI)$ is not indiscernible in $\sM$. Then, since $\sM$ is $m$-distal, there is an $m$-subtuple of $(v(c_i))_{i\leq m}$ which does not insert in $v(\sI)$.
  
        \item[] \emph{Case 3}. The sequence of types of elements in $\sI'$ is not constant. This means that there is an element $c_i$, $i\leq m$, which does not have the same type as the elements in $\sI$. 
    \end{itemize}
    In all cases, there is an $m$-subtuples of $(c_i)_{i\leq m}$ which do not insert in $\sI$. Indeed, in the first two cases this is by Proposition \ref{PropositionCharacterisationIndiscernibleLexicographicProduct}, and in the third case, we find in fact a single $c_i$ which does not insert in $\sI$.
    \end{proof}    

\section{Ultraproducts and Twin-width}\label{sec:twin-width}
We now apply our results to construct new algorithmically tame hereditary (i.e. closed under induced subgraphs) classes of graphs from given ones. The main notion of study here is that of \emph{bounded twin-width}. For a general introduction and precise definitions, we refer the reader to \cite{BKTW2021}, where the notion of twin-width was introduced. For basic background in parametrised complexity theory, we refer the reader to \cite{FG2006}.

In the next \namecref{prop:ultraproducts} we show that an ultraproduct of the class of lexicographic sums of two classes is isomorphic to a lexicographic sum of ultraproducts of these classes. More precisely:
    
\begin{proposition}\label{prop:ultraproducts}
    Fix a cardinal $\kappa$ and $\mathcal{U}$ an ultraproduct on $\kappa$. Let $\cC_1$ and $\cC_2$ be two classes of (not necessarily finite) structures. 
    Let $(G^i[H_g^i])_{i<k}$ be a sequence of lexicographic sum in $\cC_1[\cC_2]$. 
    Then we have:
    \[\prod_\mathcal{U} \left( G^i\left[H_g^i\right] \right) \simeq \left(\prod_\mathcal{U} G^i \right)  \left[ \mathcal{H}_g\right],\]
    where, for each $g = [(g_i)]_\mathcal{U}$, $\mathcal{H}_g= \prod_\mathcal{U} H_{g_i}^i$.
\end{proposition}
\begin{proof}
    First, notice that $\mathcal{H}_g$ doesn't depend on the choice of representative $(g_i)_i$ of $g$.
    Each $G^i[H_g^i]$ for $i\in \kappa$ carries a definable (through the projection $v$) equivalence relation $\sim$ where two elements $(g,h),(g',h')$ are equivalent if, and only if $g=g'$.  We remark by \Los that $\sim$ is also an equivalence relation on $\sM=\prod_\mathcal{U} \left ( G^i[H_g^i] \right)$, and we have more precisely that for $a=[((g_i,h_i))_i]_\mathcal{U}$ and $b=[((g_i',h_i'))_i]_\mathcal{U}$ in $\sM$, 
    \begin{align*}
        \sM\models a\sim b 
        & \Leftrightarrow \text{ for $\mathcal{U}$-many $i$, } G^i[H_g^i]\models (g_i,h_i) \sim (g_i',h_i')  \\
        & \Leftrightarrow (g_i)_i=(g_i')_i \mod \mathcal{U}.  
    \end{align*} 
    Similarly, if $P\in \lL_{\cC_1}$, we have for any tuples $a=[((g_i,h_i))_i]_\mathcal{U} \in \sM:$
    \begin{align*}
        \sM\models P(v(a)) 
        & \Leftrightarrow \text{ for $\mathcal{U}$-many $i$, } G^i[H_g^i]\models P(v(g_i,h_i) )  \\
        & \Leftrightarrow \text{ for $\mathcal{U}$-many $i$, } G^i\models P(g_i).  
    \end{align*} 
    It follows that $\sM/\sim$ is isomorphic to $ \prod_\mathcal{U} G^i$.
    It remains to show that each equivalence class is an ultraproduct of structures in $\cC_2$. Fix $a= [(g_i,h_i)]_\mathcal{U} \in \sM$.  By the above, an elements $b$ in $\sM$ is equivalent to $a$ if and only if it has a representative in 
    $\prod_{i\in \kappa} \{g_i\}\times H_{g_i}^i$.
    Thus the class $[a]_\sim$ of $a$ is an isomorphic copy of $\mathcal{H}_g= \prod_\mathcal{U} H_{g_i}^i$.
\end{proof}

\begin{definition}
    A class $\cC$ of structures is called \emph{monadically NIP} if any ultraproduct of structures in $\cC$ is monadically NIP. 
\end{definition}

\begin{corollary}\label{cor:monadic-NIP-sums-of-classes}
    Let $\cC_1$ and $\cC_2$ be two classes of structures that are monadically NIP. Then their lexicographic sum is monadically NIP.
\end{corollary}
\begin{proof}
    This follows immediately \Cref{cor:TransferPrincipleMonadicNIP}, the fact that monadic NIP is closed under reducts and \Cref{prop:ultraproducts}.
\end{proof}

It is conjectured (see, for instance, \cite[Conjecture 8.2]{GHPLR2020}) that for hereditary classes of graphs, under a mild assumption from descriptive complexity theory (namely that $\mathsf{FPT}\neq\mathsf{AW}[\star]$, which, in particular, implies that first-order model checking for the class of all graphs is not tractable), the algorithmic tameness condition of having \emph{fixed-parameter tractable model checking} coincides with the class being monadically NIP. There is strong evidence for this conjecture. In particular, it is true for \emph{monotone} (i.e. closed under not necessarily induced substructures) classes of relational structures \cite{BDIP2023}, and hereditary classes of ordered graphs \cite{BGOdMSTT2022}. More precisely, the latter result is the following:

\begin{fact}[{\cite[Theorems 1 and 3]{BGOdMSTT2022}}]\label{fact:twin-width}
    The following conditions are equivalent for a hereditary class $\cC$ of finite, ordered binary structures:
    \begin{enumerate}
        \item $\cC$ is monadically NIP.
        \item $\cC$ has bounded twin-width.
        \item (Assuming $\mathsf{FPT}\neq\mathsf{AW}[\star]$) Model-checking first-order logic is fixed-parameter tractable on $\cC$.
    \end{enumerate}
\end{fact}  

Observe that any graph $G$ can be expanded by a total order resulting in an \emph{ordered graph} of the same twin-width (see \cite{BKTW2021} or \cite{ST21} for an explicit argument), but it is known that finding these orders is hard to do efficiently. In any case, given a class $\cC$ of graphs with bounded twin-width we can assume that $\cC$ consists of ordered graphs and still has bounded twin-width. Moreover, given any graph $G$, the twin-width of any (induced) subgraph of $G$ is at most that of $G$ (see \cite{BKTW2021}). In particular, given a class of graphs $\cC$ with bounded twin-width we may assume that $\cC$ is hereditary.

One of the results in \cite{BERTW2022} is that the lexicographic product of two graphs of bounded twin-width has bounded twin-width. This was expanded for various other product notions (always involving two graphs) in \cite{PS2023}. Here we generalise the first result to lexicographic sums of graphs with bounded twin-width. More precisely we obtain the following:

\begin{corollary}\label{cor:sum-has-bdd-twin-width}
    Let $\cC_1$ and $\cC_2$ be two classes of finite graphs with bounded twin-width. Then, the reduct of $\cC_1[\cC_2]$ to the language of graphs has bounded twin-width.
\end{corollary}
\begin{proof}
    Without loss of generality, we may assume that $\cC_1$ and $\cC_2$ consist of ordered graphs and are hereditary. By \Cref{fact:twin-width}, it follows that $\cC_1$ and $\cC_2$ are monadically NIP. Now, by \Cref{cor:monadic-NIP-sums-of-classes}, it follows that the class $\cC_1[\cC_2]$, in which there is a natural total order coming from the orders from $\cC_1$ and $\cC_2$ is monadically NIP, and thus, its reduct to the language of ordered graphs is a monadically NIP hereditary class of ordered graphs. Thus, by \Cref{fact:twin-width} it has bounded twin-width, as claimed.
\end{proof}

\section{Open Questions}\label{sec:questions}
We conclude this paper by collecting some open questions. Recall that a structure $\sM$ \emph{admits a distal expansion} if it is a reduct of a distal structure. With respect to some combinatorial and model-theoretic aspects, it is more relevant to ask if a structure has a distal expansion, rather than if the structure itself is distal. Clearly, given two structures each of which has a distal expansion, their full product will also have a distal expansion. However, the other direction remains unclear. We have the following observation:

\begin{observation}
    An expansion of $\sM_1 \boxtimes \sM_2$ is not necessarily a reduct of a full product of expansions of $\sM_1$ and $\sM_2$. Assume $|\sM_1| = | \sM_2|$ and let $b:\sM_1  \rightarrow \sM_2 $ be a bijection. Then $(\sM_1 \boxtimes \sM_2,b)$ cannot be such a reduct, as its theory implies $ |\sM_1| = | \sM_2 |$.
\end{observation}

This motivates the following question:
\begin{question}
    Assume that a full product $\sM_1\boxtimes \sM_2$ admits a distal expansion. Do both $\sM_1$ and $\sM_2$ admit a distal expansion?    
\end{question}

In \Cref{Subsubsec:Monadic NIP}, we stated the fact that indiscernible triviality and dp-minimality characterise monadically NIP structures. Dp-minimality cannot be characterised in terms of forbidden coding (since it is not closed under bi-interpretation), but one can ask if there is a natural non-coding class strictly contained in NIP, that contains all monadically NIP structures. A natural candidate for such a class would be that of NIP structures with indiscernible triviality. This leads to the following question:

\begin{question} 
    Is it possible to characterise indiscernible triviality in terms of collapsing indiscernibles or in terms of forbidden coding?
\end{question}

Finally, the following question naturally relates to the classification of ultrahomogeneous NIP$_n$ structures:
\begin{question}
    Given an integer $n$, is there an ultrahomogeneous countable structure with the Ramsey property, that is NIP$_n$, but not $n$-ary?
\end{question}

 Some positive examples will give, by \Cref{Prop:NaryCollapsingHigherArityStructure}, some other instances of ultrahomogeneous countable structures $\mathcal{M}$ and $\mathcal{N}$ such that $\ncodingcla{\sN}$ and $\ncodingcla{\sM}$ are not included into each other.
 
 More generally, one can ask about the structure of the $\ncodingcla{\cla{K}}$-hierarchy - say, when $\cK$ ranges over the class of all ultrahomogeneous countable structures. By \Cref{fact: equality of closedness} and \Cref{example-full-product-bad}, it is a rooted meet-tree, with a countable ascending chain and an antichain of size at least two. One can then reformulate \cite[Question 4.7]{GPS21}:
\begin{question}
    What is the width of the  $\ncodingcla{\cla{K}}$-hierarchy?
\end{question}

\bibliographystyle{alpha}
\bibliography{bibtex}


\parbox{\linewidth}{\textsc{Nadav Meir}\\
\textsc{Department of Mathematics, Ben-Gurion University of the Negev, \\
84105 Be'er-Sheva, Israel} \\
and \\
\textsc{Instytut Matematyczny, Uniwersytet Wroc{\l}awski,  \\
pl. Grunwaldzki 2/4, 
50-384 Wroc{\l}aw, Poland \\}  
\url{math@nadav.me}} \\ \, \\

\parbox{\linewidth}{\textsc{Aris Papadopoulos}\\
\textsc{School of Mathematics, University of Leeds, Leeds LS2 9JT, United Kingdom} \\
\url{mmadp@leeds.ac.uk}}\\ \, \\

\parbox{\linewidth}{\textsc{Pierre Touchard}\\
\textsc{KU Leuven, Department of Mathematics, B-3001 Leuven, Belgium} \\
\url{pierre.touchard@kuleuven.be}}

\end{document}